\documentclass[12pt]{amsart}
\usepackage{graphicx}
\usepackage{amsmath}
\usepackage{amssymb}
\usepackage{setspace}
\newtheorem{thm}{Theorem}[section]

\newtheorem{lem}[thm]{Lemma}
\newtheorem{prop}[thm]{Proposition}
\theoremstyle{definition}
\newtheorem{defn}[thm]{Definition}
\theoremstyle{remark}
\newtheorem{rem}[thm]{Remark}
\theoremstyle{conclusion}

\theoremstyle{question}
\newtheorem{qu}[thm]{Question}
\numberwithin{equation}{section}

\begin{document}
\title[Bilinear and multi-parameter Hilbert Transforms]{$L^{p}$ estimates for bilinear and multi-parameter Hilbert transforms}

\author{Wei Dai and Guozhen Lu}

\address{School of Mathematical Sciences, Beijing Normal University, Beijing 100875, P. R. China}
\email{daiwei@bnu.edu.cn}

\address{Department of Mathematics, Wayne State University, Detroit, MI 48202, U. S. A.}
\email{gzlu@math.wayne.edu}

\thanks{Research of this work was partly supported by grants from the NNSF of China, the China Postdoctoral Science Foundation and a US NSF grant. \\Corresponding Author: Guozhen Lu at gzlu@math.wayne.edu}

\begin{abstract}
C. Muscalu, J. Pipher, T. Tao and C. Thiele proved in \cite{MPTT1} that the standard bilinear and bi-parameter Hilbert transform does not satisfy any $L^{p}$ estimates. They also raised a question asking if a bilinear and bi-parameter multiplier operator defined by
$$
 T_{m}(f_{1},f_{2})(x):=\int_{\mathbb{R}^{4}}m(\xi,\eta)\hat{f_{1}}(\xi_{1},\eta_{1})\hat{f_{2}}(\xi_{2},\eta_{2})e^{2\pi ix\cdot((\xi_{1},\eta_{1})+(\xi_{2},\eta_{2}))}d\xi d\eta
$$
satisfies any $L^p$ estimates, where the symbol $m$ satisfies
$$
  |\partial_{\xi}^{\alpha}\partial_{\eta}^{\beta}m(\xi,\eta)|\lesssim\frac{1}{dist(\xi,\Gamma_{1})^{|\alpha|}}\cdot\frac{1}{dist(\eta,\Gamma_{2})^{|\beta|}}
$$
for sufficiently many multi-indices $\alpha=(\alpha_{1},\alpha_{2})$ and $\beta=(\beta_{1},\beta_{2})$, $\Gamma_{i}$ ($i=1,2$) are subspaces in $\mathbb{R}^{2}$ and $dim \, \Gamma_{1}=0, \, dim \, \Gamma_{2}=1$. P. Silva answered partially this question in \cite{S} and proved that $T_{m}$ maps $L^{p_1}\times L^{p_2}\rightarrow L^{p}$ boundedly when $\frac{1}{p_1}+\frac{1}{p_2}=\frac{1}{p}$ with $p_1, p_2>1$, $\frac{1}{p_1}+\frac{2}{p_2}<2$ and $\frac{1}{p_2}+\frac{2}{p_1}<2$. One observes that the admissible range here for these tuples $(p_1,p_2,p)$ is a proper subset contained in the admissible range of BHT. 

In this paper, we establish the same $L^{p}$ estimates as BHT in the full range for the bilinear and multi-parameter Hilbert transforms with arbitrary symbols satisfying appropriate decay assumptions (Theorem 1.3). Moreover, we also establish the same $L^p$ estimates as BHT for certain modified bilinear and bi-parameter Hilbert transforms with $dim \, \Gamma_{1}=dim \, \Gamma_{2}=1$   but with a slightly better decay than that for the bilinear and bi-parameter Hilbert transform  (Theorem 1.4).
\end{abstract}
\maketitle {\small {\bf Keywords:} Bilinear and multi-parameter Hilbert transforms; $L^{p}$ estimates; polydiscs.\\

}

\section{Introduction}

The bilinear Hilbert transform is defined by
\begin{equation}\label{1.1}
  BHT(f_{1},f_{2})(x):=p. \, v. \, \int_{\mathbb{R}}f_{1}(x-t)f_{2}(x+t)\frac{dt}{t},
\end{equation}
or equivalently, it can also be written as the bilinear multiplier operator
\begin{equation}\label{1.2}
  BHT: \,\,\, (f_{1},f_{2})\mapsto\int_{\xi<\eta}\hat{f_{1}}(\xi)\hat{f_{2}}(\eta)e^{2\pi ix(\xi+\eta)}d\xi d\eta,
\end{equation}
where $f_{1}$ and $f_{2}$ are Schwartz functions on $\mathbb{R}$. In \cite{LT1,LT2}, M. Lacey and C. Thiele proved the following $L^{p}$ estimates for bilinear Hilbert transform.
\begin{thm}\label{BHT}(\cite{LT1,LT2})
The bilinear operator BHT maps $L^{p}(\mathbb{R})\times L^{q}(\mathbb{R})$ into $L^{r}(\mathbb{R})$ boundedly for any $1<p, \, q\leq\infty$ with $\frac{1}{p}+\frac{1}{q}=\frac{1}{r}$ and $\frac{2}{3}<r<\infty$.
\end{thm}

There are lots of works related to bilinear operators of BHT type. J. Gilbert and A. Nahmod \cite{GN} and F. Bernicot \cite{Be} proved that the same $L^{p}$ estimates as BHT are valid for bilinear operators with more general symbols. Uniform estimates were obtained by C. Thiele \cite{Thiele}, L. Grafakos and X. Li \cite{GL} and X. Li \cite{Li}. A maximal variant of Theorem \ref{BHT} was proved by M. Lacey \cite{Lacey}. In C. Muscalu, C. Thiele and T. Tao \cite{MTT2} and J. Jung \cite{Jung}, the authors investigated various trilinear variants of the bilinear Hilbert transform. For more related results involving estimates for multi-linear singular multiplier operators, we refer to the works, e.g., \cite{CJ, CM1, CM2, FS, GT, GT1, Jo, KS, MS, MTT1, Thiele1} and the references therein.

In multi-parameter cases, there are also large amounts of literature devoted to studying the estimates of multi-parameter and multi-linear operators (see \cite{CL, DL, DT,  HL, Kesler, Luthy, MS, MPTT1, MPTT2,S} and the references therein). In the bilinear and bi-parameter cases, let $\Gamma_{i}$ ($i=1,2$) be subspaces in $\mathbb{R}^{2}$, we consider operators $T_{m}$ defined by
\begin{equation}\label{1.3}
  T_{m}(f_{1},f_{2})(x):=\int_{\mathbb{R}^{4}}m(\xi,\eta)\hat{f_{1}}(\xi_{1},\eta_{1})\hat{f_{2}}(\xi_{2},\eta_{2})e^{2\pi ix\cdot((\xi_{1},\eta_{1})+(\xi_{2},\eta_{2}))}d\xi d\eta,
\end{equation}
where the symbol $m$ satisfies\footnote{Throughout this paper, $A\lesssim B$ means that there exists a universal constant $C>0$ such that $A\leq CB$. If necessary, we use explicitly $A\lesssim_{\star,\cdots,\star}B$ to indicate that there exists a positive constant $C_{\star,\cdots,\star}$ depending only on the quantities appearing in the subscript continuously such that $A\leq C_{\star,\cdots,\star}B$.}
\begin{equation}\label{1.4}
  |\partial_{\xi}^{\alpha}\partial_{\eta}^{\beta}m(\xi,\eta)|\lesssim\frac{1}{dist(\xi,\Gamma_{1})^{|\alpha|}}\cdot\frac{1}{dist(\eta,\Gamma_{2})^{|\beta|}}
\end{equation}
for sufficiently many multi-indices $\alpha=(\alpha_{1},\alpha_{2})$ and $\beta=(\beta_{1},\beta_{2})$. If $dim \, \Gamma_{1}=dim \, \Gamma_{2}=0$, C. Muscalu, J. Pipher, T. Tao and C. Thiele proved in \cite{MPTT1,MPTT2} that H\"{o}lder type $L^{p}$ estimates are available for $T_{m}$; however, if $dim \, \Gamma_{1}=dim \, \Gamma_{2}=1$, let $T_{m}$ be the double bilinear Hilbert transform on polydisks $BHT\otimes BHT$ defined by
\begin{equation}\label{1.5}
  BHT\otimes BHT(f_{1},f_{2})(x,y):=p. \, v. \, \int_{\mathbb{R}^{2}}f_{1}(x-s,y-t)f_{2}(x+s,y+t)\frac{ds}{s}\frac{dt}{t},
\end{equation}
they also proved in \cite{MPTT1} that the operator $BHT\otimes BHT$ does not satisfy any $L^{p}$ estimates of H\"{o}lder type by constructing a counterexample. In fact, consider bounded functions $f_{1}(x,y)=f_{2}(x,y)=e^{ixy}$, one has formally
\begin{equation*}
  BHT\otimes BHT(f_{1},f_{2})(x,y)=(f_{1}\cdot f_{2})(x,y)\int_{\mathbb{R}^{2}}\frac{e^{2ist}}{st}dsdt=i\pi(f_{1}\cdot f_{2})(x,y)\int_{\mathbb{R}}\frac{sgn(s)}{s}ds,
\end{equation*}
then localize functions $f_{1}$, $f_{2}$ and let $f^{N}_{1}(x,y)=f^{N}_{2}(x,y)=e^{ixy}\chi_{[-N,N]}(x)\chi_{[-N,N]}(y)$, one can verify the pointwise estimate
\begin{equation}\label{1.6}
  |BHT\otimes BHT(f^{N}_{1},f^{N}_{2})(x,y)|\geq|\int_{-\frac{N}{10}}^{\frac{N}{10}}\int_{-\frac{N}{10}}^{\frac{N}{10}}\frac{e^{2ist}}{st}dsdt|+O(1)\geq C\log N+O(1)
\end{equation}
for every $x,y\in[-\frac{N}{100},\frac{N}{100}]$ and sufficiently large $N\in\mathbb{Z}^{+}$, which indicates that no H\"{o}lder type $L^{p}$ estimates are available for the bilinear operator $BHT\otimes BHT$. When $dim \, \Gamma_{1}=0$ and $dim \, \Gamma_{2}=1$, C. Muscalu, J. Pipher, T. Tao and C. Thiele raised the following problem in Question 8.2 in \cite{MPTT1}.
\begin{qu}\label{open problem}(\cite{MPTT1})
Let $dim \, \Gamma_{1}=0$ and $dim \, \Gamma_{2}=1$ with $\Gamma_{2}$ non-degenerate in the sense of \cite{MTT1}. If $m$ is a multiplier satisfying \eqref{1.4}, does the corresponding operator $T_{m}$ defined by \eqref{1.3} satisfy any $L^{p}$ estimates?
\end{qu}
In \cite{S}, P. Silva answered this question partially and proved that $T_{m}$ defined by \eqref{1.3}, \eqref{1.4} with $dim \, \Gamma_{1}=0$ and $dim \, \Gamma_{2}=1$ maps $L^{p}\times L^{q}\rightarrow L^{r}$ boundedly when $\frac{1}{p}+\frac{1}{q}=\frac{1}{r}$ with $p,q>1$, $\frac{1}{p}+\frac{2}{q}<2$ and $\frac{1}{q}+\frac{2}{p}<2$. One should observe that the admissible range for these tuples $(p,q,r)$ is a proper subset of the region $p,q>1$ and $\frac{3}{4}<r<\infty$, which is also properly contained in the admissible range of BHT (see Theorem \ref{BHT}).

Naturally, we may wonder whether the bi-parameter bilinear operator $T_{m}$ given by \eqref{1.3}, \eqref{1.4} (with appropriate decay assumptions on the symbol $m$ and singularity sets $\Gamma_{1}$, $\Gamma_{2}$ satisfying $dim \, \Gamma_{1}=0  \, \text{or} \,1$, $dim \, \Gamma_{2}=1$) satisfies the same $L^{p}$ estimates as BHT.

To study this problem, we must find the implicit decay assumptions on symbol $m$ to preclude the existence of those kinds of counterexamples constructed in the above \eqref{1.6} for $BHT\otimes BHT$. To this end, let us consider first the bilinear operator $T_{m}\otimes BHT$ of tensor product type, which is defined by
\begin{equation}\label{1.7}
  T_{m}\otimes BHT(f_{1},f_{2})(x,y):=p. \, v. \, \int_{\mathbb{R}^{2}}f_{1}(x-s,y-t)f_{2}(x+s,y+t)\frac{K(s)}{t}dsdt,
\end{equation}
where the symbol $m(\xi^{1}_{1},\xi^{1}_{2})=m(\zeta):=\hat{K}(\zeta)$ with $\zeta:=\xi^{1}_{1}-\xi^{1}_{2}$ has one dimensional non-degenerate singularity set $\Gamma_{1}$. Let $f_{1}(x,y)=f_{2}(x,y)=e^{ixy}$, one can easily derive that
\begin{equation}\label{1.8}
  T_{m}\otimes BHT(f_{1},f_{2})(x,y)=(f_{1}\cdot f_{2})(x,y)\int_{\mathbb{R}^{2}}K(s)\frac{e^{2ist}}{t}dsdt.
\end{equation}
From \eqref{1.8} and the above counterexample constructed in \eqref{1.6} for operator $BHT\otimes BHT$, we observe that one sufficient condition for precluding the existence of these kinds of counterexamples is $K\in L^{1}$, or equivalently, $m=\hat{K}\in\mathcal{F}(L^{1})$. From the Riemann-Lebesgue theorem, we know that a necessary condition for $m\in\mathcal{F}(L^{1})$ is $m(\zeta)\rightarrow0$ as $|\zeta|\rightarrow\infty$. Moreover, if $K\in L^{1}(\mathbb{R})$ is odd, one can even derive that $|\int_{\mathbb{R}}\frac{m(\zeta)}{\zeta}d\zeta|\lesssim\|K\|_{L^{1}}$ (this indicates that there are many uniformly continuous functions with logarithmic decay rate do not belong to $\mathcal{F}(L^{1})$). Therefore, in order to guarantee that the same $L^{p}$ estimates as the bilinear Hilbert transform are available for bilinear operators $T_{m}\otimes BHT$ and $BHT\otimes BHT$, we need some appropriate decay assumptions on the symbol.

The purpose of this paper is to prove the same $L^{p}$ estimates as BHT for modified bilinear operators $T^{\varepsilon}_{m}\otimes BHT$ and $BHT^{\varepsilon}\otimes BHT$ with arbitrary non-smooth symbols which decay faster than the logarithmic rate.

For $d\geq2$, any two generic vectors $\xi_{1}=(\xi_{1}^{i})_{i=1}^{d}$, $\xi_{2}=(\xi_{2}^{i})_{i=1}^{d}$ in $\mathbb{R}^{d}$ generates naturally the following collection of $d$ vectors in $\mathbb{R}^{2}$:
\begin{equation}\label{1.9}
    \bar{\xi}_{1}=(\xi^{1}_{1},\xi^{1}_{2}), \,\,\,\,\,\, \bar{\xi}_{2}=(\xi^{2}_{1},\xi^{2}_{2}), \,\,\,\,\, \cdots, \,\,\,\,\, \bar{\xi}_{d}=(\xi^{d}_{1},\xi^{d}_{2}).
\end{equation}
For arbitrary small $\varepsilon>0$, let $m^{\varepsilon}=m^{\varepsilon}(\xi)=m^{\varepsilon}(\bar{\xi})$ be a bounded symbol in $L^{\infty}(\mathbb{R}^{2d})$ that is smooth away from the subspaces $\Gamma_{1}\cup\cdots\cup\Gamma_{d-1}\cup\Gamma_{d}$ and satisfying
\begin{equation}\label{1.10}
    dist(\bar{\xi}_{d},\Gamma_{d})^{|\alpha_{d}|}\cdot\int_{\mathbb{R}^{2(d-1)}}
    \frac{|\partial^{\alpha_{1}}_{\bar{\xi}_{1}}\cdots\partial^{\alpha_{d}}_{\bar{\xi}_{d}}m^{\varepsilon}(\bar{\xi})|}{\prod_{i=1}^{d-1}
    dist(\bar{\xi}_{i},\Gamma_{i})^{2-|\alpha_{i}|}}d\bar{\xi}_{1}\cdots d\bar{\xi}_{d-1}\leq B(\varepsilon)<+\infty
\end{equation}
for sufficiently many multi-indices $\alpha_{1},\cdots,\alpha_{d}$, where the constants $B(\varepsilon)\rightarrow+\infty$ as $\varepsilon\rightarrow0$, $dim \, \Gamma_{i}=0$ for $i=1,\cdots,d-1$ and $\Gamma_{d}:=\{(\xi^{d}_{1},\xi^{d}_{2})\in\mathbb{R}^{2}: \xi^{d}_{1}=\xi^{d}_{2}\}$. Denote by $T_{m^{\varepsilon}}^{(d)}$ the bilinear multiplier operator defined by
\begin{equation}\label{1.11}
   T_{m^{\varepsilon}}^{(d)}(f_{1},f_{2})(x):=\int_{\mathbb{R}^{2d}}m^{\varepsilon}(\xi)\hat{f_{1}}(\xi_{1})\hat{f_{2}}(\xi_{2})
    e^{2\pi ix\cdot(\xi_{1}+\xi_{2})}d\xi.
\end{equation}

Our result for bilinear operators $T^{(d)}_{m^{\varepsilon}}$ satisfying \eqref{1.10} and \eqref{1.11} is the following Theorem \ref{main1}.
\begin{thm}\label{main1}
For any $d\geq2$ and $\varepsilon>0$, the bilinear, $d$-parameter multiplier operator $T^{(d)}_{m^{\varepsilon}}$ maps $L^{p_{1}}(\mathbb{R}^{d})\times L^{p_{2}}(\mathbb{R}^{d})\rightarrow L^{p}(\mathbb{R}^{d})$ boundedly for any $1<p_{1}, \, p_{2}\leq\infty$ with $\frac{1}{p}=\frac{1}{p_{1}}+\frac{1}{p_{2}}$ and $\frac{2}{3}<p<\infty$. The implicit constants in the bounds depend only on $p_{1}$, $p_{2}$, $p$, $\varepsilon$, $d$ and tend to infinity as $\varepsilon\rightarrow0$.
\end{thm}

As shown in \cite{MPTT1}, the bilinear and bi-parameter Hilbert transform does not satisfy any $L^p$ estimates. This is the case when the singularity sets
$\Gamma_1$ and $\Gamma_2$ satisfy $dim\, \Gamma_1=dim \Gamma_2=1$. Thus, it is natural to ask  if the $L^p$ estimates will break down for any bilinear and bi-parameter Fourier multiplier operator with $dim\, \Gamma_1=dim \Gamma_2=1$. In other words, will a non-smooth symbol with the same dimensional singularity sets but with a slightly better decay than that for the bilinear and bi-parameter Hilbert transform assure the $L^p$ estimates? Our next theorem will address this issue.

For $d=2$ and arbitrary small $\varepsilon>0$, let $\widetilde{m}^{\varepsilon}=\widetilde{m}^{\varepsilon}(\xi)=\widetilde{m}^{\varepsilon}(\bar{\xi})$ be a bounded symbol in $L^{\infty}(\mathbb{R}^{4})$ that is smooth away from the subspaces $\Gamma_{1}\cup\Gamma_{2}$ and satisfying
\begin{equation}\label{1.10'}
    |\partial^{\alpha_{1}}_{\bar{\xi}_{1}}\partial^{\alpha_{2}}_{\bar{\xi}_{2}}\widetilde{m}^{\varepsilon}(\bar{\xi})|\lesssim
    \prod_{i=1}^{2}\frac{1}{dist(\bar{\xi}_{i},\Gamma_{i})^{|\alpha_{i}|}}\cdot\langle\log_{2}dist(\bar{\xi}_{1},\Gamma_{1})\rangle^{-(1+\varepsilon)}
\end{equation}
for sufficiently many multi-indices $\alpha_{1},\alpha_{2}$, where $\langle x\rangle:=\sqrt{1+x^{2}}$ and $\Gamma_{i}:=\{(\xi^{i}_{1},\xi^{i}_{2})\in\mathbb{R}^{2}: \xi^{i}_{1}=\xi^{i}_{2}\}$ for $i=1,2$. Denote by $T_{\widetilde{m}^{\varepsilon}}^{(2)}$ the bilinear multiplier operator defined by
\begin{equation}\label{1.11'}
   T_{\widetilde{m}^{\varepsilon}}^{(2)}(f_{1},f_{2})(x):=\int_{\mathbb{R}^{4}}\widetilde{m}^{\varepsilon}(\xi)\hat{f_{1}}(\xi_{1})\hat{f_{2}}(\xi_{2})
    e^{2\pi ix\cdot(\xi_{1}+\xi_{2})}d\xi.
\end{equation}

Our result for bilinear operators $T_{\widetilde{m}^{\varepsilon}}^{(2)}$ satisfying \eqref{1.10'} and \eqref{1.11'} is the following Theorem \ref{main2}.
\begin{thm}\label{main2}
For $d=2$ and any $\varepsilon>0$, the bilinear, bi-parameter multiplier operator $T^{(2)}_{\widetilde{m}^{\varepsilon}}$ maps $L^{p_{1}}(\mathbb{R}^{2})\times L^{p_{2}}(\mathbb{R}^{2})\rightarrow L^{p}(\mathbb{R}^{2})$ boundedly for any $1<p_{1}, \, p_{2}\leq\infty$ with $\frac{1}{p}=\frac{1}{p_{1}}+\frac{1}{p_{2}}$ and $\frac{2}{3}<p<\infty$. The implicit constants in the bounds depend only on $p_{1}$, $p_{2}$, $p$, $\varepsilon$ and tend to infinity as $\varepsilon\rightarrow0$. In addition, let the bilinear, bi-parameter operator $BHT^{\varepsilon}\otimes BHT$ be defined by
\begin{equation*}
   BHT^{\varepsilon}\otimes BHT(f_{1},f_{2})(x_{1},x_{2})=p. \, v. \, \int_{\mathbb{R}^{2}}f_{1}(x-s)f_{2}(x+s)\Psi^{\varepsilon}(s_{1})\frac{ds_{1}}{s_{1}}\frac{ds_{2}}{s_{2}}
\end{equation*}
with the function $\Psi^{\varepsilon}$ satisfying
\begin{equation}\label{1.12}
  |\partial^{\alpha_{1}}_{\bar{\xi}_{1}}\widehat{\Psi^{\varepsilon}}(\xi^{1}_{1}-\xi^{1}_{2})|\lesssim |\xi^{1}_{1}-\xi^{1}_{2}|^{-|\alpha_{1}|}\cdot\langle\log_{2}|\xi^{1}_{1}-\xi^{1}_{2}|\rangle^{-(1+\varepsilon)}
\end{equation}
for sufficiently many multi-indices $\alpha_{1}$, then it satisfies the same $L^{p}$ estimates as $T^{(2)}_{\widetilde{m}^{\varepsilon}}$.
\end{thm}
\begin{rem}
For simplicity, we will only consider the bi-parameter case $d=2$ and $\Gamma_{i}=\{(0,0)\}$ ($i=1,\cdots,d-1$) in the proof of Theorem \ref{main1}. It will be clear from the proof (see Section 4) that we can extend the argument to the general $d$-parameter and $dim \, \Gamma_{i}=0$ ($i=1,\cdots,d-1$) cases straightforwardly. In the proof of Theorem \ref{main2}, we will only prove the $L^{p}$ estimates for bilinear and bi-parameter operators $T^{(2)}_{\widetilde{m}^{\varepsilon}}$, since one can observe from the discretization procedure in Section 2 that the bilinear and bi-parameter operator $BHT^{\varepsilon}\otimes BHT$ can be reduced to the same bilinear model operators $\widetilde{\Pi}^{\varepsilon}_{\vec{\mathbb{P}}}$ as $T^{(2)}_{\widetilde{m}^{\varepsilon}}$.
\end{rem}

It's well known that a standard approach to prove $L^{p}$ estimates for one-parameter $n$-linear operators with singular symbols (e.g., Coifman-Meyer multiplier, $BHT$ and one-parameter paraproducts) is the generic estimates of the corresponding $(n+1)$-linear forms consisting of estimates for different \emph{sizes} and \emph{energies} (see \cite{Jung,MS,MTT1,MTT2}), which relied on the one dimensional $BMO$ theory, or more precisely, the John-Nirenberg type inequalities to get good control over the relevant \emph{sizes}. Unfortunately, there is no routine generalization of such approach to multi-parameter settings, for instance, we don't have analogues of the John-Nirenberg inequalities for \emph{dyadic rectangular} $BMO$ spaces in two-parameter case (see \cite{MS}). To overcome these difficulties, in \cite{MPTT1} C. Muscalu, J. Pipher, T. Tao and C. Thiele  developed a completely new approach to prove $L^{p}$ estimates for bi-parameter paraproducts, their essential ideas is to apply the \emph{stopping-time decompositions} based on hybrid square and maximal operators $MM$, $MS$, $SM$ and $SS$, the one dimensional $BMO$ theory and Journ\'{e}'s lemma, and hence could not be extended to solve the general $d$-parameter ($d\geq3$) cases. As to the general $d$-parameter ($d\geq3$) cases, by proving a generic decomposition (see Lemma \ref{generic decomposition}) in \cite{MPTT2}, the authors simplified the arguments introduced by them in \cite{MPTT1} and this simplification works equally well in all $d$-parameter settings. Recently, a pseudo-differential variant of the theorems in \cite{MPTT1,MPTT2} has been established by the current authors in \cite{DL}. Moreover, in the work \cite{CJ} by J. Chen and the second author, they offer a different proof than those in \cite{MPTT1, MPTT2} to establish a H\"{o}rmander type theorem of $L^p$ estimates (and weighted estimates as well) for multi-linear and multi-parameter Fourier multiplier operators with limited smoothness in multi-parameter Sobolev spaces.

However, in this paper, in order to prove our main Theorems \ref{main1} and \ref{main2} in bi-parameter settings, we have at least two different difficulties from \cite{MPTT1}. First, observe that if one restricts the sum of \emph{tri-tiles} $P''\in\mathbb{P}''$ in the definitions of discrete model operators (see Section 2) to a \emph{tree} then one essentially gets a discrete paraproduct on $x_{2}$ variable, which can be estimated by the $MM$, $MS$, $SM$ and $SS$ functions, but due to the \emph{extra degree of freedom} in frequency in $x_{2}$ direction, there are infinitely many such paraproducts in the summation, so it's difficult for us to carry out the \emph{stopping-time decompositions} by using the hybrid square and maximal operators. Second, in the proof of Theorem \ref{main2}, note that there are infinitely many \emph{tri-tiles} $P'\in\mathbb{P}'$ with the property that $I_{P'}=I_{0}$ for a certain fixed dyadic interval $I_{0}$ of the same length as $I_{P'}$, so we can't get estimate $\sum_{P'}|I_{P'}|\lesssim|\widetilde{I}|$ for all dyadic intervals $I_{P'}\subseteq\widetilde{I}$ with comparable lengths, and hence we can't apply the Journ\'{e}'s lemma either. By making use of the $L^{2}$ \emph{sizes} and $L^{2}$ \emph{energies} estimates of the tri-linear forms, the \emph{almost orthogonality} of \emph{wave packets} associated with different \emph{tiles} of distinct \emph{trees} and the decay assumptions on the symbols, we are able to overcome these difficulties in the proof of Theorem \ref{main1} and \ref{main2} in bi-parameter settings.

Nevertheless, in the proof of Theorem \ref{main2} in general $d$-parameter settings ($d\geq3$), one easily observe that the generic decomposition will destroy the \emph{perfect orthogonality} of \emph{wave packets} associated with distinct \emph{tiles} which have disjoint frequency intervals in both $x_{1}$ and $x_{2}$ directions, thus we can't apply the generic decomposition to extend the results of Theorem 1.4 to higher parameters $d\geq3$. For the proof of Theorem \ref{main1}, we are able to apply the generic decomposition lemma (Lemma \ref{generic decomposition}) to the $d-1$ variables $x_{1},\cdots,x_{d-1}$. Although one can't obtain that $supp \, \Phi^{3,\ell}_{\widetilde{P'}}\otimes\Phi^{3}_{P''}$ is entirely contained in the exceptional set $U$ as in \cite{MPTT2}, but one can observe that the support set is contained in $U$ in all the $x_{1},\cdots,x_{d-1}$ variables except the last $x_{d}$. Therefore, we only need to consider the distance from the support set to the set $E'_{3}$ in $x_{d}$ direction and obtain enough decay factors for summation, the extension of the proof to the general $d$-parameter ($d\geq3$) cases is straightforward.

The rest of this paper is organized as follows. In Section 2 we reduce the proof of Theorem \ref{main1} and Theorem \ref{main2} to proving restricted weak type estimates of discrete bilinear model operators $\Pi^{\varepsilon}_{\vec{\mathbb{P}}}$ and $\widetilde{\Pi}^{\varepsilon}_{\vec{\mathbb{P}}}$ (Proposition \ref{RW-type equivalent}). Section 3 is devoted to giving a review of the definitions and useful properties about trees, $L^{2}$ sizes and $L^{2}$ energies introduced in \cite{MTT2}. In Section 4 and 5 we carry out the proof of Proposition \ref{RW-type equivalent}, which completes the proof of our main theorems, Theorem \ref{main1} and Theorem \ref{main2}, respectively.

\section{Reduction to restricted weak type estimates of discrete bilinear model operators $\Pi^{\varepsilon}_{\vec{\mathbb{P}}}$ and $\widetilde{\Pi}^{\varepsilon}_{\vec{\mathbb{P}}}$}

\subsection{Discretization}
As we can see from the study of multi-parameter and multi-linear Coifman-Meyer multiplier operators (see e.g. \cite{MTT1,MPTT1,MTT2,MPTT2}), a standard approach to obtain $L^{p}$ estimates of bilinear operators $T^{(d)}_{m^{\varepsilon}}$ and $T^{(2)}_{\widetilde{m}^{\varepsilon}}$ is to reduce them into discrete sums of inner products with wave packets (see \cite{Thiele1}).

\subsubsection{Discretization for bilinear, bi-parameter operators $T^{(2)}_{m^{\varepsilon}}$ with $\Gamma_{1}=\{(0,0)\}$}
We will proceed the discretization procedure as follows. First, we need to decompose the symbol $m^{\varepsilon}(\xi)$ in a natural way. To this end, for the first spatial variable $x_{1}$, we decompose the region $\{\bar{\xi}_{1}=(\xi^{1}_{1},\xi^{1}_{2})\in\mathbb{R}^{2}\setminus\{(0,0)\}\}$ by using \emph{Whitney squares} with respect to the singularity point $\{\xi^{1}_{1}=\xi^{1}_{2}=0\}$; while for the last spatial variable $x_{2}$, we decompose the region $\{\bar{\xi}_{2}=(\xi^{2}_{1},\xi^{2}_{2})\in\mathbb{R}^{2}: \xi^{2}_{1}\neq\xi^{2}_{2}\}$ by using \emph{Whitney squares} with respect to the singularity line $\Gamma_{2}=\{\xi^{2}_{1}=\xi^{2}_{2}\}$. In order to describe our discretization procedure clearly, let us first recall some standard notation and definitions in \cite{MTT2}.

An interval $I$ on the real line $\mathbb{R}$ is called dyadic if it is of the form $I=2^{-k}[n, \, n+1]$ for some $k, \, n\in\mathbb{Z}$. An interval is said to be a \emph{shifted dyadic interval} if it is of the form $2^{-k}[j+\alpha,j+1+\alpha]$ for any $k,j\in\mathbb{Z}$ and $\alpha\in\{0,\frac{1}{3},-\frac{1}{3}\}$. A \emph{shifted dyadic cube} is a set of the form $Q=Q_{1}\times Q_{2}\times Q_{3}$, where each $Q_{j}$ is a shifted dyadic interval and they all have the same length. A \emph{shifted dyadic quasi-cube} is a set $Q=Q_{1}\times Q_{2}\times Q_{3}$, where $Q_{j}$ ($j=1,2,3$) are shifted dyadic intervals satisfying less restrictive condition $|Q_{1}|\simeq|Q_{2}|\simeq|Q_{3}|$. One easily observe that for every cube $Q\subseteq\mathbb{R}^{3}$, there exists a shifted dyadic cube $\widetilde{Q}$ such that $Q\subset\frac{7}{10}\widetilde{Q}$ (the cube having the same center as $\widetilde{Q}$ but with side length $\frac{7}{10}$ that of $\widetilde{Q}$) and $diam(Q)\simeq diam(\widetilde{Q})$.

The same terminology will also be used in the plane $\mathbb{R}^{2}$. The only difference is that the previous cubes now become squares.

For any cube and square $Q$, we will denote the side length of $Q$ by $\ell(Q)$ for short and denote the reflection of $Q$ with respect to the origin by $-Q$ hereafter.

\begin{defn}\label{bump functions}(\cite{MS,MPTT2})
For $J\subseteq\mathbb{R}$ an arbitrary interval, we say that a smooth function $\Phi_{J}$ is a bump adapted to $J$, if and only if the following inequalities hold:
\begin{equation}\label{2.1}
  |\Phi_{J}^{(l)}(x)|\lesssim_{l,\alpha}\frac{1}{|J|^{l}}\cdot\frac{1}{(1+\frac{dist(x,J)}{|J|})^{\alpha}}
\end{equation}
for every integer $\alpha\in\mathbb{N}$ and for sufficiently many derivatives $l\in\mathbb{N}$. If $\Phi_{J}$ is a bump adapted to $J$, we say that $|J|^{-\frac{1}{2}}\Phi_{J}$ is an $L^{2}$-normalized bump adapted to $J$.
\end{defn}

Now let $\varphi\in\mathcal{S}(\mathbb{R})$ be an even Schwartz function such that $supp \, \hat{\varphi}\subseteq[-\frac{3}{16},\frac{3}{16}]$ and $\hat{\varphi}(\xi)=1$ on $[-\frac{1}{6},\frac{1}{6}]$, and define $\psi\in\mathcal{S}(\mathbb{R})$ to be the Schwartz function whose Fourier transform satisfies $\hat{\psi}(\xi):=\hat{\varphi}(\frac{\xi}{4})-\hat{\varphi}(\frac{\xi}{2})$ and $supp \, \hat{\psi}\subseteq[-\frac{3}{4},-\frac{1}{3}]\cup[\frac{1}{3},\frac{3}{4}]$, such that $0\leq\hat{\varphi}(\xi), \hat{\psi}(\xi)\leq1$. Then, for every integer $k\in\mathbb{Z}$, we define $\widehat{\varphi_{k}}, \, \widehat{\psi_{k}}\in\mathcal{S}(\mathbb{R})$ by
\begin{equation}\label{2.2}
  \widehat{\varphi_{k}}(\xi):=\hat{\varphi}(\frac{\xi}{2^{k}}), \,\,\,\,\,\, \widehat{\psi_{k}}(\xi):=\hat{\psi}(\frac{\xi}{2^{k}})=\widehat{\varphi_{k+2}}(\xi)-\widehat{\varphi_{k+1}}(\xi)
\end{equation}
and observe that
$$supp \, \widehat{\varphi_{k}}\subseteq[-\frac{3}{16}\cdot2^{k},\frac{3}{16}\cdot2^{k}], \,\,\,\,\,\, supp \, \widehat{\psi_{k}}\subseteq[-\frac{3}{4}\cdot2^{k},-\frac{1}{3}\cdot2^{k}]\cup[\frac{1}{3}\cdot2^{k},\frac{3}{4}\cdot2^{k}],$$
and $supp \, \widehat{\psi_{k}}\bigcap supp \, \widehat{\psi_{k'}}=\emptyset$ for any integers $k, \, k'\in\mathbb{Z}$ such that $|k-k'|\geq2$, $supp \, \hat{\varphi}\bigcap supp \, \widehat{\psi_{k}}=\emptyset$ for any integer $k\geq0$. One easily obtain the homogeneous Littlewood-Paley dyadic decomposition
\begin{equation}\label{2.3}
  1=\sum_{k\in\mathbb{Z}}\widehat{\psi_{k}}(\xi), \,\,\,\,\,\,\,\,\,\, \forall\xi\in\mathbb{R}\setminus\{0\}
\end{equation}
and inhomogeneous Littlewood-Paley dyadic decomposition
\begin{equation}\label{2.4}
  1=\hat{\varphi}(\xi)+\sum_{k\geq-1}\widehat{\psi_{k}}(\xi), \,\,\,\,\,\,\,\,\, \forall\xi\in\mathbb{R},
\end{equation}
as a consequence, we get decomposition for the product $1(\xi^{1}_{1},\xi^{1}_{2})=1(\xi^{1}_{1})\cdot1(\xi^{1}_{2})$ as follows:
\begin{equation}\label{2.5}
  1(\xi^{1}_{1},\xi^{1}_{2})=\sum_{k'\in\mathbb{Z}}\widehat{\varphi_{k'}}(\xi^{1}_{1})\widehat{\psi_{k'}}(\xi^{1}_{2})
  +\sum_{k'\in\mathbb{Z}}\widehat{\psi_{k'}}(\xi^{1}_{1})\widehat{\widetilde{\psi_{k'}}}(\xi^{1}_{2})
  +\sum_{k'\in\mathbb{Z}}\widehat{\psi_{k'}}(\xi^{1}_{1})\widehat{\varphi_{k'}}(\xi^{1}_{2})
\end{equation}
for every $(\xi^{1}_{1},\xi^{1}_{2})\neq(0,0)$, where
$$\widehat{\widetilde{\psi_{k'}}}:=\sum_{|k-k'|\leq1, \, k\in\mathbb{Z}}\widehat{\psi_{k}}, \,\,\,\,\,\,\,\,\,\, \forall k'\in\mathbb{Z}.$$

By writing the characteristic function of the plane $(\xi^{1}_{1},\xi^{1}_{_{2}})$ into finite sums of smoothed versions of characteristic functions of cones as in \eqref{2.5}, we can decompose the operator $T^{(2)}_{m^{\varepsilon}}$ into a finite sum of several parts in $x_{1}$ direction. Since all the operators obtained in this decomposition can be treated in the same way, we will discuss in detail only one of them. More precisely, let
\begin{equation}\label{2.6}
  \widetilde{\mathbb{Q}'}:=\{\widetilde{Q'}=\widetilde{Q'_{1}}\times\widetilde{Q'_{2}}\subseteq\mathbb{R}^{2}: \widetilde{Q'_{1}}:=2^{k'}[-\frac{1}{2},\frac{1}{2}], \, \widetilde{Q'_{2}}:=2^{k'}[\frac{1}{24},\frac{25}{24}], \, \forall k'\in\mathbb{Z}\},
\end{equation}
for each square $\widetilde{Q'}\in\widetilde{\mathbb{Q}'}$, we define bump functions $\phi_{\widetilde{Q'_{i}},i}$ ($i=1,2$) adapted to intervals $\widetilde{Q'_{i}}$ and satisfying $supp \, \phi_{\widetilde{Q'_{i}},i}\subseteq\frac{9}{10}\widetilde{Q'_{i}}$ by
\begin{equation}\label{2.7}
  \phi_{\widetilde{Q'_{1}},1}(\xi):=\hat{\varphi}(\frac{\xi}{\ell(\widetilde{Q'})})=\widehat{\varphi_{k'}}(\xi)
\end{equation}
and
\begin{equation}\label{2.8}
  \phi_{\widetilde{Q'_{2}},2}(\xi):=\hat{\psi}(\frac{\xi}{\ell(\widetilde{Q'})})\cdot\chi_{\{\xi>0\}}=\widehat{\psi_{k'}}(\xi)\cdot\chi_{\{\xi>0\}},
\end{equation}
respectively, and finally define smooth bump functions $\phi_{\widetilde{Q'}}$ adapted to $\widetilde{Q'}$ and satisfying $supp \, \phi_{\widetilde{Q'}}\subseteq\frac{9}{10}\widetilde{Q'}$ by
\begin{equation}\label{2.9}
  \phi_{\widetilde{Q'}}(\xi^{1}_{1},\xi^{1}_{2}):=\phi_{\widetilde{Q'_{1}},1}(\xi^{1}_{1})\cdot\phi_{\widetilde{Q'_{2}},2}(\xi^{1}_{2}).
\end{equation}
Without loss of generality, we will only consider the smoothed characteristic function of the cone $\{(\xi^{1}_{1},\xi^{1}_{2})\in\mathbb{R}^{2}: |\xi^{1}_{1}|\lesssim|\xi^{1}_{2}|, \, \xi^{1}_{2}>0\}$ in the decomposition \eqref{2.5} from now on, which is defined by
\begin{equation}\label{2.10}
  \sum_{\widetilde{Q'}\in\widetilde{\mathbb{Q}'}}\phi_{\widetilde{Q'}}(\xi^{1}_{1},\xi^{1}_{2}).
\end{equation}

As to the $x_{2}$ direction, we consider the collection $\mathbb{Q}''$ of all shifted dyadic squares $Q''=Q''_{1}\times Q''_{2}$ satisfying
\begin{equation}\label{2.11}
  Q''\subseteq\{(\xi^{2}_{1},\xi^{2}_{2})\in\mathbb{R}^{2}: \xi^{2}_{1}\neq\xi^{2}_{2}\}, \,\,\,\,\,\,\,\,\,\, dist(Q'',\Gamma_{2})\simeq10^{4}diam(Q'').
\end{equation}
We can split the collection $\mathbb{Q''}$ into two disjoint sub-collections, that is, define
\begin{equation}\label{2.12}
  \mathbb{Q}''_{\mathbb{I}}:=\{Q''\in\mathbb{Q}'': Q''\subseteq\{\xi^{2}_{1}<\xi^{2}_{2}\}\}, \,\,\,\,\,\,\,\,\,
  \mathbb{Q}''_{\mathbb{II}}:=\{Q''\in\mathbb{Q}'': Q''\subseteq\{\xi^{2}_{1}>\xi^{2}_{2}\}\}.
\end{equation}
Since the set of squares $\{\frac{7}{10}Q'': Q''\in\mathbb{Q}''\}$ also forms a finitely overlapping cover of the region $\{\xi^{2}_{1}\neq\xi^{2}_{2}\}$, we can apply a standard partition of unity and write the symbol $\chi_{\{\xi^{2}_{1}\neq\xi^{2}_{2}\}}$ as
\begin{equation}\label{2.13}
  \chi_{\{\xi^{2}_{1}\neq\xi^{2}_{2}\}}=\sum_{Q''\in\mathbb{Q}''}\phi_{Q''}(\xi^{2}_{1},\xi^{2}_{2})=
  \{\sum_{Q''\in\mathbb{Q}''_{\mathbb{I}}}+\sum_{Q''\in\mathbb{Q}''_{\mathbb{II}}}\}\phi_{Q''}(\xi^{2}_{1},\xi^{2}_{2})
  =\chi_{\{\xi^{2}_{1}<\xi^{2}_{2}\}}+\chi_{\{\xi^{2}_{1}>\xi^{2}_{2}\}},
\end{equation}
where each $\phi_{Q''}$ is a smooth bump function adapted to $Q''$ and supported in $\frac{8}{10}Q''$.

One can easily observe that we only need to discuss in detail one term in the decomposition \eqref{2.13}, since the other term can be treated in the same way. Without loss of generality, we will only consider the first term in the decomposition \eqref{2.13}, that is, the characteristic function $\chi_{\{\xi^{2}_{1}<\xi^{2}_{2}\}}$ of the upper half plane with respect to singularity line $\Gamma_{2}$, which can be written as
\begin{equation}\label{2.14}
  \chi_{\{\xi^{2}_{1}<\xi^{2}_{2}\}}=\sum_{Q''\in\mathbb{Q}''_{\mathbb{I}}}\phi_{Q''}(\xi^{2}_{1},\xi^{2}_{2}).
\end{equation}

In a word, we only need to consider the bilinear operator $T^{(2)}_{m^{\varepsilon},(lh,\mathbb{I})}$ given by
\begin{equation}\label{2.15}
  T^{(2)}_{m^{\varepsilon},(lh,\mathbb{I})}(f_{1},f_{2})(x):=\sum_{\widetilde{Q'}\in\widetilde{\mathbb{Q}'}, \, Q''\in\mathbb{Q}''_{\mathbb{I}}}
  \int_{\mathbb{R}^{4}}m^{\varepsilon}(\xi)\phi_{\widetilde{Q'}}(\bar{\xi}_{1})\phi_{Q''}(\bar{\xi}_{2})\widehat{f_{1}}(\xi_{1})\widehat{f_{2}}(\xi_{2})
  e^{2\pi ix\cdot(\xi_{1}+\xi_{2})}d\xi
\end{equation}
from now on, and the proof of Theorem \ref{main1} can be reduced to proving the following $L^{p}$ estimates for $T^{(2)}_{m^{\varepsilon},(lh,\mathbb{I})}$:
\begin{equation}\label{2.16}
  \|T^{(2)}_{m^{\varepsilon},(lh,\mathbb{I})}(f_{1},f_{2})\|_{L^{p}(\mathbb{R}^{2})}\lesssim_{\varepsilon,p,p_{1},p_{2}}
  \|f_{1}\|_{L^{p_{1}}(\mathbb{R}^{2})}\cdot\|f_{2}\|_{L^{p_{2}}(\mathbb{R}^{2})},
\end{equation}
as long as $1<p_{1}, \, p_{2}\leq\infty$ and $0<\frac{1}{p}=\frac{1}{p_{1}}+\frac{1}{p_{2}}<\frac{3}{2}$.

On one hand, since $\xi^{1}_{1}\in supp \, \phi_{\widetilde{Q'_{1}},1}\subseteq\ell(\widetilde{Q'})[-\frac{3}{16},\frac{3}{16}]$ and $\xi^{1}_{2}\in supp \, \phi_{\widetilde{Q'_{2}},2}\subseteq\ell(\widetilde{Q'})[\frac{1}{3},\frac{3}{4}]$, it follows that $-\xi^{1}_{1}-\xi^{1}_{2}\in\ell(\widetilde{Q'})[-\frac{15}{16},-\frac{7}{48}]$, and as a consequence, there exists a interval $\widetilde{Q'_{3}}:=\ell(\widetilde{Q'})[-\frac{25}{24},-\frac{1}{24}]$ and a bump function $\phi_{\widetilde{Q'_{3}},3}$ adapted to $\widetilde{Q'_{3}}$ such that $supp \, \phi_{\widetilde{Q'_{3}},3}\subseteq\ell(\widetilde{Q'})[-\frac{23}{24},-\frac{1}{8}]\subseteq\frac{9}{10}\widetilde{Q'_{3}}$ and $\phi_{\widetilde{Q'_{3}},3}\equiv1$ on $\ell(\widetilde{Q'})[-\frac{15}{16},-\frac{7}{48}]$.

On the other hand, observe that there exist bump functions $\phi_{Q''_{i},i}$ ($i=1,2$) adapted to the shifted dyadic interval $Q''_{i}$ such that $supp \, \phi_{Q''_{i},i}\subseteq\frac{9}{10}Q''_{i}$ and $\phi_{Q''_{i},i}\equiv1$ on $\frac{8}{10}Q''_{i}$ ($i=1,2$) respectively, and $supp \, \phi_{Q''}\subseteq\frac{8}{10}Q''$, thus one has $\phi_{Q''_{1},1}\cdot\phi_{Q''_{2},2}\equiv1$ on $supp \, \phi_{Q''}$. Since $\xi^{2}_{1}\in supp \, \phi_{Q''_{1},1}\subseteq\frac{9}{10}Q''_{1}$ and $\xi^{2}_{2}\in supp \, \phi_{Q''_{2},2}\subseteq\frac{9}{10}Q''_{2}$, it follows that $-\xi^{2}_{1}-\xi^{2}_{2}\in-\frac{9}{10}Q''_{1}-\frac{9}{10}Q''_{2}$, and as a consequence, one can find a shifted dyadic interval $Q''_{3}$ with the property that $-\frac{9}{10}Q''_{1}-\frac{9}{10}Q''_{2}\subseteq\frac{7}{10}Q''_{3}$ and also satisfying $|Q''_{1}|=|Q''_{2}|\simeq|Q''_{3}|$. In particular, there exists bump function $\phi_{Q''_{3},3}$ adapted to $Q''_{3}$ and supported in $\frac{9}{10}Q''_{3}$ such that $\phi_{Q''_{3},3}\equiv1$ on $-\frac{9}{10}Q''_{1}-\frac{9}{10}Q''_{2}$.

We denote by $\widetilde{\mathbf{Q}'}$ the collection of all cubes $\widetilde{Q'}:=\widetilde{Q'_{1}}\times\widetilde{Q'_{2}}\times\widetilde{Q'_{3}}$ with $\widetilde{Q'_{1}}\times\widetilde{Q'_{2}}\in\widetilde{\mathbb{Q}'}$ and $\widetilde{Q'_{3}}$ be defined as above, and denote by $\mathbf{Q}''$ the collection of all shifted dyadic quasi-cubes $Q'':=Q''_{1}\times Q''_{2}\times Q''_{3}$ with $Q''_{1}\times Q''_{2}\in\mathbb{Q}''_{\mathbb{I}}$ and $Q''_{3}$ be defined as above.

\begin{defn}\label{cube-sparse}(\cite{MTT2})
We say that a collection of shifted dyadic quasi-cubes (cubes) is \emph{sparse} if and only if for every $j=1,2,3$,\\
(i) whenever $Q$ and $\widetilde{Q}$ belong to this collection and $|Q_{j}|<|\widetilde{Q}_{j}|$ then $10^{8}|Q_{j}|\leq|\widetilde{Q}_{j}|$;\\
(ii) whenever $Q$ and $\widetilde{Q}$ belong to this collection and $|Q_{j}|=|\widetilde{Q}_{j}|$ then $10^{8}Q_{j}\cap 10^{8}\widetilde{Q}_{j}=\emptyset$.
\end{defn}

In fact, it is not difficult to see that the collection $\mathbf{Q}''$ can be split into a sum of finitely many \emph{sparse} collection of shifted dyadic quasi-cubes. Therefore, we can assume from now on that the collection $\mathbf{Q}''$ is \emph{sparse}.

Assuming this we then observe that, for any $Q''$ in such a sparse collection $\mathbf{Q}''$, there exists a unique shifted dyadic cube $\widetilde{Q''}$ in $\mathbb{R}^{3}$ such that $Q''\subseteq\frac{7}{10}\widetilde{Q''}$ and with property that $diam(Q'')\simeq diam(\widetilde{Q''})$. This allows us in particular to assume further that $\mathbf{Q}''$ is a sparse collection of shifted dyadic cubes (that is, $|Q''_{1}|=|Q''_{2}|=|Q''_{3}|=\ell(Q'')$).

Now consider the trilinear form $\Lambda^{(2)}_{m^{\varepsilon},(lh,\mathbb{I})}(f_{1},f_{2},f_{3})$ associated to $T^{(2)}_{m^{\varepsilon},(lh,\mathbb{I})}(f_{1},f_{2})$, which can be written as
\begin{eqnarray}\label{2.17}
   &&\Lambda^{(2)}_{m^{\varepsilon},(lh,\mathbb{I})}(f_{1},f_{2},f_{3}):=\int_{\mathbb{R}^{2}}T^{(2)}_{m^{\varepsilon},(lh,\mathbb{I})}(f_{1},f_{2})(x)f_{3}(x)dx\\
 \nonumber &=&\sum_{\widetilde{Q'}\in\widetilde{\mathbf{Q}'},Q''\in\mathbf{Q}''}\int_{\xi_{1}+\xi_{2}+\xi_{3}=0}
 m^{\varepsilon}_{\widetilde{Q'},Q''}(\xi_{1},\xi_{2},\xi_{3})(f_{1}\ast(\check{\phi}_{\widetilde{Q'_{1}},1}\otimes\check{\phi}_{Q''_{1},1}))^{\wedge}(\xi_{1})\\
 \nonumber &&\quad\quad\quad\quad\quad\quad
 \times(f_{2}\ast(\check{\phi}_{\widetilde{Q'_{2}},2}\otimes\check{\phi}_{Q''_{2},2}))^{\wedge}(\xi_{2})
 (f_{3}\ast(\check{\phi}_{\widetilde{Q'_{3}},3}\otimes\check{\phi}_{Q''_{3},3}))^{\wedge}(\xi_{3})d\xi_{1}d\xi_{2}d\xi_{3},
\end{eqnarray}
where $\xi_{i}=(\xi^{1}_{i},\xi^{2}_{i})$ for $i=1,2,3$, while
\begin{equation}\label{2.18}
  m^{\varepsilon}_{\widetilde{Q'},Q''}(\xi_{1},\xi_{2},\xi_{3}):=m^{\varepsilon}(\xi_{1},\xi_{2})\cdot
  (\widetilde{\phi}_{\widetilde{Q'}}\otimes(\phi_{Q''_{1}\times Q''_{2}}\cdot\widetilde{\phi}_{Q''_{3},3}))(\xi_{1},\xi_{2},\xi_{3}),
\end{equation}
where $\widetilde{\phi}_{\widetilde{Q'}}$ is an appropriate smooth function of variable $(\xi^{1}_{1},\xi^{1}_{2},\xi^{1}_{3})$ which is supported on a slightly larger cube (with a constant magnification independent of $\ell(\widetilde{Q'})$) than $supp \, (\phi_{\widetilde{Q'_{1}},1}(\xi^{1}_{1})\phi_{\widetilde{Q'_{2}},2}(\xi^{1}_{2})\phi_{\widetilde{Q'_{3}},3}(\xi^{1}_{3}))$ and equals $1$ on
$supp \, (\phi_{\widetilde{Q'_{1}},1}(\xi^{1}_{1})\phi_{\widetilde{Q'_{2}},2}(\xi^{1}_{2})\phi_{\widetilde{Q'_{3}},3}(\xi^{1}_{3}))$, the function $\phi_{Q''_{1}\times Q''_{2}}(\xi^{2}_{1},\xi^{2}_{2})$ is one term of the partition of unity defined in \eqref{2.14}, $\widetilde{\phi}_{Q''_{3},3}$ is an appropriate smooth function of variable $\xi^{2}_{3}$ supported on a slightly larger interval (with a constant magnification independent of $\ell(Q'')$) than $supp \, \phi_{Q''_{3},3}$, which equals $1$ on $supp \, \phi_{Q''_{3},3}$. We can decompose $m^{\varepsilon}_{\widetilde{Q'},Q''}(\xi_{1},\xi_{2},\xi_{3})$ as a Fourier series:
\begin{equation}\label{2.19}
  m^{\varepsilon}_{\widetilde{Q'},Q''}(\xi_{1},\xi_{2},\xi_{3})=\sum_{\vec{n}_{1},\vec{n}_{2},\vec{n}_{3}\in\mathbb{Z}^{2}}
  C^{\varepsilon,\widetilde{Q'},Q''}_{\vec{n}_{1},\vec{n}_{2},\vec{n}_{3}}
  e^{2\pi i(n^{'}_{1},n^{'}_{2},n^{'}_{3})\cdot(\xi^{1}_{1},\xi^{1}_{2},\xi^{1}_{3})/\ell(\widetilde{Q'})}
  e^{2\pi i(n^{''}_{1},n^{''}_{2},n^{''}_{3})\cdot(\xi^{2}_{1},\xi^{2}_{2},\xi^{2}_{3})/\ell(Q'')},
\end{equation}
where the Fourier coefficients $C^{\varepsilon,\widetilde{Q'},Q''}_{\vec{n}_{1},\vec{n}_{2},\vec{n}_{3}}$ are given by
\begin{eqnarray}\label{2.20}
 \nonumber C^{\varepsilon,\widetilde{Q'},Q''}_{\vec{n}_{1},\vec{n}_{2},\vec{n}_{3}}&=&\int_{\mathbb{R}^{6}}
  m^{\varepsilon}_{\widetilde{Q'},Q''}((\ell(\widetilde{Q'})\xi^{1}_{1},\ell(Q'')\xi^{2}_{1}),(\ell(\widetilde{Q'})\xi^{1}_{2},\ell(Q'')\xi^{2}_{2}),
  (\ell(\widetilde{Q'})\xi^{1}_{3},\ell(Q'')\xi^{2}_{3}))\\
  && \quad\quad\quad\quad\quad\quad\quad\quad\quad\quad\quad\quad\quad\quad\quad\quad\quad \times e^{-2\pi i(\vec{n}_{1}\cdot\xi_{1}+\vec{n}_{2}\cdot\xi_{2}+\vec{n}_{3}\cdot\xi_{3})}
  d\xi_{1}d\xi_{2}d\xi_{3}.
\end{eqnarray}
Then, by a straightforward calculation, we can rewrite \eqref{2.17} as
\begin{eqnarray}\label{2.21}
\nonumber &&\Lambda^{(2)}_{m^{\varepsilon},(lh,\mathbb{I})}(f_{1},f_{2},f_{3})=\sum_{\widetilde{Q'}\in\widetilde{\mathbf{Q}'},Q''\in\mathbf{Q}''}
\sum_{\vec{n}_{1},\vec{n}_{2},\vec{n}_{3}\in\mathbb{Z}^{2}}C^{\varepsilon,\widetilde{Q'},Q''}_{\vec{n}_{1},\vec{n}_{2},\vec{n}_{3}}\int_{\mathbb{R}^{2}}\\
\nonumber &&(f_{1}\ast(\check{\phi}_{\widetilde{Q'_{1}},1}\otimes\check{\phi}_{Q''_{1},1}))(x-(\frac{n'_{1}}{\ell(\widetilde{Q'})},\frac{n^{''}_{1}}{\ell(Q'')}))
  (f_{2}\ast(\check{\phi}_{\widetilde{Q'_{2}},2}\otimes\check{\phi}_{Q''_{2},2}))(x-(\frac{n'_{2}}{\ell(\widetilde{Q'})},\frac{n^{''}_{2}}{\ell(Q'')}))\\
&&\quad\quad\quad\quad\quad\quad\quad\quad\quad\quad\quad\quad\quad\quad\quad\quad
\times(f_{3}\ast(\check{\phi}_{\widetilde{Q'_{3}},3}\otimes\check{\phi}_{Q''_{3},3}))(x-(\frac{n'_{3}}{\ell(\widetilde{Q'})},\frac{n^{''}_{3}}{\ell(Q'')}))dx.
\end{eqnarray}

\begin{defn}\label{wave packet}(\cite{MTT2,Thiele1})
An arbitrary dyadic rectangle of area $1$ in the phase-space plane is called a \emph{Heisenberg box} or \emph{tile}. Let $P:=I_{P}\times\omega_{P}$ be a tile. A $L^{2}$-normalized wave packet on $P$ is a function $\Phi_{p}$ which has Fourier support $supp \, \hat{\Phi}_{p}\subseteq\frac{9}{10}\omega_{P}$ and obeys the estimates
\begin{equation*}
  |\Phi_{P}(x)|\lesssim|I_{P}|^{-\frac{1}{2}}(1+\frac{dist(x,I_{P})}{|I_{P}|})^{-M}
\end{equation*}
for all $M>0$, where the implicit constant depends on $M$.
\end{defn}

Now we define $\phi^{n'_{i}}_{\widetilde{Q'_{i}},i}:=e^{2\pi in'_{i}\xi^{1}_{i}/\ell(\widetilde{Q'})}\cdot\phi_{\widetilde{Q'_{i}},i}$ and $\phi^{n''_{i}}_{Q''_{i},i}:=e^{2\pi in''_{i}\xi^{2}_{i}/\ell(Q'')}\cdot\phi_{Q''_{i},i}$ for $i=1,2,3$. Since any $\widetilde{Q'}\in\widetilde{\mathbf{Q}'}$ and $Q''\in\mathbf{Q}''$ are both shifted dyadic cubes, there exists integers $k',k''\in\mathbb{Z}$ such that $\ell(\widetilde{Q'})=|\widetilde{Q'_{1}}|=|\widetilde{Q'_{2}}|=|\widetilde{Q'_{3}}|=2^{k'}$ and $\ell(Q'')=|Q''_{1}|=|Q''_{2}|=|Q''_{3}|=2^{k''}$ respectively. By splitting the integral region $\mathbb{R}^{2}$ into the union of unit squares, the $L^{2}$-normalization procedure and simple calculations, we can rewrite \eqref{2.21} as
\begin{eqnarray}\label{2.22}
   &&\Lambda^{(2)}_{m^{\varepsilon},(lh,\mathbb{I})}(f_{1},f_{2},f_{3})\\
  \nonumber &=&\sum_{\vec{n}_{1},\vec{n}_{2},\vec{n}_{3}\in\mathbb{Z}^{2}}\sum_{\widetilde{Q'}\in\widetilde{\mathbf{Q}'},Q''\in\mathbf{Q}''}
  \int_{0}^{1}\int_{0}^{1}\sum_{\substack{\widetilde{I'} \,\, dyadic, \\ |\widetilde{I'}|=2^{-k'}}}\sum_{\substack{I'' \,\, dyadic, \\ |I''|=2^{-k''}}}\frac{C^{\varepsilon,\widetilde{Q'},Q''}_{\vec{n}_{1},\vec{n}_{2},\vec{n}_{3}}}{|\widetilde{I'}|^{\frac{1}{2}}\times|I''|^{\frac{1}{2}}}
  \langle f_{1},\check{\phi}^{n'_{1},\nu'}_{\widetilde{I'},\widetilde{Q'_{1}},1}\otimes\check{\phi}^{n''_{1},\nu''}_{I'',Q''_{1},1}\rangle\\
  \nonumber &&\quad\quad\quad\quad\quad\quad\quad\quad\quad\quad\quad\quad\times
  \langle f_{2},\check{\phi}^{n'_{2},\nu'}_{\widetilde{I'},\widetilde{Q'_{2}},2}\otimes\check{\phi}^{n''_{2},\nu''}_{I'',Q''_{2},2}\rangle
  \langle f_{3},\check{\phi}^{n'_{3},\nu'}_{\widetilde{I'},\widetilde{Q'_{3}},3}\otimes\check{\phi}^{n''_{3},\nu''}_{I'',Q''_{3},3}\rangle d\nu'd\nu''\\
 \nonumber &=:&\sum_{\vec{n}_{1},\vec{n}_{2},\vec{n}_{3}\in\mathbb{Z}^{2}}\int_{0}^{1}\int_{0}^{1}\sum_{\vec{P}:=\widetilde{P'}\otimes P''\in\vec{\mathbb{P}}}
 \frac{C^{\varepsilon}_{Q_{\vec{P}},\vec{n}_{1},\vec{n}_{2},\vec{n}_{3}}}{|I_{\vec{P}}|^{\frac{1}{2}}}
 \langle f_{1},\Phi^{1,\vec{n}_{1},\nu}_{\vec{P}_{1}}\rangle\langle f_{2},\Phi^{2,\vec{n}_{2},\nu}_{\vec{P}_{2}}\rangle
 \langle f_{3},\Phi^{3,\vec{n}_{3},\nu}_{\vec{P}_{3}}\rangle d\nu,
\end{eqnarray}
where the notation $\langle\cdot,\cdot\rangle$ denotes the complex scalar $L^{2}$ inner product, the Fourier coefficients $C^{\varepsilon}_{Q_{\vec{P}},\vec{n}_{1},\vec{n}_{2},\vec{n}_{3}}:=C^{\varepsilon,\widetilde{Q'},Q''}_{\vec{n}_{1},\vec{n}_{2},\vec{n}_{3}}$, the \emph{tri-tiles} $\widetilde{P'}:=(\widetilde{P'_{1}},\widetilde{P'_{2}},\widetilde{P'_{3}})$ and $P'':=(P''_{1},P''_{2},P''_{3})$, the \emph{tiles} $\widetilde{P'_{i}}:=I_{\widetilde{P'_{i}}}\times\omega_{\widetilde{P'_{i}}}$ with $I_{\widetilde{P'_{i}}}:=\widetilde{I'}=2^{-k'}[l',l'+1]=:I_{\widetilde{P'}}$ and the frequency intervals $\omega_{\widetilde{P'_{i}}}:=\widetilde{Q'_{i}}$ for $i=1,2,3$, the \emph{tiles} $P''_{j}:=I_{P''_{j}}\times\omega_{P''_{j}}$ with $I_{P''_{j}}:=I''=2^{-k''}[l'',l''+1]=:I_{P''}$ and the frequency intervals $\omega_{P''_{j}}:=Q''_{j}$ for $j=1,2,3$, the frequency cubes $Q_{\widetilde{P'}}:=\omega_{\widetilde{P'_{1}}}\times\omega_{\widetilde{P'_{2}}}\times\omega_{\widetilde{P'_{3}}}$ and $Q_{P''}:=\omega_{P''_{1}}\times\omega_{P''_{2}}\times\omega_{P''_{3}}$, $\widetilde{\mathbb{P}'}$ denotes a collection of such tri-tiles $\widetilde{P'}$ and $\mathbb{P}''$ denotes a collection of such tri-tiles $P''$, the bi-tiles $\vec{P}_{1}$, $\vec{P}_{2}$ and $\vec{P}_{3}$ are defined by
\[\vec{P}_{1}:=(\widetilde{P'_{1}},P''_{1})=(2^{-k'}[l',l'+1]\times2^{k'}[-\frac{1}{2},\frac{1}{2}],2^{-k''}[l'',l''+1]\times Q''_{1}),\]
\[\vec{P}_{2}:=(\widetilde{P'_{2}},P''_{2})=(2^{-k'}[l',l'+1]\times2^{k'}[\frac{1}{24},\frac{25}{24}],2^{-k''}[l'',l''+1]\times Q''_{2}),\]
\[\vec{P}_{3}:=(\widetilde{P'_{3}},P''_{3})=(2^{-k'}[l',l'+1]\times2^{k'}[-\frac{25}{24},-\frac{1}{24}],2^{-k''}[l'',l''+1]\times Q''_{3});\]
the bi-parameter tri-tile $\vec{P}:=\widetilde{P'}\otimes P''=(\vec{P}_{1},\vec{P}_{2},\vec{P}_{3})$, the rectangles $I_{\vec{P}_{i}}:=I_{\widetilde{P'_{i}}}\times I_{P''_{i}}=I_{\widetilde{P'}}\times I_{P''}=:I_{\vec{P}}$ for $i=1,2,3$ and hence $|I_{\vec{P}}|=|I_{\widetilde{P'}}\times I_{P''}|=|I_{\vec{P}_{1}}|=|I_{\vec{P}_{2}}|=|I_{\vec{P}_{3}}|=2^{-k'}\cdot2^{-k''}$, the double frequency cube $Q_{\vec{P}}:=(Q_{\widetilde{P'}},Q_{P''})=(\omega_{\widetilde{P'_{1}}}\times\omega_{\widetilde{P'_{2}}}\times\omega_{\widetilde{P'_{3}}},\omega_{P''_{1}}\times\omega_{P''_{2}}\times\omega_{P''_{3}})$, $\vec{\mathbb{P}}:=\widetilde{\mathbb{P}'}\times\mathbb{P}''$ denotes a collection of such bi-parameter tri-tiles $\vec{P}$; while the $L^{2}$-normalized wave packets $\Phi^{i,n'_{i},\nu'}_{\widetilde{P'_{i}}}$ associated with the Heisenberg boxes $\widetilde{P'_{i}}$ are defined by $\Phi^{i,n'_{i},\nu'}_{\widetilde{P'_{i}}}(x_{1}):=\check{\phi}^{n'_{i},\nu'}_{\widetilde{I'},\widetilde{Q'_{i}},i}(x_{1}):=2^{-\frac{k'}{2}}\overline{\check{\phi}^{n'_{i}}_{\widetilde{Q'_{i}},i}(2^{-k'}(l'+\nu')-x_{1})}$ for $i=1,2,3$, the $L^{2}$-normalized wave packets $\Phi^{i,n''_{i},\nu''}_{P''_{i}}$ associated with the Heisenberg boxes $P''_{i}$ are defined by $\Phi^{i,n''_{i},\nu''}_{P''_{i}}(x_{2}):=\check{\phi}^{n''_{i},\nu''}_{I'',Q''_{i},i}(x_{2}):=2^{-\frac{k''}{2}}\overline{\check{\phi}^{n''_{i}}_{Q''_{i},i}(2^{-k''}(l''+\nu'')-x_{2})}$ for $i=1,2,3$, the smooth bump functions $\Phi^{i,\vec{n}_{i},\nu}_{\vec{P}_{i}}:=\Phi^{i,n'_{i},\nu'}_{\widetilde{P'_{i}}}\otimes\Phi^{i,n''_{i},\nu''}_{P''_{i}}$ for $i=1,2,3$.

We have the following rapid decay estimates of the Fourier coefficients $C^{\varepsilon}_{Q_{\vec{P}},\vec{n}_{1},\vec{n}_{2},\vec{n}_{3}}$ with respect to the parameters $\vec{n}_{1},\vec{n}_{2},\vec{n}_{3}\in\mathbb{Z}^{2}$.
\begin{lem}\label{Fourier coefficient1}
The Fourier coefficients $C^{\varepsilon}_{Q_{\vec{P}},\vec{n}_{1},\vec{n}_{2},\vec{n}_{3}}$ satisfy estimates
\begin{equation}\label{FC-1}
  |C^{\varepsilon}_{Q_{\vec{P}},\vec{n}_{1},\vec{n}_{2},\vec{n}_{3}}|\lesssim\prod_{j=1}^{3}\frac{1}{(1+|\vec{n}_{j}|)^{M}}\cdot C^{\varepsilon}_{|I_{\widetilde{P'}}|}
\end{equation}
for any bi-parameter tri-tile $\vec{P}\in\vec{\mathbb{P}}$, where $M$ is sufficiently large and the sequence $C^{\varepsilon}_{k'}:=C^{\varepsilon}_{|I_{\widetilde{P'}}|}$ for $|I_{\widetilde{P'}}|=2^{-k'}$ ($k'\in\mathbb{Z}$) satisfies
\begin{equation}\label{FC-1'}
  \sum_{k'\in\mathbb{Z}}C^{\varepsilon}_{k'}\leq C_{\varepsilon}<+\infty
\end{equation}
and $C_{\varepsilon}\rightarrow+\infty$ as $\varepsilon\rightarrow0$.
\end{lem}
\begin{proof}
Let $\ell(Q_{\widetilde{P'}})=2^{k'}$ and $\ell(Q_{P''})=2^{k''}$ for $k',k''\in\mathbb{Z}$. For any $\varepsilon>0$, $\vec{n}_{1},\vec{n}_{2},\vec{n}_{3}\in\mathbb{Z}^{2}$ and $\vec{P}\in\vec{\mathbb{P}}$, we deduce from \eqref{2.18} and \eqref{2.20} that
\begin{eqnarray}\label{2.23}
 \nonumber C^{\varepsilon}_{Q_{\vec{P}},\vec{n}_{1},\vec{n}_{2},\vec{n}_{3}}&=&\int_{\mathbb{R}^{6}}
  m^{\varepsilon}_{Q_{\widetilde{P'}},Q_{P''}}((2^{k'}\xi^{1}_{1},2^{k''}\xi^{2}_{1}),(2^{k'}\xi^{1}_{2},2^{k''}\xi^{2}_{2}),
  (2^{k'}\xi^{1}_{3},2^{k''}\xi^{2}_{3}))\\
  && \quad\quad\quad\quad\quad\quad\quad\quad\quad\quad\quad\quad\quad\quad \times e^{-2\pi i(\vec{n}_{1}\cdot\xi_{1}+\vec{n}_{2}\cdot\xi_{2}+\vec{n}_{3}\cdot\xi_{3})}
  d\xi_{1}d\xi_{2}d\xi_{3},
\end{eqnarray}
where
\begin{eqnarray}\label{2.24}
  && m^{\varepsilon}_{Q_{\widetilde{P'}},Q_{P''}}((2^{k'}\xi^{1}_{1},2^{k''}\xi^{2}_{1}),(2^{k'}\xi^{1}_{2},2^{k''}\xi^{2}_{2}),
  (2^{k'}\xi^{1}_{3},2^{k''}\xi^{2}_{3})):=m^{\varepsilon}(2^{k'}\bar{\xi}_{1},2^{k''}\bar{\xi}_{2})\\
 \nonumber && \quad\quad\quad\quad\quad\quad\quad\quad\quad\quad\quad \times\widetilde{\phi}_{Q_{\widetilde{P'}}}(2^{k'}\xi^{1}_{1},2^{k'}\xi^{1}_{2},2^{k'}\xi^{1}_{3})\phi_{\omega_{P''_{1}}\times\omega_{P''_{2}}}(2^{k''}\bar{\xi}_{2})
 \widetilde{\phi}_{\omega_{P''_{3}},3}(2^{k''}\xi^{2}_{3}).
\end{eqnarray}
Observe that $supp \, (\widetilde{\phi}_{Q_{\widetilde{P'}}}(\xi^{1}_{1},\xi^{1}_{2},\xi^{1}_{3})\phi_{\omega_{P''_{1}}\times\omega_{P''_{2}}}(\bar{\xi}_{2})\widetilde{\phi}_{\omega_{P''_{3}},3}(\xi^{2}_{3}))\subseteq Q_{\widetilde{P'}}\times Q_{P''}$, we have that $supp \, (\widetilde{\phi}_{Q_{\widetilde{P'}}}(2^{k'}\xi^{1}_{1},2^{k'}\xi^{1}_{2},2^{k'}\xi^{1}_{3})\phi_{\omega_{P''_{1}}\times\omega_{P''_{2}}}(2^{k''}\bar{\xi}_{2})
\widetilde{\phi}_{\omega_{P''_{3}},3}(2^{k''}\xi^{2}_{3}))\subseteq Q^{0}_{\widetilde{P'}}\times Q^{0}_{P''}$, where cubes $Q^{0}_{\widetilde{P'}}$ and $Q^{0}_{P''}$ are defined by
\begin{equation}\label{2.25}
  Q^{0}_{\widetilde{P'}}=\omega^{0}_{\widetilde{P'_{1}}}\times\omega^{0}_{\widetilde{P'_{2}}}\times\omega^{0}_{\widetilde{P'_{3}}}
  :=\{(\xi^{1}_{1},\xi^{1}_{2},\xi^{1}_{3})\in\mathbb{R}^{3}: (2^{k'}\xi^{1}_{1},2^{k'}\xi^{1}_{2},2^{k'}\xi^{1}_{3})\in Q_{\widetilde{P'}}\},
\end{equation}
\begin{equation}\label{2.26}
  Q^{0}_{P''}=\omega^{0}_{P''_{1}}\times\omega^{0}_{P''_{2}}\times\omega^{0}_{P''_{3}}
  :=\{(\xi^{2}_{1},\xi^{2}_{2},\xi^{2}_{3})\in\mathbb{R}^{3}: (2^{k''}\xi^{2}_{1},2^{k''}\xi^{2}_{2},2^{k''}\xi^{2}_{3})\in Q_{P''}\}
\end{equation}
and satisfy $|Q^{0}_{\widetilde{P'}}|\simeq|Q^{0}_{P''}|\simeq1$. From the properties of the \emph{Whitney squares} we constructed above, one obtains that $dist(2^{k'}\bar{\xi}_{1},\Gamma_{1})\simeq2^{k'}$ for any $\bar{\xi}_{1}\in\omega^{0}_{\widetilde{P'_{1}}}\times\omega^{0}_{\widetilde{P'_{2}}}$ and $dist(2^{k''}\bar{\xi}_{2},\Gamma_{2})\simeq2^{k''}$ for any $\bar{\xi}_{2}\in\omega^{0}_{P''_{1}}\times\omega^{0}_{P''_{2}}$.

One can deduce from \eqref{2.23}, \eqref{2.24} and integrating by parts sufficiently many times that
\begin{eqnarray*}
&& |C^{\varepsilon}_{Q_{\vec{P}},\vec{n}_{1},\vec{n}_{2},\vec{n}_{3}}|\lesssim\prod_{j=1}^{3}\frac{1}{(1+|\vec{n}_{j}|)^{M}}\\
&& \times\int_{Q^{0}_{\widetilde{P'}}\times Q^{0}_{P''}}|\partial^{\alpha_{1}}_{\xi_{1}}\partial^{\alpha_{2}}_{\xi_{2}}\partial^{\alpha_{3}}_{\xi_{3}}
[m^{\varepsilon}_{Q_{\widetilde{P'}},Q_{P''}}((2^{k'}\xi^{1}_{1},2^{k''}\xi^{2}_{1}),(2^{k'}\xi^{1}_{2},2^{k''}\xi^{2}_{2}),
(2^{k'}\xi^{1}_{3},2^{k''}\xi^{2}_{3}))]|d\xi_{1}d\xi_{2}d\xi_{3}\\
&\lesssim& \prod_{j=1}^{3}\frac{1}{(1+|\vec{n}_{j}|)^{M}}\int_{\omega^{0}_{P''_{1}}\times\omega^{0}_{P''_{2}}}dist(2^{k''}\bar{\xi}_{2},\Gamma_{2})^{|\alpha''|}\\
&& \quad\quad\quad\quad\quad\quad\quad\quad\quad\quad\quad\quad\times\int_{\omega^{0}_{\widetilde{P'_{1}}}\times\omega^{0}_{\widetilde{P'_{2}}}}dist(2^{k'}\bar{\xi}_{1},\Gamma_{1})^{|\alpha'|}
|\partial^{\alpha'}_{\bar{\xi}_{1}}\partial^{\alpha''}_{\bar{\xi}_{2}}m^{\varepsilon}(2^{k'}\bar{\xi}_{1},2^{k''}\bar{\xi}_{2})|d\bar{\xi}_{1}d\bar{\xi}_{2}\\
&\lesssim& \prod_{j=1}^{3}\frac{1}{(1+|\vec{n}_{j}|)^{M}}\cdot\frac{1}{\ell(Q_{P''})^{2}}\int_{\omega_{P''_{1}}\times\omega_{P''_{2}}}
dist(\bar{\xi}_{2},\Gamma_{2})^{|\alpha''|}\\
&& \quad\quad\quad\times\int_{\omega_{\widetilde{P'_{1}}}\times\omega_{\widetilde{P'_{2}}}}dist(\bar{\xi}_{1},\Gamma_{1})^{|\alpha'|-2}
|\partial^{\alpha'}_{\bar{\xi}_{1}}\partial^{\alpha''}_{\bar{\xi}_{2}}m^{\varepsilon}(\bar{\xi}_{1},\bar{\xi}_{2})|d\bar{\xi}_{1}d\bar{\xi}_{2}
=:\prod_{j=1}^{3}\frac{1}{(1+|\vec{n}_{j}|)^{M}}\cdot C^{\varepsilon}_{|I_{\widetilde{P'}}|},
\end{eqnarray*}
where the multi-indices $\alpha_{i}:=(\alpha^{1}_{i},\alpha^{2}_{i})$ for $i=1,2,3$ and $|\alpha_{1}|=|\alpha_{2}|=|\alpha_{3}|=M$ are sufficiently large, the multi-indices $\alpha':=(\alpha'_{1},\alpha'_{2},\alpha'_{3})$, $\alpha'':=(\alpha''_{1},\alpha''_{2},\alpha''_{3})$ with $\alpha'_{i}\leq\alpha^{1}_{i}$ and $\alpha''_{j}\leq\alpha^{2}_{j}$ for $i,j=1,2,3$. This proves the estimates \eqref{FC-1}.

Moreover, for $|I_{\widetilde{P'}}|=2^{-k'}$, we define the sequence $C^{\varepsilon}_{k'}:=C^{\varepsilon}_{|I_{\widetilde{P'}}|}$ ($k'\in\mathbb{Z}$). From the estimates \eqref{1.10} for symbol $m^{\varepsilon}(\bar{\xi}_{1},\bar{\xi}_{2})$, we get that
\begin{equation}\label{2.27}
  dist(\bar{\xi_{2}},\Gamma_{2})^{|\alpha''|}\cdot\int_{\mathbb{R}^{2}}dist(\bar{\xi_{1}},\Gamma_{1})^{|\alpha'|-2}
  |\partial^{\alpha'}_{\bar{\xi}_{1}}\partial^{\alpha''}_{\bar{\xi}_{2}}m^{\varepsilon}(\bar{\xi})|d\bar{\xi_{1}}\leq B(\varepsilon)<+\infty,
\end{equation}
and hence we can deduce the following summable property for the sequence $\{C^{\varepsilon}_{k'}\}_{k'\in\mathbb{Z}}$:
\begin{eqnarray}\label{2.28}
  \nonumber  \sum_{k'\in\mathbb{Z}}C^{\varepsilon}_{k'}&\lesssim&\frac{1}{\ell(Q_{P''})^{2}}\int_{\omega_{P''_{1}}\times\omega_{P''_{2}}}
   dist(\bar{\xi}_{2},\Gamma_{2})^{|\alpha''|}\\
&&\quad\quad\quad\quad\times
\int_{\cup_{\widetilde{P'}\in\widetilde{\mathbb{P}'}}(\omega_{\widetilde{P'_{1}}}\times\omega_{\widetilde{P'_{2}}})_{\widetilde{P'}}}
dist(\bar{\xi}_{1},\Gamma_{1})^{|\alpha'|-2}|\partial^{\alpha'}_{\bar{\xi}_{1}}\partial^{\alpha''}_{\bar{\xi}_{2}}m^{\varepsilon}(\bar{\xi}_{1},\bar{\xi}_{2})|d\bar{\xi}_{1}d\bar{\xi}_{2}\\
 \nonumber &\lesssim& \frac{1}{\ell(Q_{P''})^{2}}\int_{\omega_{P''_{1}}\times\omega_{P''_{2}}}B(\varepsilon)d\bar{\xi}_{2}\leq C_{\varepsilon}<+\infty
\end{eqnarray}
and $C_{\varepsilon}\sim B(\varepsilon)\rightarrow+\infty$ as $\varepsilon\rightarrow0$, this ends the proof of the summable estimate \eqref{FC-1'}.
\end{proof}

Observe that the rapid decay with respect to the parameters $\vec{n}_{1},\vec{n}_{2},\vec{n}_{3}\in\mathbb{Z}^{2}$ in \eqref{FC-1} is acceptable for summation, all the functions $\Phi^{i,n'_{i},\nu'}_{\widetilde{P'_{i}}}$ ($i=1,2,3$) are $L^{2}$ normalized and are wave packets associated with the Heisenberg boxes $\widetilde{P'_{i}}$ uniformly with respect to the parameters $n'_{i}$ and all the functions $\Phi^{j,n''_{j},\nu''}_{P''_{j}}$ ($j=1,2,3$) are $L^{2}$ normalized and are wave packets associated with the Heisenberg boxes $P''_{j}$ uniformly with respect to the parameters $n''_{j}$, therefore we only need to consider from now on the part of the trilinear form $\Lambda^{(2)}_{m^{\varepsilon},(lh,\mathbb{I})}(f_{1},f_{2},f_{3})$ defined in \eqref{2.22} corresponding to $\vec{n}_{1}=\vec{n}_{2}=\vec{n}_{3}=\vec{0}$:
\begin{equation}\label{2.29}
  \dot{\Lambda}^{(2)}_{m^{\varepsilon},(lh,\mathbb{I})}(f_{1},f_{2},f_{3}):=\int_{0}^{1}\int_{0}^{1}\sum_{\vec{P}\in\vec{\mathbb{P}}}
 \frac{C^{\varepsilon}_{Q_{\vec{P}}}}{|I_{\vec{P}}|^{\frac{1}{2}}}
 \langle f_{1},\Phi^{1,\nu}_{\vec{P}_{1}}\rangle\langle f_{2},\Phi^{2,\nu}_{\vec{P}_{2}}\rangle
 \langle f_{3},\Phi^{3,\nu}_{\vec{P}_{3}}\rangle d\nu,
\end{equation}
where $C^{\varepsilon}_{Q_{\vec{P}}}:=C^{\varepsilon}_{Q_{\vec{P}},\vec{0},\vec{0},\vec{0}}$, parameters $\nu=(\nu',\nu'')$ and $\Phi^{i,\nu}_{\vec{P}_{i}}:=\Phi^{i,\vec{0},\nu}_{\vec{P}_{i}}$ for $i=1,2,3$.

\begin{rem}\label{properties}
We should point out two important properties of the tri-tiles in $\mathbb{P}''$ (see \cite{MS,MTT2}). First, if one knows the position of $P''_{1}$, $P''_{2}$ or $P''_{3}$, then one knows precisely the positions of the other two as well. Second, if one assumes for instance that all the frequency intervals $\omega_{P''_{1}}$ of the $P''_{1}$ tiles intersect each other (say, they are non-lacunary about a fixed frequency $\xi_{0}$), then the frequency intervals $\omega_{P''_{2}}$ of the corresponding $P''_{2}$ tiles are disjoint and lacunary around $\xi_{0}$ (that is, $dist(\xi_{0},\omega_{P''_{2}})\simeq|\omega_{P''_{2}}|$ for all $P''\in\mathbb{P}''$). A similar conclusion can also be drawn for the $P''_{3}$ tiles modulo certain translations. This observation motivates the introduction of \emph{trees} in Definition \ref{trees}.
\end{rem}

We review the following definitions from \cite{MTT2}.
\begin{defn}\label{tile-sparse}
A collection $\mathbb{P}$ of tri-tiles is called sparse, if all tri-tiles in $\mathbb{P}$ have the same shift and the sets $\{Q_{P}: P\in\mathbb{P}\}$ and $\{I_{P}: P\in\mathbb{P}\}$ are sparse.
\end{defn}
\begin{defn}\label{tile relations}
Let $P$ and $P'$ be tiles. Then\\
(i) we write $P'<P$ if $I_{P'}\subsetneq I_{P}$ and $\omega_{P}\subseteq3\omega_{P'}$;\\
(ii) we write $P'\leq P$ if $P'<P$ or $P'=P$;\\
(iii) we write $P'\lesssim P$ if $I_{P'}\subseteq I_{P}$ and $\omega_{P}\subseteq10^{6}\omega_{P'}$;\\
(iv) we write $P'\lesssim'P$ if $P'\lesssim P$ but $P'\nleq P$.
\end{defn}
\begin{defn}\label{rank 1}
A collection $\mathbb{P}$ of tri-tiles is said to have \emph{rank} $1$ if the following properties are satisfied for all $P,P'\in\mathbb{P}$.\\
(i) If $P\neq P'$, then $P_{j}\neq P'_{j}$ for $1\leq j\leq3$.\\
(ii) If $\omega_{P_{j}}=\omega_{P'_{j}}$ for some $j$, then $\omega_{P_{j}}=\omega_{P'_{j}}$ for all $1\leq j\leq3$.\\
(iii) If $P'_{j}\leq P_{j}$ for some $j$, then $P'_{j}\lesssim P_{j}$ for all $1\leq j\leq3$.\\
(iv) If in addition to $P'_{j}\leq P_{j}$ one also assumes that $10^{8}|I_{P'}|\leq|I_{P}|$, then one has $P'_{i}\lesssim'P_{i}$ for every $i\neq j$.
\end{defn}

It is not difficult to observe that the collection of tri-tiles $\mathbb{P}''$ can be written as a finite union of sparse collections of rank $1$, thus we may assume further that $\mathbb{P}''$ is a sparse collection of rank $1$ from now on.

The bilinear operator corresponding to the trilinear form $\dot{\Lambda}^{(2)}_{m^{\varepsilon},(lh,\mathbb{I})}(f_{1},f_{2},f_{3})$ can be written as
\begin{equation}\label{2.30}
  \dot{\Pi}^{\varepsilon}_{\vec{\mathbb{P}}}(f_{1},f_{2})(x)=\int_{0}^{1}\int_{0}^{1}\sum_{\vec{P}\in\vec{\mathbb{P}}}
 \frac{C^{\varepsilon}_{Q_{\vec{P}}}}{|I_{\vec{P}}|^{\frac{1}{2}}}
 \langle f_{1},\Phi^{1,\nu}_{\vec{P}_{1}}\rangle\langle f_{2},\Phi^{2,\nu}_{\vec{P}_{2}}\rangle\Phi^{3,\nu}_{\vec{P}_{3}}(x)d\nu.
\end{equation}
Since $\dot{\Pi}^{\varepsilon}_{\vec{\mathbb{P}}}(f_{1},f_{2})$ is an average of some discrete bilinear model operators depending on the parameters $\nu=(\nu_{1},\nu_{2})\in[0,1]^{2}$, it is enough to prove the H\"{o}lder-type $L^{p}$ estimates for each of them, uniformly with respect to parameters $\nu=(\nu_{1},\nu_{2})$. From now on, we will do this in the particular case when the parameters $\nu=(\nu_{1},\nu_{2})=(0,0)$, but the same argument works in general. By Fatou's lemma, we can also restrict the summation in the definition \eqref{2.30} of $\dot{\Pi}^{\varepsilon}_{\vec{\mathbb{P}}}(f_{1},f_{2})$ on collection $\vec{\mathbb{P}}=\widetilde{\mathbb{P}'}\times\mathbb{P}''$ with arbitrary finite collections $\widetilde{\mathbb{P}'}$ and $\mathbb{P}''$ of tri-tiles, and prove the estimates are unform with respect to different choices of the set $\vec{\mathbb{P}}$.

Therefore, one can reduce the bilinear operator $\dot{\Pi}^{\varepsilon}_{\vec{\mathbb{P}}}$ further to the discrete bilinear model operator $\Pi^{\varepsilon}_{\vec{\mathbb{P}}}$ defined by
\begin{equation}\label{model1}
  \Pi^{\varepsilon}_{\vec{\mathbb{P}}}(f_{1},f_{2})(x):=\sum_{\vec{P}\in\vec{\mathbb{P}}}\frac{C^{\varepsilon}_{Q_{\vec{P}}}}{|I_{\vec{P}}|^{\frac{1}{2}}}
  \langle f_{1},\Phi^{1}_{\vec{P}_{1}}\rangle\langle f_{2},\Phi^{2}_{\vec{P}_{2}}\rangle\Phi^{3}_{\vec{P}_{3}}(x),
\end{equation}
where $\Phi^{j}_{\vec{P}_{j}}:=\Phi^{j,(0,0)}_{\vec{P}_{j}}$ for $j=1,2,3$ respectively, $\vec{\mathbb{P}}=\widetilde{\mathbb{P}'}\times\mathbb{P}''$ with arbitrary finite collection $\widetilde{\mathbb{P}'}$ of tri-tiles and arbitrary finite sparse collection $\mathbb{P}''$ of rank $1$. As have discussed above, we now reach a conclusion that the proof of Theorem \ref{main1} can be reduced to proving the following $L^{p}$ estimates for discrete bilinear model operators $\Pi^{\varepsilon}_{\vec{\mathbb{P}}}$.

\begin{prop}\label{equivalent1}
If the finite set $\vec{\mathbb{P}}$ is chosen arbitrarily as above, then the operator $\Pi^{\varepsilon}_{\vec{\mathbb{P}}}$ given by \eqref{model1} maps $L^{p_{1}}(\mathbb{R}^{2})\times L^{p_{2}}(\mathbb{R}^{2})\rightarrow L^{p}(\mathbb{R}^{2})$ boundedly for any $1<p_{1},p_{2}\leq\infty$ satisfying $\frac{1}{p}=\frac{1}{p_{1}}+\frac{1}{p_{2}}$ and $\frac{2}{3}<p<\infty$. Moreover, the implicit constants in the bounds depend only on $\varepsilon$, $p_{1}$, $p_{2}$, $p$ and are independent of the particular choice of finite collection $\vec{\mathbb{P}}$.
\end{prop}

\subsubsection{Discretization for bilinear, bi-parameter operators $T^{(2)}_{\widetilde{m}^{\varepsilon}}$}
We will proceed the discretization procedure as follows. First, we need to decompose the symbol $\widetilde{m}^{\varepsilon}(\xi)$ in a natural way. To this end, for both the spatial variables $x_{i}$ ($i=1,2$), we decompose the regions $\{\bar{\xi}_{i}=(\xi^{i}_{1},\xi^{i}_{2})\in\mathbb{R}^{2}: \xi^{i}_{1}\neq\xi^{i}_{2}\}$ by using \emph{Whitney squares} with respect to the singularity lines $\Gamma_{i}=\{\xi^{i}_{1}=\xi^{i}_{2}\}$ ($i=1,2$) respectively. Since the \emph{Whitney dyadic square} decomposition for the $x_{2}$ direction has already been described in \eqref{2.11}, \eqref{2.12}, \eqref{2.13} and \eqref{2.14} in sub-subsection 2.1.1, we only need to discuss the \emph{Whitney} decomposition with respect to the singularity line $\Gamma_{1}$ in $x_{1}$ direction.

To be specific, we consider the collection $\mathbb{Q}'$ of all shifted dyadic squares $Q'=Q'_{1}\times Q'_{2}$ satisfying
\begin{equation}\label{2.31}
  Q'\subseteq\{(\xi^{1}_{1},\xi^{1}_{2})\in\mathbb{R}^{2}: \xi^{1}_{1}\neq\xi^{1}_{2}\}, \,\,\,\,\,\,\,\,\,\, dist(Q',\Gamma_{1})\simeq10^{4}diam(Q').
\end{equation}
We can split the collection $\mathbb{Q'}$ into two disjoint sub-collections, that is, define
\begin{equation}\label{2.32}
  \mathbb{Q}'_{\mathbb{I}}:=\{Q'\in\mathbb{Q}': Q'\subseteq\{\xi^{1}_{1}<\xi^{1}_{2}\}\}, \,\,\,\,\,\,\,\,\,
  \mathbb{Q}'_{\mathbb{II}}:=\{Q'\in\mathbb{Q}': Q'\subseteq\{\xi^{1}_{1}>\xi^{1}_{2}\}\}.
\end{equation}
Since the set of squares $\{\frac{7}{10}Q': Q'\in\mathbb{Q}'\}$ also forms a finitely overlapping cover of the region $\{\xi^{1}_{1}\neq\xi^{1}_{2}\}$, we can apply a standard partition of unity and write the symbol $\chi_{\{\xi^{1}_{1}\neq\xi^{1}_{2}\}}$ as
\begin{equation}\label{2.33}
  \chi_{\{\xi^{1}_{1}\neq\xi^{1}_{2}\}}=\sum_{Q'\in\mathbb{Q}'}\phi_{Q'}(\xi^{1}_{1},\xi^{1}_{2})=
  \{\sum_{Q'\in\mathbb{Q}'_{\mathbb{I}}}+\sum_{Q'\in\mathbb{Q}'_{\mathbb{II}}}\}\phi_{Q'}(\xi^{1}_{1},\xi^{1}_{2})
  =\chi_{\{\xi^{1}_{1}<\xi^{1}_{2}\}}+\chi_{\{\xi^{1}_{1}>\xi^{1}_{2}\}},
\end{equation}
where each $\phi_{Q'}$ is a smooth bump function adapted to $Q'$ and supported in $\frac{8}{10}Q'$.

Notice that by splitting the symbol $\widetilde{m}^{\varepsilon}(\xi)$, we can decompose the operator $T^{(2)}_{\widetilde{m}^{\varepsilon}}$ correspondingly into a finite sum of several parts and we only need to discuss in detail arbitrary one of them. From the decompositions \eqref{2.13} and \eqref{2.33}, we obtain that
\begin{eqnarray}\label{2.34}
 \nonumber \widetilde{m}^{\varepsilon}(\bar{\xi}_{1},\bar{\xi}_{2})&=&\{\sum_{\substack{Q'\in\mathbb{Q}'_{\mathbb{I}},\\Q''\in\mathbb{Q}''_{\mathbb{I}}}}
  +\sum_{\substack{Q'\in\mathbb{Q}'_{\mathbb{I}},\\Q''\in\mathbb{Q}''_{\mathbb{II}}}}+\sum_{\substack{Q'\in\mathbb{Q}'_{\mathbb{II}},\\Q''\in\mathbb{Q}''_{\mathbb{I}}}}
  +\sum_{\substack{Q'\in\mathbb{Q}'_{\mathbb{II}},\\Q''\in\mathbb{Q}''_{\mathbb{II}}}}\}\phi_{Q'}(\xi^{1}_{1},\xi^{1}_{2})\phi_{Q''}(\xi^{2}_{1},\xi^{2}_{2})
  \cdot\widetilde{m}^{\varepsilon}(\bar{\xi}_{1},\bar{\xi}_{2})\\
  &=:&\widetilde{m}^{\varepsilon}_{\mathbb{I},\mathbb{I}}(\bar{\xi}_{1},\bar{\xi}_{2})+\widetilde{m}^{\varepsilon}_{\mathbb{I},\mathbb{II}}(\bar{\xi}_{1},\bar{\xi}_{2})
  +\widetilde{m}^{\varepsilon}_{\mathbb{II},\mathbb{I}}(\bar{\xi}_{1},\bar{\xi}_{2})+\widetilde{m}^{\varepsilon}_{\mathbb{II},\mathbb{II}}(\bar{\xi}_{1},\bar{\xi}_{2}).
\end{eqnarray}
One can easily observe that we only need to discuss in detail one term in the decomposition \eqref{2.34}, since the other term can be treated in the same way. Without loss of generality, we will only consider the third term in the decomposition \eqref{2.34}, which can be written as
\begin{equation}\label{2.35}
  \widetilde{m}^{\varepsilon}_{\mathbb{II},\mathbb{I}}(\bar{\xi}_{1},\bar{\xi}_{2}):=\sum_{Q'\in\mathbb{Q}'_{\mathbb{II}},Q''\in\mathbb{Q}''_{\mathbb{I}}}
  \widetilde{m}^{\varepsilon}(\bar{\xi}_{1},\bar{\xi}_{2})\phi_{Q'}(\xi^{1}_{1},\xi^{1}_{2})\phi_{Q''}(\xi^{2}_{1},\xi^{2}_{2}).
\end{equation}

In a word, we only need to consider the bilinear operator $T^{(2)}_{\widetilde{m}^{\varepsilon}_{\mathbb{II},\mathbb{I}}}$ given by
\begin{equation}\label{2.36}
  T^{(2)}_{\widetilde{m}^{\varepsilon}_{\mathbb{II},\mathbb{I}}}(f_{1},f_{2})(x):=\sum_{Q'\in\mathbb{Q}'_{\mathbb{II}}, \, Q''\in\mathbb{Q}''_{\mathbb{I}}}
  \int_{\mathbb{R}^{4}}\widetilde{m}^{\varepsilon}(\xi)\phi_{Q'}(\bar{\xi}_{1})\phi_{Q''}(\bar{\xi}_{2})\widehat{f_{1}}(\xi_{1})\widehat{f_{2}}(\xi_{2})
  e^{2\pi ix\cdot(\xi_{1}+\xi_{2})}d\xi
\end{equation}
from now on, and the proof of Theorem \ref{main2} can be reduced to proving the following $L^{p}$ estimates for $T^{(2)}_{\widetilde{m}^{\varepsilon}_{\mathbb{II},\mathbb{I}}}$:
\begin{equation}\label{2.37}
  \|T^{(2)}_{\widetilde{m}^{\varepsilon}_{\mathbb{II},\mathbb{I}}}(f_{1},f_{2})\|_{L^{p}(\mathbb{R}^{2})}\lesssim_{\varepsilon,p,p_{1},p_{2}}
  \|f_{1}\|_{L^{p_{1}}(\mathbb{R}^{2})}\cdot\|f_{2}\|_{L^{p_{2}}(\mathbb{R}^{2})},
\end{equation}
as long as $1<p_{1}, \, p_{2}\leq\infty$ and $0<\frac{1}{p}=\frac{1}{p_{1}}+\frac{1}{p_{2}}<\frac{3}{2}$.

Observe that there exist bump functions $\phi_{Q'_{i},i}$ ($i=1,2$) adapted to the shifted dyadic interval $Q'_{i}$ such that $supp \, \phi_{Q'_{i},i}\subseteq\frac{9}{10}Q'_{i}$ and $\phi_{Q'_{i},i}\equiv1$ on $\frac{8}{10}Q'_{i}$ ($i=1,2$) respectively, and $supp \, \phi_{Q'}\subseteq\frac{8}{10}Q'$, thus one has $\phi_{Q'_{1},1}\cdot\phi_{Q'_{2},2}\equiv1$ on $supp \, \phi_{Q'}$. Since $\xi^{1}_{1}\in supp \, \phi_{Q'_{1},1}\subseteq\frac{9}{10}Q'_{1}$ and $\xi^{1}_{2}\in supp \, \phi_{Q'_{2},2}\subseteq\frac{9}{10}Q'_{2}$, it follows that $-\xi^{1}_{1}-\xi^{1}_{2}\in-\frac{9}{10}Q'_{1}-\frac{9}{10}Q'_{2}$, and as a consequence, one can find a shifted dyadic interval $Q'_{3}$ with the property that $-\frac{9}{10}Q'_{1}-\frac{9}{10}Q'_{2}\subseteq\frac{7}{10}Q'_{3}$ and also satisfying $|Q'_{1}|=|Q'_{2}|\simeq|Q'_{3}|$. In particular, there exists bump function $\phi_{Q'_{3},3}$ adapted to $Q'_{3}$ and supported in $\frac{9}{10}Q'_{3}$ such that $\phi_{Q'_{3},3}\equiv1$ on $-\frac{9}{10}Q'_{1}-\frac{9}{10}Q'_{2}$. Recall that the smooth functions $\phi_{Q''_{j},j}$ ($j=1,2,3$) and shifted dyadic intervals $Q''_{3}$ have already been defined in sub-subsection 2.1.1.

We denote by $\mathbf{Q}'$ the collection of all shifted dyadic quasi-cubes $Q':=Q'_{1}\times Q'_{2}\times Q'_{3}$ with $Q'_{1}\times Q'_{2}\in\mathbb{Q}'_{\mathbb{II}}$ and $Q'_{3}$ be defined as above, and denote by $\mathbf{Q}''$ the collection of all shifted dyadic quasi-cubes $Q'':=Q''_{1}\times Q''_{2}\times Q''_{3}$ with $Q''_{1}\times Q''_{2}\in\mathbb{Q}''_{\mathbb{I}}$ and $Q''_{3}$ be defined in sub-subsection 2.1.1.

In fact, it is not difficult to see that the collections $\mathbf{Q}'$ and $\mathbf{Q}''$ can be split into a sum of finitely many \emph{sparse} collection of shifted dyadic quasi-cubes. Therefore, we can assume from now on that the collections $\mathbf{Q}'$ and $\mathbf{Q}''$ is \emph{sparse}.

Assuming this we then observe that, for any $Q'$ in such a sparse collection $\mathbf{Q}'$, there exists a unique shifted dyadic cube $\widetilde{Q'}$ in $\mathbb{R}^{3}$ such that $Q'\subseteq\frac{7}{10}\widetilde{Q'}$ and with property that $diam(Q')\simeq diam(\widetilde{Q'})$. This allows us in particular to assume further that $\mathbf{Q}'$ is a sparse collection of shifted dyadic cubes (that is, $|Q'_{1}|=|Q'_{2}|=|Q'_{3}|=\ell(Q')$). Similarly, we can also assume that $\mathbf{Q}''$ is a sparse collection of shifted dyadic cubes.

Now consider the trilinear form $\Lambda^{(2)}_{\widetilde{m}^{\varepsilon}_{\mathbb{II},\mathbb{I}}}(f_{1},f_{2},f_{3})$ associated to $T^{(2)}_{\widetilde{m}^{\varepsilon}_{\mathbb{II},\mathbb{I}}}(f_{1},f_{2})$, which can be written as
\begin{eqnarray}\label{2.38}
   &&\Lambda^{(2)}_{\widetilde{m}^{\varepsilon}_{\mathbb{II},\mathbb{I}}}(f_{1},f_{2},f_{3})
   :=\int_{\mathbb{R}^{2}}T^{(2)}_{\widetilde{m}^{\varepsilon}_{\mathbb{II},\mathbb{I}}}(f_{1},f_{2})(x)f_{3}(x)dx\\
 \nonumber &=&\sum_{Q'\in\mathbf{Q}',Q''\in\mathbf{Q}''}\int_{\xi_{1}+\xi_{2}+\xi_{3}=0}
 \widetilde{m}^{\varepsilon}_{Q',Q''}(\xi_{1},\xi_{2},\xi_{3})(f_{1}\ast(\check{\phi}_{Q'_{1},1}\otimes\check{\phi}_{Q''_{1},1}))^{\wedge}(\xi_{1})\\
 \nonumber &&\quad\quad\quad\quad\quad\quad
 \times(f_{2}\ast(\check{\phi}_{Q'_{2},2}\otimes\check{\phi}_{Q''_{2},2}))^{\wedge}(\xi_{2})
 (f_{3}\ast(\check{\phi}_{Q'_{3},3}\otimes\check{\phi}_{Q''_{3},3}))^{\wedge}(\xi_{3})d\xi_{1}d\xi_{2}d\xi_{3},
\end{eqnarray}
where $\xi_{i}=(\xi^{1}_{i},\xi^{2}_{i})$ for $i=1,2,3$, while
\begin{equation}\label{2.39}
  \widetilde{m}^{\varepsilon}_{Q',Q''}(\xi_{1},\xi_{2},\xi_{3}):=\widetilde{m}^{\varepsilon}(\xi_{1},\xi_{2})\cdot
  ((\phi_{Q'_{1}\times Q'_{2}}\cdot\widetilde{\phi}_{Q'_{3},3})\otimes(\phi_{Q''_{1}\times Q''_{2}}\cdot\widetilde{\phi}_{Q''_{3},3}))(\xi_{1},\xi_{2},\xi_{3}),
\end{equation}
where $\widetilde{\phi}_{Q'_{3},3}$ is an appropriate smooth function of variable $\xi^{1}_{3}$ supported on a slightly larger interval (with a constant magnification independent of $\ell(Q')$) than $supp \, \phi_{Q'_{3},3}$, which equals $1$ on $supp \, \phi_{Q'_{3},3}$, and $\widetilde{\phi}_{Q''_{3},3}$ is an appropriate smooth function of variable $\xi^{2}_{3}$ supported on a slightly larger interval (with a constant magnification independent of $\ell(Q'')$) than $supp \, \phi_{Q''_{3},3}$, which equals $1$ on $supp \, \phi_{Q''_{3},3}$. We can decompose $\widetilde{m}^{\varepsilon}_{Q',Q''}(\xi_{1},\xi_{2},\xi_{3})$ as a Fourier series:
\begin{equation}\label{2.40}
  \widetilde{m}^{\varepsilon}_{Q',Q''}(\xi_{1},\xi_{2},\xi_{3})=\sum_{\vec{l}_{1},\vec{l}_{2},\vec{l}_{3}\in\mathbb{Z}^{2}}
  \widetilde{C}^{\varepsilon,Q',Q''}_{\vec{l}_{1},\vec{l}_{2},\vec{l}_{3}}
  e^{2\pi i(l^{'}_{1},l^{'}_{2},l^{'}_{3})\cdot(\xi^{1}_{1},\xi^{1}_{2},\xi^{1}_{3})/\ell(Q')}
  e^{2\pi i(l^{''}_{1},l^{''}_{2},l^{''}_{3})\cdot(\xi^{2}_{1},\xi^{2}_{2},\xi^{2}_{3})/\ell(Q'')},
\end{equation}
where the Fourier coefficients $C^{\varepsilon,Q',Q''}_{\vec{l}_{1},\vec{l}_{2},\vec{l}_{3}}$ are given by
\begin{eqnarray}\label{2.41}
 \nonumber \widetilde{C}^{\varepsilon,Q',Q''}_{\vec{l}_{1},\vec{l}_{2},\vec{l}_{3}}&=&\int_{\mathbb{R}^{6}}
  \widetilde{m}^{\varepsilon}_{Q',Q''}((\ell(Q')\xi^{1}_{1},\ell(Q'')\xi^{2}_{1}),(\ell(Q')\xi^{1}_{2},\ell(Q'')\xi^{2}_{2}),
  (\ell(Q')\xi^{1}_{3},\ell(Q'')\xi^{2}_{3}))\\
  && \quad\quad\quad\quad\quad\quad\quad\quad\quad\quad\quad\quad\quad\quad\quad\quad\quad \times e^{-2\pi i(\vec{l}_{1}\cdot\xi_{1}+\vec{l}_{2}\cdot\xi_{2}+\vec{l}_{3}\cdot\xi_{3})}d\xi_{1}d\xi_{2}d\xi_{3}.
\end{eqnarray}
Then, by a straightforward calculation, we can rewrite \eqref{2.38} as
\begin{eqnarray}\label{2.42}
\nonumber &&\Lambda^{(2)}_{\widetilde{m}^{\varepsilon}_{\mathbb{II},\mathbb{I}}}(f_{1},f_{2},f_{3})=\sum_{Q'\in\mathbf{Q}',Q''\in\mathbf{Q}''}
\sum_{\vec{l}_{1},\vec{l}_{2},\vec{l}_{3}\in\mathbb{Z}^{2}}\widetilde{C}^{\varepsilon,Q',Q''}_{\vec{l}_{1},\vec{l}_{2},\vec{l}_{3}}\int_{\mathbb{R}^{2}}\\
\nonumber &&(f_{1}\ast(\check{\phi}_{Q'_{1},1}\otimes\check{\phi}_{Q''_{1},1}))(x-(\frac{l'_{1}}{\ell(Q')},\frac{l^{''}_{1}}{\ell(Q'')}))
  (f_{2}\ast(\check{\phi}_{Q'_{2},2}\otimes\check{\phi}_{Q''_{2},2}))(x-(\frac{l'_{2}}{\ell(Q')},\frac{l^{''}_{2}}{\ell(Q'')}))\\
&&\quad\quad\quad\quad\quad\quad\quad\quad\quad\quad\quad\quad\quad\quad\quad\quad
\times(f_{3}\ast(\check{\phi}_{Q'_{3},3}\otimes\check{\phi}_{Q''_{3},3}))(x-(\frac{l'_{3}}{\ell(Q')},\frac{l^{''}_{3}}{\ell(Q'')}))dx.
\end{eqnarray}

Now we define $\phi^{l'_{i}}_{Q'_{i},i}:=e^{2\pi il'_{i}\xi^{1}_{i}/\ell(Q')}\cdot\phi_{Q'_{i},i}$ and $\phi^{l''_{i}}_{Q''_{i},i}:=e^{2\pi il''_{i}\xi^{2}_{i}/\ell(Q'')}\cdot\phi_{Q''_{i},i}$ for $i=1,2,3$. Since any $Q'\in\mathbf{Q}'$ and $Q''\in\mathbf{Q}''$ are both shifted dyadic cubes, there exists integers $k',k''\in\mathbb{Z}$ such that $\ell(Q')=|Q'_{1}|=|Q'_{2}|=|Q'_{3}|=2^{k'}$ and $\ell(Q'')=|Q''_{1}|=|Q''_{2}|=|Q''_{3}|=2^{k''}$ respectively. By splitting the integral region $\mathbb{R}^{2}$ into the union of unit squares, the $L^{2}$-normalization procedure and simple calculations, we can rewrite \eqref{2.42} as
\begin{eqnarray}\label{2.43}
   &&\Lambda^{(2)}_{\widetilde{m}^{\varepsilon}_{\mathbb{II},\mathbb{I}}}(f_{1},f_{2},f_{3})\\
  \nonumber &=&\sum_{\vec{l}_{1},\vec{l}_{2},\vec{l}_{3}\in\mathbb{Z}^{2}}\sum_{Q'\in\mathbf{Q}',Q''\in\mathbf{Q}''}
  \int_{0}^{1}\int_{0}^{1}\sum_{\substack{I' \,\, dyadic, \\ |I'|=2^{-k'}}}\sum_{\substack{I'' \,\, dyadic, \\ |I''|=2^{-k''}}}\frac{\widetilde{C}^{\varepsilon,Q',Q''}_{\vec{l}_{1},\vec{l}_{2},\vec{l}_{3}}}{|I'|^{\frac{1}{2}}\times|I''|^{\frac{1}{2}}}
  \langle f_{1},\check{\phi}^{l'_{1},\lambda'}_{I',Q'_{1},1}\otimes\check{\phi}^{l''_{1},\lambda''}_{I'',Q''_{1},1}\rangle\\
  \nonumber &&\quad\quad\quad\quad\quad\quad\quad\quad\quad\quad\quad\quad\times
  \langle f_{2},\check{\phi}^{l'_{2},\lambda'}_{I',Q'_{2},2}\otimes\check{\phi}^{l''_{2},\lambda''}_{I'',Q''_{2},2}\rangle
  \langle f_{3},\check{\phi}^{l'_{3},\lambda'}_{I',Q'_{3},3}\otimes\check{\phi}^{l''_{3},\lambda''}_{I'',Q''_{3},3}\rangle d\lambda'd\lambda''\\
 \nonumber &=:&\sum_{\vec{l}_{1},\vec{l}_{2},\vec{l}_{3}\in\mathbb{Z}^{2}}\int_{0}^{1}\int_{0}^{1}\sum_{\vec{P}:=P'\otimes P''\in\vec{\mathbb{P}}}
 \frac{\widetilde{C}^{\varepsilon}_{Q_{\vec{P}},\vec{l}_{1},\vec{l}_{2},\vec{l}_{3}}}{|I_{\vec{P}}|^{\frac{1}{2}}}
 \langle f_{1},\Phi^{1,\vec{l}_{1},\lambda}_{\vec{P}_{1}}\rangle\langle f_{2},\Phi^{2,\vec{l}_{2},\lambda}_{\vec{P}_{2}}\rangle
 \langle f_{3},\Phi^{3,\vec{l}_{3},\lambda}_{\vec{P}_{3}}\rangle d\lambda,
\end{eqnarray}
where the Fourier coefficients $\widetilde{C}^{\varepsilon}_{Q_{\vec{P}},\vec{l}_{1},\vec{l}_{2},\vec{l}_{3}}:=\widetilde{C}^{\varepsilon,Q',Q''}_{\vec{l}_{1},\vec{l}_{2},\vec{l}_{3}}$, the \emph{tri-tiles} $P':=(P'_{1},P'_{2},P'_{3})$ and $P'':=(P''_{1},P''_{2},P''_{3})$, the \emph{tiles} $P'_{i}:=I_{P'_{i}}\times\omega_{P'_{i}}$ with $I_{P'_{i}}:=I'=2^{-k'}[n',n'+1]=:I_{P'}$ and the frequency intervals $\omega_{P'_{i}}:=Q'_{i}$ for $i=1,2,3$, the \emph{tiles} $P''_{j}:=I_{P''_{j}}\times\omega_{P''_{j}}$ with $I_{P''_{j}}:=I''=2^{-k''}[n'',n''+1]=:I_{P''}$ and the frequency intervals $\omega_{P''_{j}}:=Q''_{j}$ for $j=1,2,3$, the frequency cubes $Q_{P'}:=\omega_{P'_{1}}\times\omega_{P'_{2}}\times\omega_{P'_{3}}$ and $Q_{P''}:=\omega_{P''_{1}}\times\omega_{P''_{2}}\times\omega_{P''_{3}}$, $\mathbb{P}'$ denotes a collection of such tri-tiles $P'$ and $\mathbb{P}''$ denotes a collection of such tri-tiles $P''$, the bi-tiles $\vec{P}_{1}$, $\vec{P}_{2}$ and $\vec{P}_{3}$ are defined by
\[\vec{P}_{i}:=(P'_{i},P''_{i})=(2^{-k'}[n',n'+1]\times Q'_{i}, \, 2^{-k''}[n'',n''+1]\times Q''_{i})\]
for $i=1,2,3$; the bi-parameter tri-tile $\vec{P}:=P'\otimes P''=(\vec{P}_{1},\vec{P}_{2},\vec{P}_{3})$, the rectangles $I_{\vec{P}_{i}}:=I_{P'_{i}}\times I_{P''_{i}}=I_{P'}\times I_{P''}=:I_{\vec{P}}$ for $i=1,2,3$ and hence $|I_{\vec{P}}|=|I_{P'}\times I_{P''}|=|I_{\vec{P}_{1}}|=|I_{\vec{P}_{2}}|=|I_{\vec{P}_{3}}|=2^{-k'}\cdot2^{-k''}$, the double frequency cube $Q_{\vec{P}}:=(Q_{P'},Q_{P''})=(\omega_{P'_{1}}\times\omega_{P'_{2}}\times\omega_{P'_{3}},\omega_{P''_{1}}\times\omega_{P''_{2}}\times\omega_{P''_{3}})$, $\vec{\mathbb{P}}:=\mathbb{P}'\times\mathbb{P}''$ denotes a collection of such bi-parameter tri-tiles $\vec{P}$; while the $L^{2}$-normalized wave packets $\Phi^{i,l'_{i},\lambda'}_{P'_{i}}$ associated with the Heisenberg boxes $P'_{i}$ are defined by $\Phi^{i,l'_{i},\lambda'}_{P'_{i}}(x_{1}):=\check{\phi}^{l'_{i},\lambda'}_{I',Q'_{i},i}(x_{1}):=2^{-\frac{k'}{2}}\overline{\check{\phi}^{l'_{i}}_{Q'_{i},i}(2^{-k'}(n'+\lambda')-x_{1})}$ for $i=1,2,3$, the $L^{2}$-normalized wave packets $\Phi^{i,l''_{i},\lambda''}_{P''_{i}}$ associated with the Heisenberg boxes $P''_{i}$ are defined by $\Phi^{i,l''_{i},\lambda''}_{P''_{i}}(x_{2}):=\check{\phi}^{l''_{i},\lambda''}_{I'',Q''_{i},i}(x_{2}):=2^{-\frac{k''}{2}}\overline{\check{\phi}^{l''_{i}}_{Q''_{i},i}(2^{-k''}(n''+\lambda'')-x_{2})}$ for $i=1,2,3$, the smooth bump functions $\Phi^{i,\vec{l}_{i},\lambda}_{\vec{P}_{i}}:=\Phi^{i,l'_{i},\lambda'}_{P'_{i}}\otimes\Phi^{i,l''_{i},\lambda''}_{P''_{i}}$ for $i=1,2,3$.

We have the following rapid decay estimates of the Fourier coefficients $\widetilde{C}^{\varepsilon}_{Q_{\vec{P}},\vec{l}_{1},\vec{l}_{2},\vec{l}_{3}}$ with respect to the parameters $\vec{l}_{1},\vec{l}_{2},\vec{l}_{3}\in\mathbb{Z}^{2}$.
\begin{lem}\label{Fourier coefficient2}
The Fourier coefficients $\widetilde{C}^{\varepsilon}_{Q_{\vec{P}},\vec{l}_{1},\vec{l}_{2},\vec{l}_{3}}$ satisfy estimates
\begin{equation}\label{FC-2}
  |\widetilde{C}^{\varepsilon}_{Q_{\vec{P}},\vec{l}_{1},\vec{l}_{2},\vec{l}_{3}}|\lesssim \prod_{j=1}^{3}\frac{1}{(1+|\vec{l}_{j}|)^{M}}
  \cdot\langle\log_{2}\ell(Q_{P'})\rangle^{-(1+\varepsilon)}
\end{equation}
for any bi-parameter tri-tile $\vec{P}\in\vec{\mathbb{P}}$, where $M$ is sufficiently large.
\end{lem}
\begin{proof}
Let $\ell(Q_{P'})=2^{k'}$ and $\ell(Q_{P''})=2^{k''}$ for $k',k''\in\mathbb{Z}$. For any $\varepsilon>0$, $\vec{l}_{1},\vec{l}_{2},\vec{l}_{3}\in\mathbb{Z}^{2}$ and $\vec{P}\in\vec{\mathbb{P}}$, we deduce from \eqref{2.39} and \eqref{2.41} that
\begin{eqnarray}\label{2.44}
 \nonumber \widetilde{C}^{\varepsilon}_{Q_{\vec{P}},\vec{l}_{1},\vec{l}_{2},\vec{l}_{3}}&=&\int_{\mathbb{R}^{6}}
  \widetilde{m}^{\varepsilon}_{Q_{P'},Q_{P''}}((2^{k'}\xi^{1}_{1},2^{k''}\xi^{2}_{1}),(2^{k'}\xi^{1}_{2},2^{k''}\xi^{2}_{2}),
  (2^{k'}\xi^{1}_{3},2^{k''}\xi^{2}_{3}))\\
  && \quad\quad\quad\quad\quad\quad\quad\quad\quad\quad\quad\quad\quad\quad \times e^{-2\pi i(\vec{l}_{1}\cdot\xi_{1}+\vec{l}_{2}\cdot\xi_{2}+\vec{l}_{3}\cdot\xi_{3})}
  d\xi_{1}d\xi_{2}d\xi_{3},
\end{eqnarray}
where
\begin{eqnarray}\label{2.45}
  && \widetilde{m}^{\varepsilon}_{Q_{P'},Q_{P''}}((2^{k'}\xi^{1}_{1},2^{k''}\xi^{2}_{1}),(2^{k'}\xi^{1}_{2},2^{k''}\xi^{2}_{2}),
  (2^{k'}\xi^{1}_{3},2^{k''}\xi^{2}_{3})):=\widetilde{m}^{\varepsilon}(2^{k'}\bar{\xi}_{1},2^{k''}\bar{\xi}_{2})\\
 \nonumber && \quad\quad\quad\quad\quad\quad\quad\quad\quad \times\phi_{\omega_{P'_{1}}\times\omega_{P'_{2}}}(2^{k'}\bar{\xi}_{1})
 \widetilde{\phi}_{\omega_{P'_{3}},3}(2^{k'}\xi^{1}_{3})\phi_{\omega_{P''_{1}}\times\omega_{P''_{2}}}(2^{k''}\bar{\xi}_{2})
 \widetilde{\phi}_{\omega_{P''_{3}},3}(2^{k''}\xi^{2}_{3}).
\end{eqnarray}
Observe that $supp \, (\phi_{\omega_{P'_{1}}\times\omega_{P'_{2}}}(\bar{\xi}_{1})
 \widetilde{\phi}_{\omega_{P'_{3}},3}(\xi^{1}_{3})\phi_{\omega_{P''_{1}}\times\omega_{P''_{2}}}(\bar{\xi}_{2})
 \widetilde{\phi}_{\omega_{P''_{3}},3}(\xi^{2}_{3}))\subseteq Q_{P'}\times Q_{P''}$, we have that $supp \, (\phi_{\omega_{P'_{1}}\times\omega_{P'_{2}}}(2^{k'}\bar{\xi}_{1})
 \widetilde{\phi}_{\omega_{P'_{3}},3}(2^{k'}\xi^{1}_{3})\phi_{\omega_{P''_{1}}\times\omega_{P''_{2}}}(2^{k''}\bar{\xi}_{2})
 \widetilde{\phi}_{\omega_{P''_{3}},3}(2^{k''}\xi^{2}_{3}))\subseteq Q^{0}_{P'}\times Q^{0}_{P''}$, where cubes $Q^{0}_{P'}$ and $Q^{0}_{P''}$ are defined by
\begin{equation}\label{2.46}
  Q^{0}_{P'}=\omega^{0}_{P'_{1}}\times\omega^{0}_{P'_{2}}\times\omega^{0}_{P'_{3}}
  :=\{(\xi^{1}_{1},\xi^{1}_{2},\xi^{1}_{3})\in\mathbb{R}^{3}: (2^{k'}\xi^{1}_{1},2^{k'}\xi^{1}_{2},2^{k'}\xi^{1}_{3})\in Q_{P'}\},
\end{equation}
\begin{equation}\label{2.47}
  Q^{0}_{P''}=\omega^{0}_{P''_{1}}\times\omega^{0}_{P''_{2}}\times\omega^{0}_{P''_{3}}
  :=\{(\xi^{2}_{1},\xi^{2}_{2},\xi^{2}_{3})\in\mathbb{R}^{3}: (2^{k''}\xi^{2}_{1},2^{k''}\xi^{2}_{2},2^{k''}\xi^{2}_{3})\in Q_{P''}\}
\end{equation}
and satisfy $|Q^{0}_{P'}|\simeq|Q^{0}_{P''}|\simeq1$. From the properties of the \emph{Whitney squares} we constructed above, one obtains that $dist(2^{k'}\bar{\xi}_{1},\Gamma_{1})\simeq2^{k'}$ for any $\bar{\xi}_{1}\in\omega^{0}_{P'_{1}}\times\omega^{0}_{P'_{2}}$ and $dist(2^{k''}\bar{\xi}_{2},\Gamma_{2})\simeq2^{k''}$ for any $\bar{\xi}_{2}\in\omega^{0}_{P''_{1}}\times\omega^{0}_{P''_{2}}$.

By taking advantage of the estimates \eqref{1.10'} for symbol $\widetilde{m}^{\varepsilon}(\bar{\xi})$, one can deduce from \eqref{2.44}, \eqref{2.45} and integrating by parts sufficiently many times that
\begin{eqnarray*}
&& |\widetilde{C}^{\varepsilon}_{Q_{\vec{P}},\vec{l}_{1},\vec{l}_{2},\vec{l}_{3}}|\lesssim\prod_{j=1}^{3}\frac{1}{(1+|\vec{l}_{j}|)^{M}}\\
&& \times\int_{Q^{0}_{P'}\times Q^{0}_{P''}}|\partial^{\alpha_{1}}_{\xi_{1}}\partial^{\alpha_{2}}_{\xi_{2}}\partial^{\alpha_{3}}_{\xi_{3}}
[\widetilde{m}^{\varepsilon}_{Q_{P'},Q_{P''}}((2^{k'}\xi^{1}_{1},2^{k''}\xi^{2}_{1}),(2^{k'}\xi^{1}_{2},2^{k''}\xi^{2}_{2}),
(2^{k'}\xi^{1}_{3},2^{k''}\xi^{2}_{3}))]|d\xi_{1}d\xi_{2}d\xi_{3}\\
&\lesssim& \prod_{j=1}^{3}\frac{1}{(1+|\vec{l}_{j}|)^{M}}\int_{\omega^{0}_{P''_{1}}\times\omega^{0}_{P''_{2}}}dist(2^{k''}\bar{\xi}_{2},\Gamma_{2})^{|\alpha''|}\\
&& \quad\quad\quad\quad\quad\quad\quad\quad\quad\quad\quad\quad\times\int_{\omega^{0}_{P'_{1}}\times\omega^{0}_{P'_{2}}}dist(2^{k'}\bar{\xi}_{1},\Gamma_{1})^{|\alpha'|}
|\partial^{\alpha'}_{\bar{\xi}_{1}}\partial^{\alpha''}_{\bar{\xi}_{2}}\widetilde{m}^{\varepsilon}(2^{k'}\bar{\xi}_{1},2^{k''}\bar{\xi}_{2})|d\bar{\xi}_{1}d\bar{\xi}_{2}\\
&\lesssim& \prod_{j=1}^{3}\frac{1}{(1+|\vec{l}_{j}|)^{M}}\cdot2^{-2k'}2^{-2k''}\int_{\omega_{P''_{1}}\times\omega_{P''_{2}}}\int_{\omega_{P'_{1}}\times\omega_{P'_{2}}}\\
&& \quad\quad\quad\quad\quad\quad\quad\quad\quad\quad\quad\quad\quad\quad\, dist(\bar{\xi}_{2},\Gamma_{2})^{|\alpha''|}\cdot dist(\bar{\xi}_{1},\Gamma_{1})^{|\alpha'|}
|\partial^{\alpha'}_{\bar{\xi}_{1}}\partial^{\alpha''}_{\bar{\xi}_{2}}\widetilde{m}^{\varepsilon}(\bar{\xi}_{1},\bar{\xi}_{2})|d\bar{\xi}_{1}d\bar{\xi}_{2}\\
&\lesssim&\prod_{j=1}^{3}\frac{1}{(1+|\vec{l}_{j}|)^{M}}\cdot\langle\log_{2}\ell(Q_{P'})\rangle^{-(1+\varepsilon)},
\end{eqnarray*}
where the multi-indices $\alpha_{i}:=(\alpha^{1}_{i},\alpha^{2}_{i})$ for $i=1,2,3$ and $|\alpha_{1}|=|\alpha_{2}|=|\alpha_{3}|=M$ are sufficiently large, the multi-indices $\alpha':=(\alpha'_{1},\alpha'_{2},\alpha'_{3})$, $\alpha'':=(\alpha''_{1},\alpha''_{2},\alpha''_{3})$ with $\alpha'_{i}\leq\alpha^{1}_{i}$ and $\alpha''_{j}\leq\alpha^{2}_{j}$ for $i,j=1,2,3$. This ends our proof of estimates \eqref{FC-2}.
\end{proof}

Observe that the rapid decay with respect to the parameters $\vec{l}_{1},\vec{l}_{2},\vec{l}_{3}\in\mathbb{Z}^{2}$ in \eqref{FC-2} is acceptable for summation, all the functions $\Phi^{i,l'_{i},\lambda'}_{P'_{i}}$ ($i=1,2,3$) are $L^{2}$ normalized and are wave packets associated with the Heisenberg boxes $P'_{i}$ uniformly with respect to the parameters $l'_{i}$ and all the functions $\Phi^{j,l''_{j},\lambda''}_{P''_{j}}$ ($j=1,2,3$) are $L^{2}$ normalized and are wave packets associated with the Heisenberg boxes $P''_{j}$ uniformly with respect to the parameters $l''_{j}$, therefore we only need to consider from now on the part of the trilinear form $\Lambda^{(2)}_{\widetilde{m}^{\varepsilon}_{\mathbb{II},\mathbb{I}}}(f_{1},f_{2},f_{3})$ defined in \eqref{2.43} corresponding to $\vec{l}_{1}=\vec{l}_{2}=\vec{l}_{3}=\vec{0}$:
\begin{equation}\label{2.48}
  \dot{\Lambda}^{(2)}_{\widetilde{m}^{\varepsilon}_{\mathbb{II},\mathbb{I}}}(f_{1},f_{2},f_{3}):=\int_{0}^{1}\int_{0}^{1}\sum_{\vec{P}\in\vec{\mathbb{P}}}
 \frac{\widetilde{C}^{\varepsilon}_{Q_{\vec{P}}}}{|I_{\vec{P}}|^{\frac{1}{2}}}
 \langle f_{1},\Phi^{1,\lambda}_{\vec{P}_{1}}\rangle\langle f_{2},\Phi^{2,\lambda}_{\vec{P}_{2}}\rangle
 \langle f_{3},\Phi^{3,\lambda}_{\vec{P}_{3}}\rangle d\lambda,
\end{equation}
where $\widetilde{C}^{\varepsilon}_{Q_{\vec{P}}}:=\widetilde{C}^{\varepsilon}_{Q_{\vec{P}},\vec{0},\vec{0},\vec{0}}$, parameters $\lambda=(\lambda',\lambda'')$ and $\Phi^{i,\lambda}_{\vec{P}_{i}}:=\Phi^{i,\vec{0},\lambda}_{\vec{P}_{i}}$ for $i=1,2,3$.

The tri-tiles $P'=(P'_{1},P'_{2},P'_{3})$ in collection $\mathbb{P}'$ also satisfy the same properties (as $P''\in\mathbb{P}''$) described in Remark \ref{properties}. It is not difficult to observe that both the collections of tri-tiles $\mathbb{P}'$ and $\mathbb{P}''$ can be written as a finite union of sparse collections of rank $1$, thus we may assume further that $\mathbb{P}'$ and $\mathbb{P}''$ are sparse collection of rank $1$ from now on.

The bilinear operator corresponding to the trilinear form $\dot{\Lambda}^{(2)}_{\widetilde{m}^{\varepsilon}_{\mathbb{II},\mathbb{I}}}(f_{1},f_{2},f_{3})$ can be written as
\begin{equation}\label{2.49}
  \dot{\widetilde{\Pi}}^{\varepsilon}_{\vec{\mathbb{P}}}(f_{1},f_{2})(x)=\int_{0}^{1}\int_{0}^{1}\sum_{\vec{P}\in\vec{\mathbb{P}}}
 \frac{\widetilde{C}^{\varepsilon}_{Q_{\vec{P}}}}{|I_{\vec{P}}|^{\frac{1}{2}}}
 \langle f_{1},\Phi^{1,\lambda}_{\vec{P}_{1}}\rangle\langle f_{2},\Phi^{2,\lambda}_{\vec{P}_{2}}\rangle\Phi^{3,\lambda}_{\vec{P}_{3}}(x)d\lambda.
\end{equation}
Since $\dot{\widetilde{\Pi}}^{\varepsilon}_{\vec{\mathbb{P}}}(f_{1},f_{2})$ is an average of some discrete bilinear model operators depending on the parameters $\lambda=(\lambda_{1},\lambda_{2})\in[0,1]^{2}$, it is enough to prove the H\"{o}lder-type $L^{p}$ estimates for each of them, uniformly with respect to parameters $\lambda=(\lambda_{1},\lambda_{2})$. From now on, we will do this in the particular case when the parameters $\lambda=(\lambda_{1},\lambda_{2})=(0,0)$, but the same argument works in general. By Fatou's lemma, we can also restrict the summation in the definition \eqref{2.49} of $\dot{\widetilde{\Pi}}^{\varepsilon}_{\vec{\mathbb{P}}}(f_{1},f_{2})$ on collection $\vec{\mathbb{P}}=\mathbb{P}'\times\mathbb{P}''$ with arbitrary finite collections $\mathbb{P}'$ and $\mathbb{P}''$ of tri-tiles, and prove the estimates are unform with respect to different choices of the set $\vec{\mathbb{P}}$.

\begin{defn}
A finite collection $\vec{\mathbb{P}}=\mathbb{P}'\times\mathbb{P}''$ of bi-parameter tri-tiles is said to be sparse and rank $1$, if both the finite collections $\mathbb{P}'$ and $\mathbb{P}''$ are sparse and rank $1$.
\end{defn}

Therefore, one can reduce the bilinear operator $\dot{\widetilde{\Pi}}^{\varepsilon}_{\vec{\mathbb{P}}}$ further to the discrete bilinear model operator $\widetilde{\Pi}^{\varepsilon}_{\vec{\mathbb{P}}}$ defined by
\begin{equation}\label{model2}
  \widetilde{\Pi}^{\varepsilon}_{\vec{\mathbb{P}}}(f_{1},f_{2})(x):=\sum_{\vec{P}\in\vec{\mathbb{P}}}\frac{\widetilde{C}^{\varepsilon}_{Q_{\vec{P}}}}{|I_{\vec{P}}|^{\frac{1}{2}}}
  \langle f_{1},\Phi^{1}_{\vec{P}_{1}}\rangle\langle f_{2},\Phi^{2}_{\vec{P}_{2}}\rangle\Phi^{3}_{\vec{P}_{3}}(x),
\end{equation}
where $\Phi^{j}_{\vec{P}_{j}}:=\Phi^{j,(0,0)}_{\vec{P}_{j}}$ for $j=1,2,3$ respectively, the finite set $\vec{\mathbb{P}}=\mathbb{P}'\times\mathbb{P}''$ is an arbitrary sparse collection (of bi-parameter tri-tiles) of rank $1$. As have discussed above, we now reach a conclusion that the proof of Theorem \ref{main2} can be reduced to proving the following $L^{p}$ estimates for discrete bilinear model operators $\widetilde{\Pi}^{\varepsilon}_{\vec{\mathbb{P}}}$.

\begin{prop}\label{equivalent2}
If the finite set $\vec{\mathbb{P}}$ is an arbitrary sparse collection of rank $1$, then operator $\widetilde{\Pi}^{\varepsilon}_{\vec{\mathbb{P}}}$ given by \eqref{model2} maps $L^{p_{1}}(\mathbb{R}^{2})\times L^{p_{2}}(\mathbb{R}^{2})\rightarrow L^{p}(\mathbb{R}^{2})$ boundedly for any $1<p_{1},p_{2}\leq\infty$ satisfying $\frac{1}{p}=\frac{1}{p_{1}}+\frac{1}{p_{2}}$ and $\frac{2}{3}<p<\infty$. Moreover, the implicit constants in the bounds depend only on $\varepsilon$, $p_{1}$, $p_{2}$, $p$ and are independent of the particular finite sparse collection $\vec{\mathbb{P}}$ of rank $1$.
\end{prop}

\subsection{Multi-linear interpolations}
First, let's review the following terminologies and definitions of multi-linear interpolation arguments from \cite{MS,MTT1}.
\begin{defn}\label{tuples}(\cite{MS,MTT1})
An $n$-tuple $\beta=(\beta_{1},\cdots,\beta_{n})$ is said to be \emph{admissible} if and only if $\beta_{j}<1$ for every $1\leq j\leq n$, $\sum_{j=1}^{n}\beta_{j}=1$ and there is at most one index $j$ for which $\beta_{j}<0$. An index $j$ is called \emph{good} if $\beta_{j}\geq0$ and \emph{bad} if $\beta_{j}<0$. A \emph{good tuple} is an admissible tuple that contains only good indices; a \emph{bad tuple} is an admissible tuple that contains precisely one bad index.
\end{defn}
\begin{defn}\label{major}(\cite{MTT1})
Let $E$, $E'$ be sets of finite measure. We say that $E'$ is a \emph{major subset} of $E$ if $E'\subseteq E$ and $|E'|\geq\frac{1}{2}|E|$.
\end{defn}
\begin{defn}\label{RW tri-linear}(\cite{MS,MTT1})
If $\beta=(\beta_{1},\cdots,\beta_{n})$ is an admissible tuple, we say that an $n$-linear form $\Lambda$ is of \emph{restricted weak type} $\beta$ if and only if, for every sequence $E_{1},\cdots,E_{n}$ of measurable sets with positive and finite measure, there exists a major subset $E'_{j}$ of $E_{j}$ for each bad index $j$ (one or none) such that
\begin{equation}\label{2.50}
  |\Lambda(f_{1},\cdots,f_{n})|\lesssim|E_{1}|^{\beta_{1}}\cdots|E_{n}|^{\beta_{n}}
\end{equation}
for every measurable functions $|f_{i}|\leq\chi_{E'_{i}}$ ($i=1,\cdots,n$), where we adopt the convention $E'_{i}=E_{i}$ for good indices $i$. If $\beta$ is bad with bad index $j_{0}$, and it happens that one can choose the major subset $E'_{j_{0}}\subseteq E_{j_{0}}$ in a way that depends only on the measurable sets $E_{1},\cdots,E_{n}$ and not on $\beta$, we say that $\Lambda$ is of \emph{uniformly restricted weak type}.
\end{defn}
\begin{defn}\label{RW-type estimates}(\cite{MS})
Let $1<p_{1},p_{2}\leq\infty$ and $0<p<\infty$ be such that $\frac{1}{p}=\frac{1}{p_{1}}+\frac{1}{p_{2}}$. An arbitrary bilinear operator $T$ is said to be of the restricted weak type $(p_{1},p_{2},p)$ if and only if for all measurable sets $E_{1}$, $E_{2}$, $E$ of finite measure there exists $E'\subseteq E$ with $|E'|\simeq|E|$ such that
\begin{equation}\label{2.51}
  |\int_{\mathbb{R}^{d}}T(f_{1},f_{2})(x)f(x)dx|\lesssim|E_{1}|^{\frac{1}{p_{1}}}|E_{2}|^{\frac{1}{p_{2}}}|E'|^{\frac{1}{p'}}
\end{equation}
for every $|f_{1}|\leq\chi_{E_{1}}$, $|f_{2}|\leq\chi_{E_{2}}$ and $|f|\leq\chi_{E'}$.
\end{defn}

By using multi-linear interpolation (see \cite{GrTa,Janson,MS,MTT1}) and the symmetry of operators $\Pi^{\varepsilon}_{\vec{\mathbb{P}}}$ and $\widetilde{\Pi}^{\varepsilon}_{\vec{\mathbb{P}}}$, we can reduce further the proof of Proposition \ref{equivalent1} and Proposition \ref{equivalent2} to proving the following restricted weak type estimates for the model operators $\Pi^{\varepsilon}_{\vec{\mathbb{P}}}$ and $\widetilde{\Pi}^{\varepsilon}_{\vec{\mathbb{P}}}$.
\begin{prop}\label{RW-type equivalent}
Let $p_{1}$ and $p_{2}$ be such that $p_{1}$ is strictly larger than $1$ and arbitrarily close to $1$ and $p_{2}$ is strictly smaller than $2$ and arbitrarily close to $2$ and such that for $\frac{1}{p}:=\frac{1}{p_{1}}+\frac{1}{p_{2}}$, one has $\frac{2}{3}<p<1$. Then both the model operators $\Pi^{\varepsilon}_{\vec{\mathbb{P}}}$ and $\widetilde{\Pi}^{\varepsilon}_{\vec{\mathbb{P}}}$ defined in \eqref{model1} and \eqref{model2} are of the restricted weak type $(p_{1},p_{2},p)$. Moreover, the implicit constants in the bounds depend only on $\varepsilon$, $p_{1}$, $p_{2}$, $p$ and are independent of the particular choice of the finite collection $\vec{\mathbb{P}}$.
\end{prop}

Indeed, first we should observe that if $p_{1}$, $p_{2}$, $p$ are as in Proposition \ref{equivalent1} and \ref{equivalent2} then the $3$-tuple $(\frac{1}{p_{1}},\frac{1}{p_{2}},\frac{1}{p'})$ lies in the interior of the convex hull of the following six extremal points: $\beta^{1}:=(-\frac{1}{2},\frac{1}{2},1)$, $\beta^{2}:=(-\frac{1}{2},1,\frac{1}{2})$, $\beta^{3}:=(\frac{1}{2},-\frac{1}{2},1)$, $\beta^{4}:=(1,-\frac{1}{2},\frac{1}{2})$, $\beta^{5}:=(\frac{1}{2},1,-\frac{1}{2})$ and $\beta^{6}:=(1,\frac{1}{2},-\frac{1}{2})$. Then, if we assume that Proposition \ref{RW-type equivalent} has been proved, from the symmetry of operators $\Pi^{\varepsilon}_{\vec{\mathbb{P}}}$ and $\widetilde{\Pi}^{\varepsilon}_{\vec{\mathbb{P}}}$ and their adjoints, we deduce that both the tri-linear forms associated to bilinear operators $\Pi^{\varepsilon}_{\vec{\mathbb{P}}}$ and $\widetilde{\Pi}^{\varepsilon}_{\vec{\mathbb{P}}}$ are of \emph{uniformly restricted weak type} $\beta$ for $3$-tuples $\beta=(\beta_{1},\beta_{2},\beta_{3})$ arbitrarily close to the six extremal points $\beta^{1},\cdots,\beta^{6}$ inside the convex hull of them and satisfying if $\beta_{j}$ is close to $\frac{1}{2}$ for some $j=1,2,3$ then $\beta_{j}$ is strictly larger than $\frac{1}{2}$. By using multi-linear interpolation lemma 9.4 and 9.6 in \cite{MS} or lemma 3.8 in \cite{MTT1}, we first obtain restricted weak type estimates of $\Lambda$ for good tuples inside the smaller convex hull of the three coordinate points $(1,0,0)$, $(0,1,0)$ and $(0,0,1)$. After that, we use the interpolation lemma 9.5 in \cite{MS} or lemma 3.10 in \cite{MTT1} to obtain restricted weak type estimates of $\Lambda$ for bad tuples and finally conclude that restricted weak type estimates of $\Lambda$ hold for all tuples $\beta$ inside the convex hull of the six extremal points $\beta^{1},\cdots,\beta^{6}$.

It only remains to convert these restricted weak type estimates into strong type estimates. To do this, one just has to apply (exactly as in \cite{MTT1}) the multi-linear Marcinkiewicz interpolation theorem in \cite{Janson} in the case of good tuples and the interpolation lemma 3.11 in \cite{MTT1} in the case of bad tuples. This ends the proof of Proposition \ref{equivalent1} and \ref{equivalent2}, and as a consequence, completes the proof of our main results, Theorem \ref{main1} and \ref{main2}. Therefore, we only have the task of proving Proposition \ref{RW-type equivalent} from now on.

\section{Trees, $L^{2}$ sizes and $L^{2}$ energies}

\subsection{Trees}
We should recall that for discrete bilinear paraproducts, the frequency intervals have been already organized with the lacunary properties (see \cite{MS,MPTT1,MPTT2}), we could use square function and Maximal function estimates to handle the corresponding terms easily, at least in the Banach case. By the properties of the collection $\mathbb{P}''$ of tri-tiles we have explained in Remark \ref{properties}, we can organize our collections of tri-tiles $\mathbb{P}'$, $\mathbb{P''}$ into trees as in \cite{GL}, which satisfy lacunary properties about certain frequency. We review the following standard definitions and properties for trees from \cite{MTT2}.
\begin{defn}\label{trees}
Let $\mathbb{P}$ be a sparse rank-$1$ collection of tri-tiles and $j\in\{1,2,3\}$. A sub-collection $T\subseteq\mathbb{P}$ is called a $j$-\emph{tree} if and only if there exists a tri-tile $P_{T}$ (called the top of the tree) such that
\begin{equation}\label{3.1}
  P_{j}\leq P_{T,j}
\end{equation}
for every $P\in T$.
\end{defn}
\begin{rem}
Note that a tree does not necessarily have to contain the corresponding top $P_{T}$. From now on, we will write $I_{T}$ and $\omega_{T,j}$ for $I_{P_{T}}$ and $\omega_{P_{T},j}$ for $j=1,2,3$ respectively. Then, we simply say that $T$ is a tree if it is a $j$-tree for some $j=1,2,3$.
\end{rem}

For every given dyadic interval $I_{0}$, there are potentially many tri-tiles $P$ in collections $\mathbb{P}'$ and $\mathbb{P}''$ with the property that $I_{P}=I_{0}$. Due to this \emph{extra degree of freedom} in frequency, we have infinitely many trees in our collections $\mathbb{P}'$ and $\mathbb{P}''$. We need to estimate each of these trees separately, and then add all these estimates together, by using the \emph{almost orthogonality} conditions for distinct trees. This motivates the following definition.

\begin{defn}\label{chain of disjoint trees}
Let $1\leq i\leq3$. A finite sequence of trees $T_{1},\cdots,T_{M}$ is said to be a \emph{chain of strongly} $i$-\emph{disjoint trees} if and only if \\
(i) $P_{i}\neq P'_{i}$ for every $P\in T_{l_{1}}$ and $P'\in T_{l_{2}}$ with $l_{1}\neq l_{2}$;\\
(ii) whenever $P\in T_{l_{1}}$ and $P'\in T_{l_{2}}$ with $l_{1}\neq l_{2}$ are such that $2\omega_{P_{i}}\cap2\omega_{P'_{i}}\neq\emptyset$, then if $|\omega_{P_{i}}|<|\omega_{P'_{i}}|$ one has $I_{P'}\cap I_{T_{l_{1}}}=\emptyset$ and if $|\omega_{P'_{i}}|<|\omega_{P_{i}}|$ one has $I_{P}\cap I_{T_{l_{2}}}=\emptyset$;\\
(iii) whenever $P\in T_{l_{1}}$ and $P'\in T_{l_{2}}$ with $l_{1}<l_{2}$ are such that $2\omega_{P_{i}}\cap2\omega_{P'_{i}}\neq\emptyset$, then if $|\omega_{P_{i}}|=|\omega_{P'_{i}}|$ one has $I_{P'}\cap I_{T_{l_{1}}}=\emptyset$.
\end{defn}

\subsection{$L^{2}$ sizes and $L^{2}$ energies}
Following \cite{MTT2}, we give the definitions of standard norms on sequences of tiles as follows.
\begin{defn}\label{sizes and energies}
Let $\mathbb{P}$ be a finite collection of tri-tiles, $j\in\{1,2,3\}$, and $f$ be an arbitrary function. We define the \emph{size} of the sequence $(\langle f,\Phi^{j}_{P_{j}}\rangle)_{P\in\mathbb{P}}$ by
\begin{equation}\label{3.2}
  size_{j}((\langle f,\Phi^{j}_{P_{j}}\rangle)_{P\in\mathbb{P}}):=\sup_{T\subseteq\mathbb{P}}(\frac{1}{|I_{T}|}\sum_{P\in T}|\langle f,\Phi^{j}_{P_{j}}\rangle|^{2})^{\frac{1}{2}},
\end{equation}
where $T$ ranges over all trees in $\mathbb{P}$ that are $i$-trees for some $i\neq j$. For $j=1,2,3$, we define the \emph{energy} of the sequence $(\langle f,\Phi^{j}_{P_{j}}\rangle)_{P\in\mathbb{P}}$ by
\begin{equation}\label{3.3}
  energy_{j}((\langle f,\Phi^{j}_{P_{j}}\rangle)_{P\in\mathbb{P}}):=\sup_{n\in\mathbb{Z}}\sup_{\mathbb{T}}2^{n}(\sum_{T\in\mathbb{T}}|I_{T}|)^{\frac{1}{2}},
\end{equation}
where now $\mathbb{T}$ ranges over all chains of strongly $j$-disjoint trees in $\mathbb{P}$ (which are $i$-trees for some $i\neq j$) having the property that
\begin{equation}\label{3.4}
  (\sum_{P\in T}|\langle f,\Phi^{j}_{P_{j}}\rangle|^{2})^{\frac{1}{2}}\geq2^{n}|I_{T}|^{\frac{1}{2}}
\end{equation}
for all $T\in\mathbb{T}$ and such that
\begin{equation}\label{3.5}
  (\sum_{P\in T'}|\langle f,\Phi^{j}_{P_{j}}\rangle|^{2})^{\frac{1}{2}}\leq2^{n+1}|I_{T'}|^{\frac{1}{2}}
\end{equation}
for all subtrees $T'\subseteq T\in\mathbb{T}$.
\end{defn}

The \emph{size} measures the extent to which the sequences $(\langle f,\Phi^{j}_{P_{j}}\rangle)_{P\in\mathbb{P}}$ ($j=1,2,3$) can concentrate on a single tree and should be thought of as a phase-space variant of the \emph{BMO} norm. The \emph{energy} is a phase-space variant of the $L^{2}$ norm. As the notation suggests, the number $\langle f,\Phi^{j}_{P_{j}}\rangle$ should be thought of as being associated with the tile $P_{j}$ ($j=1,2,3$) rather than the full tri-tile $P$.

Let $\mathbb{P}$ be a finite collection of tri-tiles. Denote by $\Pi_{\mathbb{P}}$ the discrete bilinear operator given by
\begin{equation*}
  \Pi_{\mathbb{P}}(f_{1},f_{2})(x)=\sum_{P\in\mathbb{P}}\frac{1}{|I_{P}|^{\frac{1}{2}}}\langle f_{1},\Phi^{1}_{P_{1}}\rangle\langle f_{2},\Phi^{2}_{P_{2}}\rangle
  \Phi^{3}_{P_{3}}(x).
\end{equation*}
The following proposition provides a way of estimating the trilinear form associated with bilinear operator $\Pi_{\mathbb{P}}(f_{1},f_{2})$. We define
\begin{equation*}
  \Lambda_{\mathbb{P}}(f_{1},f_{2},f_{3}):=\int_{\mathbb{R}}T_{\mathbb{P}}(f_{1},f_{2})(x)f_{3}(x)dx.
\end{equation*}
\begin{prop}\label{tile norm estimates}(\cite{MTT2})
Let $\mathbb{P}$ be a finite collection of tri-tiles. Then
\begin{equation}\label{3.6}
  |\Lambda_{\mathbb{P}}(f_{1},f_{2},f_{3})|\lesssim\prod_{j=1}^{3}(size_{j}((\langle f_{j},\Phi^{j}_{P_{j}}\rangle)_{P\in\mathbb{P}}))^{\theta_{j}}(energy_{j}((\langle f_{j},\Phi^{j}_{P_{j}}\rangle)_{P\in\mathbb{P}}))^{1-\theta_{j}}
\end{equation}
for any $0\leq\theta_{1},\theta_{2},\theta_{3}<1$ with $\theta_{1}+\theta_{2}+\theta_{3}=1$; the implicit constants depend on the $\theta_{j}$ but are independent of the other parameters.
\end{prop}

\subsection{Estimates for sizes and energies}
In order to apply Proposition \ref{tile norm estimates}, we need to estimate further the sizes and energies appearing on the right-hand side of \eqref{3.6}.
\begin{lem}\label{size estimates}(\cite{MS,MTT2})
Let $j\in\{1,2,3\}$ and $f\in L^{2}(\mathbb{R})$. Then one has
\begin{equation}\label{3.7}
  size_{j}((\langle f,\Phi^{j}_{P_{j}}\rangle)_{P\in\mathbb{P}})\lesssim\sup_{P\in\mathbb{P}}\frac{1}{|I_{P}|}\int_{\mathbb{R}}|f|\widetilde{\chi}^{M}_{I_{P}}dx
\end{equation}
for every $M>0$, where the approximate cutoff function $\widetilde{\chi}^{M}_{I_{P}}(x):=(1+\frac{dist(x,I_{P})}{|I_{P}|})^{-M}$ and the implicit constants depend on $M$.
\end{lem}
\begin{lem}\label{energy estimates}(Bessel-type estimates, \cite{MTT2}).
Let $j\in\{1,2,3\}$ and $f\in L^{2}(\mathbb{R})$. Then
\begin{equation}\label{3.8}
  energy_{j}((\langle f,\Phi^{j}_{P_{j}}\rangle)_{P\in\mathbb{P}})\lesssim\|f\|_{L^{2}}.
\end{equation}
\end{lem}

\section{Proof of Theorem \ref{main1}}

In this section, we prove Theorem \ref{main1} by carrying out the proof of Proposition \ref{RW-type equivalent} for model operators $\Pi^{\varepsilon}_{\vec{\mathbb{P}}}$ defined in \eqref{model1} with $\vec{\mathbb{P}}=\widetilde{\mathbb{P}'}\times\mathbb{P}''$.

Fix indices $p_{1}$, $p_{2}$, $p$ as in the hypothesis of Proposition \ref{RW-type equivalent}. Fix arbitrary measurable sets $E_{1}$, $E_{2}$, $E_{3}$ of finite measure (by using the scaling invariance of $\Pi^{\varepsilon}_{\vec{\mathbb{P}}}$, we can assume further that $|E_{3}|=1$). Our goal is to find $E'_{3}\subseteq E_{3}$ with $|E'_{3}|\simeq|E_{3}|=1$ such that, for any functions $|f_{1}|\leq\chi_{E_{1}}$, $|f_{2}|\leq\chi_{E_{2}}$ and $|f_{3}|\leq\chi_{E'_{3}}$, one has the corresponding trilinear forms $\Lambda^{\varepsilon}_{\vec{\mathbb{P}}}(f_{1},f_{2},f_{3})$ defined by
\begin{equation}\label{4.1}
  \Lambda^{\varepsilon}_{\vec{\mathbb{P}}}(f_{1},f_{2},f_{3}):=\int_{\mathbb{R}^{2}}\Pi^{\varepsilon}_{\vec{\mathbb{P}}}(f_{1},f_{2})(x)f_{3}(x)dx
\end{equation}
satisfy estimates
\begin{equation}\label{4.2}
  |\Lambda^{\varepsilon}_{\vec{\mathbb{P}}}(f_{1},f_{2},f_{3})|=|\sum_{\vec{P}\in\vec{\mathbb{P}}}\frac{C^{\varepsilon}_{Q_{\vec{P}}}}{|I_{\vec{P}}|^{\frac{1}{2}}}
  \langle f_{1},\Phi^{1}_{\vec{P}_{1}}\rangle\langle f_{2},\Phi^{2}_{\vec{P}_{2}}\rangle\langle f_{3},\Phi^{3}_{\vec{P}_{3}}\rangle|\lesssim_{\varepsilon,p,p_{1},p_{2}}|E_{1}|^{\frac{1}{p_{1}}}|E_{2}|^{\frac{1}{p_{2}}},
\end{equation}
where $p_{1}$ is larger than but close to $1$, while $p_{2}$ is smaller than but close to $2$.

From \cite{MPTT2}, we can find the following generic decomposition lemma.
\begin{lem}\label{generic decomposition}
Let $J\subseteq\mathbb{R}$ be a fixed interval. Then every smooth bump function $\phi_{J}$ adapted to $J$ can be naturally decomposed as follows:
\begin{equation*}
  \phi_{J}=\sum_{\ell\in\mathbb{N}}2^{-100\ell}\phi^{\ell}_{J},
\end{equation*}
where for every $\ell\in\mathbb{N}$, $\phi^{\ell}_{J}$ is also a bump function adapted to $J$ but having the additional property that $supp(\phi^{\ell}_{J})\subseteq2^{\ell}J$. If in addition we assume that $\int_{\mathbb{R}}\phi_{J}(x)dx=0$, then the functions $\phi^{\ell}_{J}$ can be chosen such that $\int_{\mathbb{R}}\phi^{\ell}_{J}(x)dx=0$ for every $\ell\in\mathbb{N}$.
\end{lem}

We use $2^{\ell}J$ to denote the interval having the same center as $J$ but with length $2^{\ell}$ times that of $J$ hereafter.

By using Lemma \ref{generic decomposition}, we can estimate the left-hand side of \eqref{4.2} by
\begin{equation}\label{4.3}
  |\Lambda^{\varepsilon}_{\vec{\mathbb{P}}}(f_{1},f_{2},f_{3})|\lesssim\sum_{\ell\in\mathbb{N}}2^{-100\ell}
  \Lambda^{\varepsilon,\ell}_{\vec{\mathbb{P}}}(f_{1},f_{2},f_{3}).
\end{equation}
The tri-linear forms $\Lambda^{\varepsilon,\ell}_{\vec{\mathbb{P}}}(f_{1},f_{2},f_{3})$ ($\ell\in\mathbb{N}$) are defined by
\begin{equation}\label{4.4}
  \Lambda^{\varepsilon,\ell}_{\vec{\mathbb{P}}}(f_{1},f_{2},f_{3}):=\sum_{\vec{P}\in\vec{\mathbb{P}}}
  \frac{|C^{\varepsilon}_{Q_{\vec{P}}}|}{|I_{\vec{P}}|^{\frac{1}{2}}}
  |\langle f_{1},\Phi^{1}_{\vec{P}_{1}}\rangle||\langle f_{2},\Phi^{2}_{\vec{P}_{2}}\rangle||\langle f_{3},\Phi^{3,\ell}_{\vec{P}_{3}}\rangle|,
\end{equation}
where the new bi-parameter wave packets $\Phi^{3,\ell}_{\vec{P}_{3}}:=\Phi^{3,\ell}_{\widetilde{P'_{3}}}\otimes\Phi^{3}_{P''_{3}}$ with additional property that $supp(\Phi^{3,\ell}_{\widetilde{P'_{3}}})\subseteq2^{\ell}I_{\widetilde{P'_{3}}}=2^{\ell}I_{\widetilde{P'}}$.

For every $\ell\in\mathbb{N}$, we define the sets as follows:
\begin{equation}\label{4.5}
  \Omega_{-10\ell}:=\{x\in\mathbb{R}^{2}:MM(\frac{\chi_{E_{1}}}{|E_{1}|})(x)>C2^{10\ell}\}
  \cup\{x\in\mathbb{R}^{2}:MM(\frac{\chi_{E_{2}}}{|E_{2}|})(x)>C2^{10\ell}\},
\end{equation}
and
\begin{equation}\label{4.6}
  \widetilde{\Omega}_{-10\ell}:=\{x\in\mathbb{R}^{2}:MM(\chi_{\Omega_{-10\ell}})(x)>2^{-\ell}\},
\end{equation}
where the double maximal operator $MM$ is given by
\begin{equation}\label{4.7}
  MM(h)(x,y):=\sup_{\substack{\text{dyadic rectangle} \, R \\ (x,y)\in R}}\frac{1}{|R|}\int_{R}|h(u,v)|dudv.
\end{equation}
Finally, we define the exceptional set
\begin{equation}\label{4.8}
  U:=\bigcup_{\ell\in\mathbb{N}}\widetilde{\Omega}_{-10\ell}.
\end{equation}
It is clear that $|U|<\frac{1}{10}$ if $C$ is a large enough constant, which we fix from now on. Then, we define $E'_{3}:=E_{3}\setminus U$ and observe that $|E'_{3}|\simeq1$.

Now fix $\ell\in\mathbb{N}$, and split the trilinear form $\Lambda^{\varepsilon,\ell}_{\vec{\mathbb{P}}}(f_{1},f_{2},f_{3})$ defined in \eqref{4.4} into two parts as follows:
\begin{eqnarray}\label{4.9}
 \nonumber \Lambda^{\varepsilon,\ell}_{\vec{\mathbb{P}}}(f_{1},f_{2},f_{3})&=& \sum_{\substack{\vec{P}\in\vec{\mathbb{P}}:\\I_{\vec{P}}\cap\Omega^{c}_{-10\ell}\neq\emptyset}}
  \frac{|C^{\varepsilon}_{Q_{\vec{P}}}|}{|I_{\vec{P}}|^{\frac{1}{2}}}
  |\langle f_{1},\Phi^{1}_{\vec{P}_{1}}\rangle||\langle f_{2},\Phi^{2}_{\vec{P}_{2}}\rangle||\langle f_{3},\Phi^{3,\ell}_{\vec{P}_{3}}\rangle| \\
  &&+\sum_{\substack{\vec{P}\in\vec{\mathbb{P}}:\\I_{\vec{P}}\cap\Omega^{c}_{-10\ell}=\emptyset}}
  \frac{|C^{\varepsilon}_{Q_{\vec{P}}}|}{|I_{\vec{P}}|^{\frac{1}{2}}}
  |\langle f_{1},\Phi^{1}_{\vec{P}_{1}}\rangle||\langle f_{2},\Phi^{2}_{\vec{P}_{2}}\rangle||\langle f_{3},\Phi^{3,\ell}_{\vec{P}_{3}}\rangle|\\
 \nonumber &=:& \Lambda^{\varepsilon,\ell}_{\vec{\mathbb{P}},I}(f_{1},f_{2},f_{3})+\Lambda^{\varepsilon,\ell}_{\vec{\mathbb{P}},II}(f_{1},f_{2},f_{3}),
\end{eqnarray}
where the notation $A^{c}$ denotes the complementary set of a set $A$.

\subsection{Estimates for trilinear form $\Lambda^{\varepsilon,\ell}_{\vec{\mathbb{P}},I}(f_{1},f_{2},f_{3})$}
We can decompose the collection $\widetilde{\mathbb{P}'}$ of tri-tiles into
\begin{equation}\label{4.10}
  \widetilde{\mathbb{P}'}=\bigcup_{k'\in\mathbb{Z}}\widetilde{\mathbb{P}'_{k'}},
\end{equation}
where
\begin{equation}\label{4.11}
  \widetilde{\mathbb{P}'_{k'}}:=\{\widetilde{P'}\in\widetilde{\mathbb{P}'}: |I_{\widetilde{P'}}|=2^{-k'}\}.
\end{equation}
As a consequence, we can split the trilinear form $\Lambda^{\varepsilon,\ell}_{\vec{\mathbb{P}},I}(f_{1},f_{2},f_{3})$ into
\begin{eqnarray}\label{4.12}
  \Lambda^{\varepsilon,\ell}_{\vec{\mathbb{P}},I}(f_{1},f_{2},f_{3})&=&\sum_{k'\in\mathbb{Z}}
  \sum_{\substack{\vec{P}\in\widetilde{\mathbb{P}'_{k'}}\times\mathbb{P}'': \\ I_{\vec{P}}\cap\Omega^{c}_{-10\ell}\neq\emptyset}}
  |C^{\varepsilon}_{Q_{\vec{P}}}|\frac{|I_{\widetilde{P'}}|}{|I_{P''}|^{\frac{1}{2}}}
  |\langle \frac{\langle f_{1},\Phi^{1}_{\widetilde{P'_{1}}}\rangle}{|I_{\widetilde{P'}}|^{\frac{1}{2}}},\Phi^{1}_{P''_{1}}\rangle|\\
  \nonumber && \quad\quad\quad\quad\quad\quad\quad\quad\quad \times|\langle \frac{\langle f_{2},\Phi^{2}_{\widetilde{P'_{2}}}\rangle}{|I_{\widetilde{P'}}|^{\frac{1}{2}}},\Phi^{2}_{P''_{2}}\rangle|
  |\langle \frac{\langle f_{3},\Phi^{3,\ell}_{\widetilde{P'_{3}}}\rangle}{|I_{\widetilde{P'}}|^{\frac{1}{2}}},\Phi^{3}_{P''_{3}}\rangle|.
\end{eqnarray}

Note that by Lemma \ref{Fourier coefficient1}, we can estimate the Fourier coefficients $C^{\varepsilon}_{Q_{\vec{P}}}:=C^{\varepsilon}_{Q_{\vec{P}},\vec{0},\vec{0},\vec{0}}$ for each $\vec{P}\in\widetilde{\mathbb{P}'_{k'}}\times\mathbb{P}''$ ($k'\in\mathbb{Z}$) by
\begin{equation}\label{4.13}
  |C^{\varepsilon}_{Q_{\vec{P}}}|\lesssim C^{\varepsilon}_{k'} \,\,\,\,\,\,\,\,\, with \,\,\,\,\,\,\,\,\, \sum_{k'\in\mathbb{Z}}C^{\varepsilon}_{k'}\lesssim_{\varepsilon}1.
\end{equation}
For each fixed $\widetilde{P'}\in\widetilde{\mathbb{P}'}$, we define the sub-collection
\begin{equation*}
  \mathbb{P}''_{\widetilde{P'}}:=\{P''\in\mathbb{P}'': \, I_{\vec{P}}\cap\Omega^{c}_{-10\ell}\neq\emptyset\}.
\end{equation*}
Therefore, by using Proposition \ref{tile norm estimates}, we derive the following estimates
\begin{eqnarray}\label{4.14}
\nonumber \Lambda^{\varepsilon,\ell}_{\vec{\mathbb{P}},I}(f_{1},f_{2},f_{3})&\lesssim&\sum_{k'\in\mathbb{Z}}C^{\varepsilon}_{k'}
\sum_{\widetilde{P'}\in\widetilde{\mathbb{P}'_{k'}}}|I_{\widetilde{P'}}|
[\prod_{j=1}^{2}(energy_{j}((\langle \frac{\langle f_{j},\Phi^{j}_{\widetilde{P'_{j}}}\rangle}{|I_{\widetilde{P'}}|^{\frac{1}{2}}},\Phi^{j}_{P''_{j}}\rangle)_{P''\in\mathbb{P}''_{\widetilde{P'}}}))^{1-\theta_{j}}\\
  &\times&(size_{j}((\langle \frac{\langle f_{j},\Phi^{j}_{\widetilde{P'_{j}}}\rangle}{|I_{\widetilde{P'}}|^{\frac{1}{2}}},\Phi^{j}_{P''_{j}}\rangle)_{P''\in\mathbb{P}''_{\widetilde{P'}}}))^{\theta_{j}}]
 (size_{3}((\langle \frac{\langle f_{3},\Phi^{3,\ell}_{\widetilde{P'_{3}}}\rangle}{|I_{\widetilde{P'}}|^{\frac{1}{2}}},\Phi^{3}_{P''_{3}}\rangle)_{P''\in\mathbb{P}''_{\widetilde{P'}}}))^{\theta_{3}}\\
 \nonumber && \quad\quad\quad\quad\quad\quad\quad\quad\quad\quad\quad\quad \times(energy_{3}((\langle \frac{\langle f_{3},\Phi^{3,\ell}_{\widetilde{P'_{3}}}\rangle}{|I_{\widetilde{P'}}|^{\frac{1}{2}}},\Phi^{3}_{P''_{3}}\rangle)_{P''\in\mathbb{P}''_{\widetilde{P'}}}))^{1-\theta_{3}}
\end{eqnarray}
for any $0\leq\theta_{1},\theta_{2},\theta_{3}<1$ with $\theta_{1}+\theta_{2}+\theta_{3}=1$.

To estimate the right-hand side of \eqref{4.14}, note that $I_{\vec{P}}\cap\Omega^{c}_{-10\ell}\neq\emptyset$ and $supp \, f_{3}\subseteq E'_{3}\subseteq\mathbb{R}^{2}\setminus U$, we apply the \emph{size} estimates in Lemma \ref{size estimates} and get for each $\widetilde{P'}\in\widetilde{\mathbb{P}'_{k'}}$,
\begin{equation}\label{4.15}
  size_{1}((\langle \frac{\langle f_{1},\Phi^{1}_{\widetilde{P'_{1}}}\rangle}{|I_{\widetilde{P'}}|^{\frac{1}{2}}},\Phi^{1}_{P''_{1}}\rangle)_{P''\in\mathbb{P}''_{\widetilde{P'}}})
  \lesssim\sup_{P''\in\mathbb{P}''_{\widetilde{P'}}}\frac{1}{|I_{P''}|}\int_{\mathbb{R}}|\frac{\langle f_{1},\Phi^{1}_{\widetilde{P'_{1}}}\rangle}{|I_{\widetilde{P'}}|^{\frac{1}{2}}}|\widetilde{\chi}^{M}_{I_{P''}}dx
  \lesssim2^{10\ell}|E_{1}|,
\end{equation}
\begin{equation}\label{4.16}
  size_{2}((\langle \frac{\langle f_{2},\Phi^{2}_{\widetilde{P'_{2}}}\rangle}{|I_{\widetilde{P'}}|^{\frac{1}{2}}},\Phi^{2}_{P''_{2}}\rangle)_{P''\in\mathbb{P}''_{\widetilde{P'}}})
  \lesssim\sup_{P''\in\mathbb{P}''_{\widetilde{P'}}}\frac{1}{|I_{P''}|}\int_{\mathbb{R}}|\frac{\langle f_{2},\Phi^{2}_{\widetilde{P'_{2}}}\rangle}{|I_{\widetilde{P'}}|^{\frac{1}{2}}}|\widetilde{\chi}^{M}_{I_{P''}}dx
  \lesssim2^{10\ell}|E_{2}|,
\end{equation}
\begin{equation}\label{4.17}
  size_{3}((\langle \frac{\langle f_{3},\Phi^{3,\ell}_{\widetilde{P'_{3}}}\rangle}{|I_{\widetilde{P'}}|^{\frac{1}{2}}},\Phi^{3}_{P''_{3}}\rangle)_{P''\in\mathbb{P}''_{\widetilde{P'}}})
  \lesssim\sup_{P''\in\mathbb{P}''_{\widetilde{P'}}}\frac{1}{|I_{P''}|}\int_{\mathbb{R}}|\frac{\langle f_{3},\Phi^{3,\ell}_{\widetilde{P'_{3}}}\rangle}{|I_{\widetilde{P'}}|^{\frac{1}{2}}}|\widetilde{\chi}^{M}_{I_{P''}}dx
  \lesssim1,
\end{equation}
where $M>0$ is sufficiently large. By applying the \emph{energy} estimates in Lemma \ref{energy estimates} and H\"{o}lder estimates, we have for each $\widetilde{P'}\in\widetilde{\mathbb{P}'_{k'}}$,
\begin{equation}\label{4.18}
  energy_{1}((\langle \frac{\langle f_{1},\Phi^{1}_{\widetilde{P'_{1}}}\rangle}{|I_{\widetilde{P'}}|^{\frac{1}{2}}},\Phi^{1}_{P''_{1}}\rangle)_{P''\in\mathbb{P}''_{\widetilde{P'}}})\lesssim\|\frac{\langle f_{1},\Phi^{1}_{\widetilde{P'_{1}}}\rangle}{|I_{\widetilde{P'}}|^{\frac{1}{2}}}\|_{L^{2}(\mathbb{R})}\lesssim
  (\int_{E_{1}}\frac{\widetilde{\chi}^{100}_{I_{\widetilde{P'}}}(x_{1})}{|I_{\widetilde{P'}}|}dx_{1}dx_{2})^{\frac{1}{2}},
\end{equation}
\begin{equation}\label{4.19}
  energy_{2}((\langle \frac{\langle f_{2},\Phi^{2}_{\widetilde{P'_{2}}}\rangle}{|I_{\widetilde{P'}}|^{\frac{1}{2}}},\Phi^{2}_{P''_{2}}\rangle)_{P''\in\mathbb{P}''_{\widetilde{P'}}})\lesssim\|\frac{\langle f_{2},\Phi^{2}_{\widetilde{P'_{2}}}\rangle}{|I_{\widetilde{P'}}|^{\frac{1}{2}}}\|_{L^{2}(\mathbb{R})}\lesssim
  (\int_{E_{2}}\frac{\widetilde{\chi}^{100}_{I_{\widetilde{P'}}}(x_{1})}{|I_{\widetilde{P'}}|}dx_{1}dx_{2})^{\frac{1}{2}},
\end{equation}
\begin{equation}\label{4.20}
  energy_{3}((\langle \frac{\langle f_{3},\Phi^{3,\ell}_{\widetilde{P'_{3}}}\rangle}{|I_{\widetilde{P'}}|^{\frac{1}{2}}},\Phi^{3}_{P''_{3}}\rangle)_{P''\in\mathbb{P}''_{\widetilde{P'}}})\lesssim\|\frac{\langle f_{3},\Phi^{3,\ell}_{\widetilde{P'_{3}}}\rangle}{|I_{\widetilde{P'}}|^{\frac{1}{2}}}\|_{L^{2}(\mathbb{R})}\lesssim
  (\int_{E'_{3}}\frac{\widetilde{\chi}^{100,\ell}_{I_{\widetilde{P'}}}(x_{1})}{|I_{\widetilde{P'}}|}dx_{1}dx_{2})^{\frac{1}{2}},
\end{equation}
where the approximate cutoff function $\widetilde{\chi}^{100,\ell}_{I_{\widetilde{P'}}}(x_{1})$ decays rapidly (of order 100) away from the interval $I_{\widetilde{P'}}$ at scale $|I_{\widetilde{P'}}|$ and satisfies additional property that $supp \, \widetilde{\chi}^{100,\ell}_{I_{\widetilde{P'}}}\subseteq2^{\ell}I_{\widetilde{P'}}$.

Now we insert the size and energy estimates \eqref{4.15}-\eqref{4.20} into \eqref{4.14} and get
\begin{eqnarray}\label{4.21}
  &&\Lambda^{\varepsilon,\ell}_{\vec{\mathbb{P}},I}(f_{1},f_{2},f_{3}) \\
 \nonumber &\lesssim&2^{10\ell}|E_{1}|^{\theta_{1}}|E_{2}|^{\theta_{2}}
  \sum_{k'\in\mathbb{Z}}C^{\varepsilon}_{k'}\sum_{\widetilde{P'}\in\widetilde{\mathbb{P}'_{k'}}}
  (\int_{E_{1}}\widetilde{\chi}^{100}_{I_{\widetilde{{P'}}}}dx)^{\frac{1-\theta_{1}}{2}}
  (\int_{E_{2}}\widetilde{\chi}^{100}_{I_{\widetilde{{P'}}}}dx)^{\frac{1-\theta_{2}}{2}}
  (\int_{E'_{3}}\widetilde{\chi}^{100,\ell}_{I_{\widetilde{{P'}}}}dx)^{\frac{1-\theta_{3}}{2}}.
\end{eqnarray}
Since $|I_{\widetilde{P'}}|=2^{-k'}$ for every $\widetilde{P'}\in\widetilde{\mathbb{P}'_{k'}}$, all the dyadic intervals $I_{\widetilde{P'}}$ are disjoint, thus by using H\"{o}lder inequality, we can estimate the inner sum in the right-hand side of \eqref{4.21} by
\begin{eqnarray}\label{4.22}
  && (\sum_{\widetilde{P'}\in\widetilde{\mathbb{P}'_{k'}}}\int_{E_{1}}\widetilde{\chi}^{100}_{I_{\widetilde{{P'}}}}dx)^{\frac{1-\theta_{1}}{2}}
  (\sum_{\widetilde{P'}\in\widetilde{\mathbb{P}'_{k'}}}\int_{E_{2}}\widetilde{\chi}^{100}_{I_{\widetilde{{P'}}}}dx)^{\frac{1-\theta_{2}}{2}}
  (\sum_{\widetilde{P'}\in\widetilde{\mathbb{P}'_{k'}}}\int_{E'_{3}}\widetilde{\chi}^{100,\ell}_{I_{\widetilde{{P'}}}}dx)^{\frac{1-\theta_{3}}{2}}\\
  \nonumber &\lesssim&|E_{1}|^{\frac{1-\theta_{1}}{2}}|E_{2}|^{\frac{1-\theta_{2}}{2}}.
\end{eqnarray}
Combining the estimates \eqref{4.13}, \eqref{4.21} and \eqref{4.22}, we arrive at
\begin{eqnarray}\label{4.23}
  \Lambda^{\varepsilon,\ell}_{\vec{\mathbb{P}},I}(f_{1},f_{2},f_{3})&\lesssim& 2^{10\ell}|E_{1}|^{\theta_{1}}|E_{2}|^{\theta_{2}}|E_{1}|^{\frac{1-\theta_{1}}{2}}|E_{2}|^{\frac{1-\theta_{2}}{2}}
  \sum_{k'\in\mathbb{Z}}C^{\varepsilon}_{k'}\\
 \nonumber &&\lesssim_{\varepsilon,\theta_{1},\theta_{2},\theta_{3}} 2^{20\ell}|E_{1}|^{\frac{1+\theta_{1}}{2}}|E_{2}|^{\frac{1+\theta_{2}}{2}}
\end{eqnarray}
for every $\ell\in\mathbb{N}$ and $0\leq\theta_{1},\theta_{2},\theta_{3}<1$ with $\theta_{1}+\theta_{2}+\theta_{3}=1$.

By taking $\theta_{1}$ sufficiently close to $1$ and $\theta_{2}$ sufficiently close to $0$, one can make the exponents $\frac{2}{1+\theta_{1}}=p_{1}$ strictly larger than $1$ and close to $1$ and $\frac{2}{1+\theta_{2}}=p_{2}$ strictly smaller than $2$ and close to $2$. We finally get the estimate
\begin{equation}\label{I}
  \Lambda^{\varepsilon,\ell}_{\vec{\mathbb{P}},I}(f_{1},f_{2},f_{3})\lesssim_{\varepsilon,p,p_{1},p_{2}}2^{10\ell}|E_{1}|^{\frac{1}{p_{1}}}|E_{2}|^{\frac{1}{p_{2}}}
\end{equation}
for every $\ell\in\mathbb{N}$, $\varepsilon>0$ and $p$, $p_{1}$, $p_{2}$ satisfy the hypothesis of Proposition \ref{RW-type equivalent}.

\subsection{Estimates for trilinear form $\Lambda^{\varepsilon,\ell}_{\vec{\mathbb{P}},II}(f_{1},f_{2},f_{3})$}
One can observe that if $I_{\vec{P}}\subseteq\Omega_{-10\ell}$, then $2^{\ell}I_{\widetilde{P'}}\times I_{P''}\subseteq\widetilde{\Omega}_{-10\ell}$. Therefore, for each fixed $\widetilde{P'}\in\widetilde{\mathbb{P}'}$, we define the corresponding sub-collection of $\mathbb{P}''$ by
\begin{equation*}
  \mathbb{P}''_{\widetilde{P'}}:=\{P''\in\mathbb{P}'': \, I_{\vec{P}}\subseteq\Omega_{-10\ell}\},
\end{equation*}
then we can decompose the collection $\mathbb{P}''_{\widetilde{P'}}$ further, as follows:
\begin{equation}\label{4.24}
  \mathbb{P}''_{\widetilde{P'}}=\bigcup_{d''\in\mathbb{N}}\mathbb{P}''_{\widetilde{P'},d''},
\end{equation}
where
\begin{equation}\label{4.25}
  \mathbb{P}''_{\widetilde{P'},d''}:=\{P''\in\mathbb{P}''_{\widetilde{P'}}: \, 2^{\ell}I_{\widetilde{P'}}\times2^{d''}I_{P''}\subseteq\widetilde{\Omega}_{-10\ell}\}
\end{equation}
and $d''$ is maximal with this property.

Now we apply both the decompositions of $\widetilde{\mathbb{P}'}$ and $\mathbb{P}''_{\widetilde{P'}}$ defined in \eqref{4.10}, \eqref{4.24} at the same time, and split the trilinear form $\Lambda^{\varepsilon,\ell}_{\vec{\mathbb{P}},II}(f_{1},f_{2},f_{3})$ into
\begin{eqnarray}\label{4.26}
 \nonumber \Lambda^{\varepsilon,\ell}_{\vec{\mathbb{P}},II}(f_{1},f_{2},f_{3})&=&\sum_{k'\in\mathbb{Z}}
  \sum_{\widetilde{P'}\in\widetilde{\mathbb{P}'_{k'}}}
  |C^{\varepsilon}_{Q_{\vec{P}}}||I_{\widetilde{P'}}|\sum_{d''\in\mathbb{N}}\sum_{P''\in\mathbb{P}''_{\widetilde{P'},d''}}\frac{1}{|I_{P''}|^{\frac{1}{2}}}\\
  && \quad\quad\quad\quad\quad \times|\langle \frac{\langle f_{1},\Phi^{1}_{\widetilde{P'_{1}}}\rangle}{|I_{\widetilde{P'}}|^{\frac{1}{2}}},\Phi^{1}_{P''_{1}}\rangle||\langle \frac{\langle f_{2},\Phi^{2}_{\widetilde{P'_{2}}}\rangle}{|I_{\widetilde{P'}}|^{\frac{1}{2}}},\Phi^{2}_{P''_{2}}\rangle|
  |\langle \frac{\langle f_{3},\Phi^{3,\ell}_{\widetilde{P'_{3}}}\rangle}{|I_{\widetilde{P'}}|^{\frac{1}{2}}},\Phi^{3}_{P''_{3}}\rangle|.
\end{eqnarray}

In the inner sum of the above \eqref{4.26}, since $ 2^{\ell}I_{\widetilde{P'}}\times2^{d''}I_{P''}\subseteq\widetilde{\Omega}_{-10\ell}$, $supp(\Phi^{3,\ell}_{\widetilde{P'_{3}}})\subseteq2^{\ell}I_{\widetilde{P'}}$ and $supp \, f_{3}\subseteq E'_{3}\subseteq\mathbb{R}^{2}\setminus U$, we can assume hereafter in this subsection that
\begin{equation}\label{4.27}
  |f_{3}|\leq\chi_{E'_{3}}\chi_{2^{\ell}I_{\widetilde{P'}}}\chi_{(2^{d''}I_{P''})^{c}}.
\end{equation}

By using Proposition \ref{tile norm estimates} and \eqref{4.13}, we derive from \eqref{4.26} the following estimates
\begin{eqnarray}\label{4.28}
 &&\Lambda^{\varepsilon,\ell}_{\vec{\mathbb{P}},II}(f_{1},f_{2},f_{3})\\
\nonumber &\lesssim&\sum_{k'\in\mathbb{Z}}C^{\varepsilon}_{k'}
\sum_{\widetilde{P'}\in\widetilde{\mathbb{P}'_{k'}}}|I_{\widetilde{P'}}|\sum_{d''\in\mathbb{N}}
[\prod_{j=1}^{2}(energy_{j}((\langle \frac{\langle f_{j},\Phi^{j}_{\widetilde{P'_{j}}}\rangle}{|I_{\widetilde{P'}}|^{\frac{1}{2}}},\Phi^{j}_{P''_{j}}\rangle)_{P''\in\mathbb{P}''_{\widetilde{P'},d''}}))^{1-\theta_{j}}\\
 \nonumber &&\times(size_{j}((\langle \frac{\langle f_{j},\Phi^{j}_{\widetilde{P'_{j}}}\rangle}{|I_{\widetilde{P'}}|^{\frac{1}{2}}},\Phi^{j}_{P''_{j}}\rangle)_{P''\in\mathbb{P}''_{\widetilde{P'},d''}}))^{\theta_{j}}]
 (size_{3}((\langle \frac{\langle f_{3},\Phi^{3,\ell}_{\widetilde{P'_{3}}}\rangle}{|I_{\widetilde{P'}}|^{\frac{1}{2}}},\Phi^{3}_{P''_{3}}\rangle)_{P''\in\mathbb{P}''_{\widetilde{P'},d''}}))^{\theta_{3}}\\
 \nonumber && \times(energy_{3}((\langle \frac{\langle f_{3},\Phi^{3,\ell}_{\widetilde{P'_{3}}}\rangle}{|I_{\widetilde{P'}}|^{\frac{1}{2}}},\Phi^{3}_{P''_{3}}\rangle)_{P''\in\mathbb{P}''_{\widetilde{P'},d''}}))^{1-\theta_{3}}
\end{eqnarray}
for any $0\leq\theta_{1},\theta_{2},\theta_{3}<1$ with $\theta_{1}+\theta_{2}+\theta_{3}=1$.

To estimate the inner sum in the right-hand side of \eqref{4.28}, note that $I_{\vec{P}}\subseteq\Omega_{-10\ell}$, $P''\in\mathbb{P}''_{\widetilde{P'},d''}$ and $f_{3}$ satisfies \eqref{4.27}, we apply the \emph{size} estimates in Lemma \ref{size estimates} and get for each $\widetilde{P'}\in\widetilde{\mathbb{P}'_{k'}}$ and $d''\in\mathbb{N}$,
\begin{equation}\label{4.29}
  size_{1}((\langle \frac{\langle f_{1},\Phi^{1}_{\widetilde{P'_{1}}}\rangle}{|I_{\widetilde{P'}}|^{\frac{1}{2}}},\Phi^{1}_{P''_{1}}\rangle)_{P''\in\mathbb{P}''_{\widetilde{P'},d''}})
  \lesssim\sup_{P''\in\mathbb{P}''_{\widetilde{P'},d''}}\frac{1}{|I_{P''}|}\int_{\mathbb{R}}|\frac{\langle f_{1},\Phi^{1}_{\widetilde{P'_{1}}}\rangle}{|I_{\widetilde{P'}}|^{\frac{1}{2}}}|\widetilde{\chi}^{M}_{I_{P''}}dx
  \lesssim2^{11\ell+d''}|E_{1}|,
\end{equation}
\begin{equation}\label{4.30}
  size_{2}((\langle \frac{\langle f_{2},\Phi^{2}_{\widetilde{P'_{2}}}\rangle}{|I_{\widetilde{P'}}|^{\frac{1}{2}}},\Phi^{2}_{P''_{2}}\rangle)_{P''\in\mathbb{P}''_{\widetilde{P'},d''}})
  \lesssim\sup_{P''\in\mathbb{P}''_{\widetilde{P'},d''}}\frac{1}{|I_{P''}|}\int_{\mathbb{R}}|\frac{\langle f_{2},\Phi^{2}_{\widetilde{P'_{2}}}\rangle}{|I_{\widetilde{P'}}|^{\frac{1}{2}}}|\widetilde{\chi}^{M}_{I_{P''}}dx
  \lesssim2^{11\ell+d''}|E_{2}|,
\end{equation}
\begin{equation}\label{4.31}
  size_{3}((\langle \frac{\langle f_{3},\Phi^{3,\ell}_{\widetilde{P'_{3}}}\rangle}{|I_{\widetilde{P'}}|^{\frac{1}{2}}},\Phi^{3}_{P''_{3}}\rangle)_{P''\in\mathbb{P}''_{\widetilde{P'},d''}})
  \lesssim\sup_{P''\in\mathbb{P}''_{\widetilde{P'},d''}}\frac{1}{|I_{P''}|}\int_{\mathbb{R}}|\frac{\langle f_{3},\Phi^{3,\ell}_{\widetilde{P'_{3}}}\rangle}{|I_{\widetilde{P'}}|^{\frac{1}{2}}}|\widetilde{\chi}^{M}_{I_{P''}}dx
  \lesssim 2^{-(M-100)d''},
\end{equation}
where $M>0$ is arbitrarily large. Similar to the energy estimates obtained in \eqref{4.18}, \eqref{4.19} and \eqref{4.20}, by applying the \emph{energy} estimates in Lemma \ref{energy estimates} and H\"{o}lder estimates, we have for each $\widetilde{P'}\in\widetilde{\mathbb{P}'_{k'}}$ and $d''\in\mathbb{N}$,
\begin{equation}\label{4.32}
  energy_{1}((\langle \frac{\langle f_{1},\Phi^{1}_{\widetilde{P'_{1}}}\rangle}{|I_{\widetilde{P'}}|^{\frac{1}{2}}},\Phi^{1}_{P''_{1}}\rangle)_{P''\in\mathbb{P}''_{\widetilde{P'},d''}})\lesssim
  (\int_{E_{1}}\frac{\widetilde{\chi}^{100}_{I_{\widetilde{P'}}}(x_{1})}{|I_{\widetilde{P'}}|}dx_{1}dx_{2})^{\frac{1}{2}},
\end{equation}
\begin{equation}\label{4.33}
  energy_{2}((\langle \frac{\langle f_{2},\Phi^{2}_{\widetilde{P'_{2}}}\rangle}{|I_{\widetilde{P'}}|^{\frac{1}{2}}},\Phi^{2}_{P''_{2}}\rangle)_{P''\in\mathbb{P}''_{\widetilde{P'},d''}})\lesssim
  (\int_{E_{2}}\frac{\widetilde{\chi}^{100}_{I_{\widetilde{P'}}}(x_{1})}{|I_{\widetilde{P'}}|}dx_{1}dx_{2})^{\frac{1}{2}},
\end{equation}
\begin{equation}\label{4.34}
  energy_{3}((\langle \frac{\langle f_{3},\Phi^{3,\ell}_{\widetilde{P'_{3}}}\rangle}{|I_{\widetilde{P'}}|^{\frac{1}{2}}},\Phi^{3}_{P''_{3}}\rangle)_{P''\in\mathbb{P}''_{\widetilde{P'},d''}})\lesssim
  (\int_{E'_{3}}\frac{\widetilde{\chi}^{100,\ell}_{I_{\widetilde{P'}}}(x_{1})}{|I_{\widetilde{P'}}|}dx_{1}dx_{2})^{\frac{1}{2}},
\end{equation}
where the approximate cutoff function $\widetilde{\chi}^{100,\ell}_{I_{\widetilde{P'}}}(x_{1})$ decays rapidly (of order 100) away from the interval $I_{\widetilde{P'}}$ at scale $|I_{\widetilde{P'}}|$ and satisfies additional property that $supp \, \widetilde{\chi}^{100,\ell}_{I_{\widetilde{P'}}}\subseteq2^{\ell}I_{\widetilde{P'}}$.

Now we insert the size and energy estimates \eqref{4.29}-\eqref{4.34} into \eqref{4.28}, by using the estimates \eqref{4.13}, \eqref{4.22} and H\"{o}lder inequality, we get
\begin{eqnarray}\label{4.35}
  \nonumber \Lambda^{\varepsilon,\ell}_{\vec{\mathbb{P}},II}(f_{1},f_{2},f_{3})&\lesssim&2^{11\ell}|E_{1}|^{\theta_{1}}|E_{2}|^{\theta_{2}}
  \sum_{k'\in\mathbb{Z}}C^{\varepsilon}_{k'}\sum_{d''\in\mathbb{N}}2^{-(M\theta_{3}-100)d''}\\
  &\times&(\sum_{\widetilde{P'}\in\widetilde{\mathbb{P}'_{k'}}}\int_{E_{1}}\widetilde{\chi}^{100}_{I_{\widetilde{{P'}}}}dx)^{\frac{1-\theta_{1}}{2}}
  (\sum_{\widetilde{P'}\in\widetilde{\mathbb{P}'_{k'}}}\int_{E_{2}}\widetilde{\chi}^{100}_{I_{\widetilde{{P'}}}}dx)^{\frac{1-\theta_{2}}{2}}
  (\sum_{\widetilde{P'}\in\widetilde{\mathbb{P}'_{k'}}}\int_{E'_{3}}\widetilde{\chi}^{100,\ell}_{I_{\widetilde{{P'}}}}dx)^{\frac{1-\theta_{3}}{2}}\\
 \nonumber &&\lesssim_{\varepsilon,\theta_{1},\theta_{2},\theta_{3},M}2^{11\ell}|E_{1}|^{\frac{1+\theta_{1}}{2}}|E_{2}|^{\frac{1+\theta_{2}}{2}}
  \sum_{d''\in\mathbb{N}}2^{-(M\theta_{3}-100)d''}.
\end{eqnarray}
for every $\ell\in\mathbb{N}$ and $0\leq\theta_{1},\theta_{2},\theta_{3}<1$ with $\theta_{1}+\theta_{2}+\theta_{3}=1$.

By taking $\theta_{1}$ sufficiently close to $1$ and $\theta_{2}$ sufficiently close to $0$, one can make the exponents $\frac{2}{1+\theta_{1}}=p_{1}$ strictly larger than $1$ and close to $1$ and $\frac{2}{1+\theta_{2}}=p_{2}$ strictly smaller than $2$ and close to $2$. The series over $d''\in\mathbb{N}$ in \eqref{4.35} is summable if we choose $M$ large enough (say, $M\simeq200\theta_{3}^{-1}$). We finally get the estimate
\begin{equation}\label{II}
  \Lambda^{\varepsilon,\ell}_{\vec{\mathbb{P}},II}(f_{1},f_{2},f_{3})\lesssim_{\varepsilon,p,p_{1},p_{2}}2^{11\ell}|E_{1}|^{\frac{1}{p_{1}}}|E_{2}|^{\frac{1}{p_{2}}}
\end{equation}
for every $\ell\in\mathbb{N}$, $\varepsilon>0$ and $p$, $p_{1}$, $p_{2}$ satisfy the hypothesis of Proposition \ref{RW-type equivalent}.

\subsection{Conclusions}
By inserting the estimates \eqref{4.9}, \eqref{I} and \eqref{II} into \eqref{4.3}, we finally get
\begin{equation}\label{4.36}
   |\Lambda^{\varepsilon}_{\vec{\mathbb{P}}}(f_{1},f_{2},f_{3})|\lesssim_{\varepsilon,p,p_{1},p_{2}}\sum_{\ell\in\mathbb{N}}
   2^{-100\ell}2^{12\ell}|E_{1}|^{\frac{1}{p_{1}}}|E_{2}|^{\frac{1}{p_{2}}}\lesssim_{\varepsilon,p,p_{1},p_{2}}|E_{1}|^{\frac{1}{p_{1}}}|E_{2}|^{\frac{1}{p_{2}}}
\end{equation}
for any $\varepsilon>0$, which completes the proof of Proposition \ref{RW-type equivalent} for the model operators $\Pi^{\varepsilon}_{\vec{\mathbb{P}}}$.

This concludes the proof of Theorem \ref{main1}.

\section{Proof of Theorem \ref{main2}}

In this section, we prove Theorem \ref{main2} by carrying out the proof of Proposition \ref{RW-type equivalent} for model operators $\widetilde{\Pi}^{\varepsilon}_{\vec{\mathbb{P}}}$ defined in \eqref{model2} with $\vec{\mathbb{P}}=\mathbb{P}'\times\mathbb{P}''$.

Fix indices $p_{1}$, $p_{2}$, $p$ as in the hypothesis of Proposition \ref{RW-type equivalent}. Fix arbitrary measurable sets $E_{1}$, $E_{2}$, $E_{3}$ of finite measure (by using the scaling invariance of $\widetilde{\Pi}^{\varepsilon}_{\vec{\mathbb{P}}}$, we can assume further that $|E_{3}|=1$). Our goal is to find $E'_{3}\subseteq E_{3}$ with $|E'_{3}|\simeq|E_{3}|=1$ such that, for any functions $|f_{1}|\leq\chi_{E_{1}}$, $|f_{2}|\leq\chi_{E_{2}}$ and $|f_{3}|\leq\chi_{E'_{3}}$, one has the corresponding trilinear forms $\widetilde{\Lambda}^{\varepsilon}_{\vec{\mathbb{P}}}(f_{1},f_{2},f_{3})$ defined by
\begin{equation}\label{5.1}
  \widetilde{\Lambda}^{\varepsilon}_{\vec{\mathbb{P}}}(f_{1},f_{2},f_{3}):=\int_{\mathbb{R}^{2}}\widetilde{\Pi}^{\varepsilon}_{\vec{\mathbb{P}}}(f_{1},f_{2})(x)f_{3}(x)dx
\end{equation}
satisfy estimates
\begin{equation}\label{5.2}
  |\widetilde{\Lambda}^{\varepsilon}_{\vec{\mathbb{P}}}(f_{1},f_{2},f_{3})|=|\sum_{\vec{P}\in\vec{\mathbb{P}}}\frac{\widetilde{C}^{\varepsilon}_{Q_{\vec{P}}}}{|I_{\vec{P}}|^{\frac{1}{2}}}
  \langle f_{1},\Phi^{1}_{\vec{P}_{1}}\rangle\langle f_{2},\Phi^{2}_{\vec{P}_{2}}\rangle\langle f_{3},\Phi^{3}_{\vec{P}_{3}}\rangle|\lesssim_{\varepsilon,p,p_{1},p_{2}}|E_{1}|^{\frac{1}{p_{1}}}|E_{2}|^{\frac{1}{p_{2}}},
\end{equation}
where $p_{1}$ is larger than but close to $1$, while $p_{2}$ is smaller than but close to $2$.

We define the exceptional set
\begin{equation}\label{5.3}
  \Omega:=\{x\in\mathbb{R}^{2}:MM(\frac{\chi_{E_{1}}}{|E_{1}|})(x)>C\}
  \cup\{x\in\mathbb{R}^{2}:MM(\frac{\chi_{E_{2}}}{|E_{2}|})(x)>C\}.
\end{equation}
It is clear that $|\Omega|<\frac{1}{10}$ if $C$ is a large enough constant, which we fix from now on. Then, we define $E'_{3}:=E_{3}\setminus\Omega$ and observe that $|E'_{3}|\simeq1$.

Now we estimate the trilinear form $\widetilde{\Lambda}^{\varepsilon}_{\vec{\mathbb{P}}}(f_{1},f_{2},f_{3})$ defined in \eqref{5.1} by two terms as follows:
\begin{eqnarray}\label{5.4}
 \nonumber |\widetilde{\Lambda}^{\varepsilon}_{\vec{\mathbb{P}}}(f_{1},f_{2},f_{3})|&\lesssim& \sum_{\substack{\vec{P}\in\vec{\mathbb{P}}:\\I_{\vec{P}}\cap\Omega^{c}\neq\emptyset}}
  \frac{|\widetilde{C}^{\varepsilon}_{Q_{\vec{P}}}|}{|I_{\vec{P}}|^{\frac{1}{2}}}
  |\langle f_{1},\Phi^{1}_{\vec{P}_{1}}\rangle||\langle f_{2},\Phi^{2}_{\vec{P}_{2}}\rangle||\langle f_{3},\Phi^{3}_{\vec{P}_{3}}\rangle| \\
  &&+\sum_{\substack{\vec{P}\in\vec{\mathbb{P}}:\\I_{\vec{P}}\cap\Omega^{c}=\emptyset}}
  \frac{|\widetilde{C}^{\varepsilon}_{Q_{\vec{P}}}|}{|I_{\vec{P}}|^{\frac{1}{2}}}
  |\langle f_{1},\Phi^{1}_{\vec{P}_{1}}\rangle||\langle f_{2},\Phi^{2}_{\vec{P}_{2}}\rangle||\langle f_{3},\Phi^{3}_{\vec{P}_{3}}\rangle|\\
 \nonumber &=:& \widetilde{\Lambda}^{\varepsilon}_{\vec{\mathbb{P}},I}(f_{1},f_{2},f_{3})+\widetilde{\Lambda}^{\varepsilon}_{\vec{\mathbb{P}},II}(f_{1},f_{2},f_{3}).
\end{eqnarray}

\subsection{Estimates for trilinear form $\widetilde{\Lambda}^{\varepsilon}_{\vec{\mathbb{P}},I}(f_{1},f_{2},f_{3})$}
We can decompose the collection $\widetilde{\mathbb{P}'}$ of tri-tiles into
\begin{equation}\label{5.5}
  \mathbb{P}'=\bigcup_{k'\in\mathbb{Z}}\mathbb{P}'_{k'},
\end{equation}
where
\begin{equation}\label{5.6}
  \mathbb{P}'_{k'}:=\{P'\in\mathbb{P}': \, \ell(Q_{P'})=2^{k'}\}.
\end{equation}
As a consequence, we can split the trilinear form $\widetilde{\Lambda}^{\varepsilon}_{\vec{\mathbb{P}},I}(f_{1},f_{2},f_{3})$ into
\begin{eqnarray}\label{5.7}
  \widetilde{\Lambda}^{\varepsilon}_{\vec{\mathbb{P}},I}(f_{1},f_{2},f_{3})&=&\sum_{k'\in\mathbb{Z}}
  \sum_{\substack{\vec{P}\in\mathbb{P}'_{k'}\times\mathbb{P}'': \\ I_{\vec{P}}\cap\Omega^{c}\neq\emptyset}}
  |\widetilde{C}^{\varepsilon}_{Q_{\vec{P}}}|\frac{|I_{P'}|}{|I_{P''}|^{\frac{1}{2}}}
  |\langle \frac{\langle f_{1},\Phi^{1}_{P'_{1}}\rangle}{|I_{P'}|^{\frac{1}{2}}},\Phi^{1}_{P''_{1}}\rangle|\\
  \nonumber && \quad\quad\quad\quad\quad\quad\quad\quad\quad \times|\langle \frac{\langle f_{2},\Phi^{2}_{P'_{2}}\rangle}{|I_{P'}|^{\frac{1}{2}}},\Phi^{2}_{P''_{2}}\rangle|
  |\langle \frac{\langle f_{3},\Phi^{3}_{P'_{3}}\rangle}{|I_{P'}|^{\frac{1}{2}}},\Phi^{3}_{P''_{3}}\rangle|.
\end{eqnarray}

Note that by Lemma \ref{Fourier coefficient2}, we can estimate the Fourier coefficients $\widetilde{C}^{\varepsilon}_{Q_{\vec{P}}}:=\widetilde{C}^{\varepsilon}_{Q_{\vec{P}},\vec{0},\vec{0},\vec{0}}$ for each $\vec{P}\in\mathbb{P}'_{k'}\times\mathbb{P}''$ ($k'\in\mathbb{Z}$) by
\begin{equation}\label{5.8}
  |\widetilde{C}^{\varepsilon}_{Q_{\vec{P}}}|\lesssim \widetilde{C}^{\varepsilon}_{k'}:=\langle k'\rangle^{-(1+\varepsilon)}=(1+|k'|^{2})^{-\frac{1+\varepsilon}{2}}.
\end{equation}
For each fixed $P'\in\mathbb{P}'$, we define the sub-collection $\mathbb{P}''_{P'}$ of $\mathbb{P}''$ by
\begin{equation*}
  \mathbb{P}''_{P'}:=\{P''\in\mathbb{P}'': \, I_{\vec{P}}\cap\Omega^{c}\neq\emptyset\}.
\end{equation*}
Therefore, by using Proposition \ref{tile norm estimates}, we derive the following estimates
\begin{eqnarray}\label{5.9}
\nonumber \widetilde{\Lambda}^{\varepsilon}_{\vec{\mathbb{P}},I}(f_{1},f_{2},f_{3})&\lesssim&\sum_{k'\in\mathbb{Z}}\widetilde{C}^{\varepsilon}_{k'}
\sum_{P'\in\mathbb{P}'_{k'}}|I_{P'}|[\prod_{j=1}^{2}(energy_{j}((\langle \frac{\langle f_{j},\Phi^{j}_{P'_{j}}\rangle}{|I_{P'}|^{\frac{1}{2}}},\Phi^{j}_{P''_{j}}\rangle)_{P''\in\mathbb{P}''_{P'}}))^{1-\theta_{j}}\\
  &\times&(size_{j}((\langle \frac{\langle f_{j},\Phi^{j}_{P'_{j}}\rangle}{|I_{P'}|^{\frac{1}{2}}},\Phi^{j}_{P''_{j}}\rangle)_{P''\in\mathbb{P}''_{P'}}))^{\theta_{j}}]
 (size_{3}((\langle \frac{\langle f_{3},\Phi^{3}_{P'_{3}}\rangle}{|I_{P'}|^{\frac{1}{2}}},\Phi^{3}_{P''_{3}}\rangle)_{P''\in\mathbb{P}''_{P'}}))^{\theta_{3}}\\
 \nonumber && \quad\quad\quad\quad\quad\quad\quad\quad\quad\quad\quad\quad \times(energy_{3}((\langle \frac{\langle f_{3},\Phi^{3}_{P'_{3}}\rangle}{|I_{P'}|^{\frac{1}{2}}},\Phi^{3}_{P''_{3}}\rangle)_{P''\in\mathbb{P}''_{P'}}))^{1-\theta_{3}}
\end{eqnarray}
for any $0\leq\theta_{1},\theta_{2},\theta_{3}<1$ with $\theta_{1}+\theta_{2}+\theta_{3}=1$.

To estimate the right-hand side of \eqref{5.9}, note that $I_{\vec{P}}\cap\Omega^{c}\neq\emptyset$ and $supp \, f_{3}\subseteq E'_{3}$, we apply the \emph{size} estimates in Lemma \ref{size estimates} and get for each $P'\in\mathbb{P}'_{k'}$ and $j=1,2,3$,
\begin{equation}\label{5.10}
  size_{j}((\langle \frac{\langle f_{j},\Phi^{j}_{P'_{j}}\rangle}{|I_{P'}|^{\frac{1}{2}}},\Phi^{j}_{P''_{j}}\rangle)_{P''\in\mathbb{P}''_{P'}})
  \lesssim\sup_{P''\in\mathbb{P}''_{P'}}\frac{1}{|I_{P''}|}\int_{\mathbb{R}}|\frac{\langle f_{j},\Phi^{j}_{P'_{j}}\rangle}{|I_{P'}|^{\frac{1}{2}}}|\widetilde{\chi}^{M}_{I_{P''}}dx
  \lesssim|E_{j}|,
\end{equation}
where $M>0$ is sufficiently large. By applying the \emph{energy} estimates in Lemma \ref{energy estimates}, we have for each $P'\in\mathbb{P}'_{k'}$ and $j=1,2,3$,
\begin{equation}\label{5.11}
  energy_{j}((\langle \frac{\langle f_{j},\Phi^{j}_{P'_{j}}\rangle}{|I_{P'}|^{\frac{1}{2}}},\Phi^{j}_{P''_{j}}\rangle)_{P''\in\mathbb{P}''_{P'}})\lesssim
  \frac{1}{|I_{P'}|^{\frac{1}{2}}}(\int_{\mathbb{R}}|\langle f_{j},\Phi^{j}_{P'_{j}}\rangle|^{2}dx_{2})^{\frac{1}{2}}.
\end{equation}

Now we insert the size and energy estimates \eqref{5.10}, \eqref{5.11} into \eqref{5.9} and get
\begin{equation}\label{5.12}
\widetilde{\Lambda}^{\varepsilon}_{\vec{\mathbb{P}},I}(f_{1},f_{2},f_{3})\lesssim|E_{1}|^{\theta_{1}}|E_{2}|^{\theta_{2}}
\sum_{k'\in\mathbb{Z}}\widetilde{C}^{\varepsilon}_{k'}\sum_{P'\in\mathbb{P}'_{k'}}
\{\prod_{j=1}^{3}(\int_{\mathbb{R}}|\langle f_{j},\Phi^{j}_{P'_{j}}\rangle|^{2}dx_{2})^{\frac{1-\theta_{j}}{2}}\}.
\end{equation}
Observe that for any different tri-tiles $\bar{P'}\in\mathbb{P}'_{k'}$ and $\bar{\bar{P'}}\in\mathbb{P}'_{k'}$, one has $I_{\bar{P'}}\cap I_{\bar{\bar{P'}}}=\emptyset$, or otherwise, one has $I_{\bar{P'}}=I_{\bar{\bar{P'}}}$ but $\omega_{\bar{P'_{j}}}\cap\omega_{\bar{\bar{P'_{j}}}}=\emptyset$ for every $j=1,2,3$. By taking advantage of such orthogonality in $L^{2}$ of the wave packets $\Phi^{j}_{P'_{j}}$ corresponding to the tiles $P'_{j}$ ($j=1,2,3$), one has that for any function $F\in L^{2}(\mathbb{R})$ and $k'\in\mathbb{Z}$,
\begin{eqnarray}\label{Bessel-1}
  \nonumber \|\sum_{P'\in\mathbb{P}'_{k'}}\langle F,\Phi^{j}_{P'_{j}}\rangle\Phi^{j}_{P'_{j}}\|^{2}_{L^{2}}
  &\leq&\sum_{\substack{\bar{P'}, \, \bar{\bar{P'}}\in\mathbb{P}'_{k'}: \\ \omega_{\bar{P}'_{j}}=\omega_{\bar{\bar{P}}'_{j}}; \, I_{\bar{P'}}\cap I_{\bar{\bar{P'}}}=\emptyset}}|\langle F,\Phi^{j}_{\bar{P}'_{j}}\rangle||\langle F,\Phi^{j}_{\bar{\bar{P}}'_{j}}\rangle||\langle\Phi^{j}_{\bar{P}'_{j}},\Phi^{j}_{\bar{\bar{P}}'_{j}}\rangle|\\
  &\lesssim&2^{k'}\sum_{\bar{P'}\in\mathbb{P}'_{k'}}|\langle F,\Phi^{j}_{\bar{P}'_{j}}\rangle|^{2}
  \sum_{\substack{\bar{\bar{P'}}\in\mathbb{P}'_{k'}: \\ \omega_{\bar{P}'_{j}}=\omega_{\bar{\bar{P}}'_{j}}; \, I_{\bar{P'}}\cap I_{\bar{\bar{P'}}}=\emptyset}}|\langle\widetilde{\chi}^{1000}_{I_{\bar{P'}}},\widetilde{\chi}^{1000}_{I_{\bar{\bar{P'}}}}\rangle|\\
 \nonumber &\lesssim&\sum_{\bar{P'}\in\mathbb{P}'_{k'}}|\langle F,\Phi^{j}_{\bar{P}'_{j}}\rangle|^{2}\sum_{\substack{\bar{\bar{P'}}\in\mathbb{P}'_{k'}: \\ \omega_{\bar{P}'_{j}}=\omega_{\bar{\bar{P}}'_{j}}; \, I_{\bar{P'}}\cap I_{\bar{\bar{P'}}}=\emptyset}}(1+\frac{dist(I_{\bar{P'}},I_{\bar{\bar{P'}}})}{|I_{\bar{P'}}|})^{-100}\\
  \nonumber &\lesssim&\sum_{P'\in\mathbb{P}'_{k'}}|\langle F,\Phi^{j}_{P'_{j}}\rangle|^{2},
\end{eqnarray}
from which we deduce the following Bessel-type inequality
\begin{eqnarray}\label{Bessel-2}
  \sum_{P'\in\mathbb{P}'_{k'}}|\langle F,\Phi^{j}_{P'_{j}}\rangle|^{2}&=&|\langle\sum_{P'\in\mathbb{P}'_{k'}}\langle F,\Phi^{j}_{P'_{j}}\rangle\Phi^{j}_{P'_{j}},F\rangle|\\
  \nonumber &\leq&\|\sum_{P'\in\mathbb{P}'_{k'}}\langle F,\Phi^{j}_{P'_{j}}\rangle\Phi^{j}_{P'_{j}}\|_{L^{2}}\cdot\|F\|_{L^{2}}\lesssim\|F\|^{2}_{L^{2}},
\end{eqnarray}
where the implicit constants in the bounds are independent of $k'\in\mathbb{Z}$. Then, we can use Bessel-type inequality \eqref{Bessel-2} and H\"{o}lder inequality to estimate the inner sum in the right-hand side of \eqref{5.12} by
\begin{eqnarray}\label{5.13}
  && \sum_{P'\in\mathbb{P}'_{k'}}
\{\prod_{j=1}^{3}(\int_{\mathbb{R}}|\langle f_{j},\Phi^{j}_{P'_{j}}\rangle|^{2}dx_{2})^{\frac{1-\theta_{j}}{2}}\}\lesssim
\prod_{j=1}^{3}(\int_{\mathbb{R}}\sum_{P'\in\mathbb{P}'_{k'}}|\langle f_{j},\Phi^{j}_{P'_{j}}\rangle|^{2}dx_{2})^{\frac{1-\theta_{j}}{2}}\\
  \nonumber &\lesssim&\prod_{j=1}^{3}\|f_{j}\|^{1-\theta_{j}}_{L^{2}(\mathbb{R}^{2})}\lesssim|E_{1}|^{\frac{1-\theta_{1}}{2}}|E_{2}|^{\frac{1-\theta_{2}}{2}}.
\end{eqnarray}
Combining the estimates \eqref{5.8}, \eqref{5.12} and \eqref{5.13}, we arrive at
\begin{equation}\label{5.14}
\widetilde{\Lambda}^{\varepsilon}_{\vec{\mathbb{P}},I}(f_{1},f_{2},f_{3})\lesssim
|E_{1}|^{\theta_{1}}|E_{2}|^{\theta_{2}}|E_{1}|^{\frac{1-\theta_{1}}{2}}|E_{2}|^{\frac{1-\theta_{2}}{2}}\sum_{k'\in\mathbb{Z}}\widetilde{C}^{\varepsilon}_{k'}
\lesssim_{\varepsilon,\theta_{1},\theta_{2},\theta_{3}}|E_{1}|^{\frac{1+\theta_{1}}{2}}|E_{2}|^{\frac{1+\theta_{2}}{2}}
\end{equation}
for any $0\leq\theta_{1},\theta_{2},\theta_{3}<1$ with $\theta_{1}+\theta_{2}+\theta_{3}=1$.

By taking $\theta_{1}$ sufficiently close to $1$ and $\theta_{2}$ sufficiently close to $0$, one can make the exponents $\frac{2}{1+\theta_{1}}=p_{1}$ strictly larger than $1$ and close to $1$ and $\frac{2}{1+\theta_{2}}=p_{2}$ strictly smaller than $2$ and close to $2$. We finally get the estimate
\begin{equation}\label{5.15}
  \widetilde{\Lambda}^{\varepsilon}_{\vec{\mathbb{P}},I}(f_{1},f_{2},f_{3})\lesssim_{\varepsilon,p,p_{1},p_{2}}|E_{1}|^{\frac{1}{p_{1}}}|E_{2}|^{\frac{1}{p_{2}}}
\end{equation}
for every $\varepsilon>0$ and $p$, $p_{1}$, $p_{2}$ satisfy the hypothesis of Proposition \ref{RW-type equivalent}.

\subsection{Estimates for trilinear form $\widetilde{\Lambda}^{\varepsilon}_{\vec{\mathbb{P}},II}(f_{1},f_{2},f_{3})$}
For each fixed $P'\in\mathbb{P}'$, we define the corresponding sub-collection of $\mathbb{P}''$ by
\begin{equation*}
  \mathbb{P}''_{P'}:=\{P''\in\mathbb{P}'': \, I_{\vec{P}}\subseteq\Omega\},
\end{equation*}
then we can decompose the collection $\mathbb{P}''_{P'}$ further, as follows:
\begin{equation}\label{5.16}
  \mathbb{P}''_{P'}=\bigcup_{\mu\in\mathbb{N}}\mathbb{P}''_{P',\mu},
\end{equation}
where
\begin{equation}\label{5.17}
  \mathbb{P}''_{P',\mu}:=\{P''\in\mathbb{P}''_{P'}: \, Dil_{2^{\mu}}(I_{P'}\times I_{P''})\subseteq\Omega\}
\end{equation}
and $\mu$ is maximal with this property. By $Dil_{2^{\mu}}(I_{\vec{P}})$ we denote the rectangle having the same center as the original $I_{\vec{P}}$ but whose side-lengths are $2^{\mu}$ times larger.

Now we apply both the decompositions of $\widetilde{\mathbb{P}'}$ and $\mathbb{P}''_{P'}$ defined in \eqref{5.5}, \eqref{5.16} at the same time, and split the trilinear form $\widetilde{\Lambda}^{\varepsilon}_{\vec{\mathbb{P}},II}(f_{1},f_{2},f_{3})$ into
\begin{eqnarray}\label{5.18}
 \nonumber \widetilde{\Lambda}^{\varepsilon}_{\vec{\mathbb{P}},II}(f_{1},f_{2},f_{3})&=&\sum_{k'\in\mathbb{Z}}
  \sum_{P'\in\mathbb{P}'_{k'}}
  |\widetilde{C}^{\varepsilon}_{Q_{\vec{P}}}||I_{P'}|\sum_{\mu\in\mathbb{N}}\sum_{P''\in\mathbb{P}''_{P',\mu}}\frac{1}{|I_{P''}|^{\frac{1}{2}}}\\
  && \quad\quad\quad\quad\quad \times|\langle \frac{\langle f_{1},\Phi^{1}_{P'_{1}}\rangle}{|I_{P'}|^{\frac{1}{2}}},\Phi^{1}_{P''_{1}}\rangle||\langle \frac{\langle f_{2},\Phi^{2}_{P'_{2}}\rangle}{|I_{P'}|^{\frac{1}{2}}},\Phi^{2}_{P''_{2}}\rangle|
  |\langle \frac{\langle f_{3},\Phi^{3}_{P'_{3}}\rangle}{|I_{P'}|^{\frac{1}{2}}},\Phi^{3}_{P''_{3}}\rangle|.
\end{eqnarray}

In the inner sum of the above \eqref{5.18}, since $Dil_{2^{\mu}}(I_{P'}\times I_{P''})\subseteq\Omega$, and $supp \, f_{3}\subseteq E'_{3}\subseteq\mathbb{R}^{2}\setminus\Omega$, we get that
\begin{equation}\label{5.19}
  |f_{3}|\leq\chi_{E'_{3}}\chi_{(Dil_{2^{\mu}}(I_{P'}\times I_{P''}))^{c}}
  =\chi_{E'_{3}}\{\chi_{(2^{\mu}I_{P'})^{c}}+\chi_{(2^{\mu}I_{P''})^{c}}-\chi_{(2^{\mu}I_{P'})^{c}}\chi_{(2^{\mu}I_{P''})^{c}}\},
\end{equation}
and hence we can assume hereafter in this subsection that
\begin{equation}\label{5.20}
   |f_{3}|\leq\chi_{E'_{3}}\chi_{(2^{\mu}I_{P'})^{c}},
\end{equation}
and the other two terms can be handled similarly.

By using Proposition \ref{tile norm estimates} and \eqref{5.8}, we derive from \eqref{5.18} the following estimates
\begin{eqnarray}\label{5.21}
 &&\widetilde{\Lambda}^{\varepsilon}_{\vec{\mathbb{P}},II}(f_{1},f_{2},f_{3})\\
\nonumber &\lesssim&\sum_{k'\in\mathbb{Z}}\widetilde{C}^{\varepsilon}_{k'}
\sum_{P'\in\mathbb{P}'_{k'}}|I_{P'}|\sum_{\mu\in\mathbb{N}}
[\prod_{j=1}^{2}(energy_{j}((\langle \frac{\langle f_{j},\Phi^{j}_{P'_{j}}\rangle}{|I_{P'}|^{\frac{1}{2}}},\Phi^{j}_{P''_{j}}\rangle)_{P''\in\mathbb{P}''_{P',\mu}}))^{1-\theta_{j}}\\
 \nonumber &&\times(size_{j}((\langle \frac{\langle f_{j},\Phi^{j}_{P'_{j}}\rangle}{|I_{P'}|^{\frac{1}{2}}},\Phi^{j}_{P''_{j}}\rangle)_{P''\in\mathbb{P}''_{P',\mu}}))^{\theta_{j}}]
 (size_{3}((\langle \frac{\langle f_{3},\Phi^{3}_{P'_{3}}\rangle}{|I_{P'}|^{\frac{1}{2}}},\Phi^{3}_{P''_{3}}\rangle)_{P''\in\mathbb{P}''_{P',\mu}}))^{\theta_{3}}\\
 \nonumber && \times(energy_{3}((\langle \frac{\langle f_{3},\Phi^{3}_{P'_{3}}\rangle}{|I_{P'}|^{\frac{1}{2}}},\Phi^{3}_{P''_{3}}\rangle)_{P''\in\mathbb{P}''_{P',\mu}}))^{1-\theta_{3}}
\end{eqnarray}
for any $0\leq\theta_{1},\theta_{2},\theta_{3}<1$ with $\theta_{1}+\theta_{2}+\theta_{3}=1$.

To estimate the inner sum in the right-hand side of \eqref{5.21}, note that $I_{\vec{P}}\subseteq\Omega$, $P''\in\mathbb{P}''_{P',\mu}$ and $f_{3}$ satisfies \eqref{5.20}, we apply the \emph{size} estimates in Lemma \ref{size estimates} and get for each $P'\in\mathbb{P}'_{k'}$ and $\mu\in\mathbb{N}$,
\begin{equation}\label{5.22}
  size_{1}((\langle \frac{\langle f_{1},\Phi^{1}_{P'_{1}}\rangle}{|I_{P'}|^{\frac{1}{2}}},\Phi^{1}_{P''_{1}}\rangle)_{P''\in\mathbb{P}''_{P',\mu}})
  \lesssim\sup_{P''\in\mathbb{P}''_{P',\mu}}\frac{1}{|I_{P''}|}\int_{\mathbb{R}}|\frac{\langle f_{1},\Phi^{1}_{P'_{1}}\rangle}{|I_{P'}|^{\frac{1}{2}}}|\widetilde{\chi}^{M}_{I_{P''}}dx
  \lesssim2^{2\mu}|E_{1}|,
\end{equation}
\begin{equation}\label{5.23}
  size_{2}((\langle \frac{\langle f_{2},\Phi^{2}_{P'_{2}}\rangle}{|I_{P'}|^{\frac{1}{2}}},\Phi^{2}_{P''_{2}}\rangle)_{P''\in\mathbb{P}''_{P',\mu}})
  \lesssim\sup_{P''\in\mathbb{P}''_{P',\mu}}\frac{1}{|I_{P''}|}\int_{\mathbb{R}}|\frac{\langle f_{2},\Phi^{2}_{P'_{2}}\rangle}{|I_{P'}|^{\frac{1}{2}}}|\widetilde{\chi}^{M}_{I_{P''}}dx
  \lesssim2^{2\mu}|E_{2}|,
\end{equation}
\begin{equation}\label{5.24}
  size_{3}((\langle \frac{\langle f_{3},\Phi^{3}_{P'_{3}}\rangle}{|I_{P'}|^{\frac{1}{2}}},\Phi^{3}_{P''_{3}}\rangle)_{P''\in\mathbb{P}''_{P',\mu}})
  \lesssim\sup_{P''\in\mathbb{P}''_{P',\mu}}\frac{1}{|I_{P''}|}\int_{\mathbb{R}}|\frac{\langle f_{3},\Phi^{3}_{P'_{3}}\rangle}{|I_{P'}|^{\frac{1}{2}}}|\widetilde{\chi}^{M}_{I_{P''}}dx
  \lesssim 2^{-N\mu},
\end{equation}
where $M>0$ and $N>0$ are arbitrarily large. By applying the \emph{energy} estimates in Lemma \ref{energy estimates}, we have for each $P'\in\mathbb{P}'_{k'}$, $\mu\in\mathbb{N}$ and $j=1,2,3$,
\begin{equation}\label{5.25}
  energy_{j}((\langle \frac{\langle f_{j},\Phi^{j}_{P'_{j}}\rangle}{|I_{P'}|^{\frac{1}{2}}},\Phi^{j}_{P''_{j}}\rangle)_{P''\in\mathbb{P}''_{P',\mu}})\lesssim
  \frac{1}{|I_{P'}|^{\frac{1}{2}}}(\int_{\mathbb{R}}|\langle f_{j},\Phi^{j}_{P'_{j}}\rangle|^{2}dx_{2})^{\frac{1}{2}}.
\end{equation}

Now we insert the size and energy estimates \eqref{5.22}-\eqref{5.25} into \eqref{5.21}, by using the estimates \eqref{5.8} and \eqref{5.13}, we derive that
\begin{eqnarray}\label{5.26}
  \nonumber \widetilde{\Lambda}^{\varepsilon}_{\vec{\mathbb{P}},II}(f_{1},f_{2},f_{3})&\lesssim&|E_{1}|^{\theta_{1}}|E_{2}|^{\theta_{2}}
\sum_{k'\in\mathbb{Z}}\widetilde{C}^{\varepsilon}_{k'}\sum_{\mu\in\mathbb{N}}2^{-(N\theta_{3}-2)\mu}\sum_{P'\in\mathbb{P}'_{k'}}
\{\prod_{j=1}^{3}(\int_{\mathbb{R}}|\langle f_{j},\Phi^{j}_{P'_{j}}\rangle|^{2}dx_{2})^{\frac{1-\theta_{j}}{2}}\}\\
  &&\lesssim_{\varepsilon,\theta_{1},\theta_{2},\theta_{3},N}|E_{1}|^{\frac{1+\theta_{1}}{2}}|E_{2}|^{\frac{1+\theta_{2}}{2}}
  \sum_{\mu\in\mathbb{N}}2^{-(N\theta_{3}-2)\mu}.
\end{eqnarray}
for every $0\leq\theta_{1},\theta_{2},\theta_{3}<1$ with $\theta_{1}+\theta_{2}+\theta_{3}=1$.

By taking $\theta_{1}$ sufficiently close to $1$ and $\theta_{2}$ sufficiently close to $0$, one can make the exponents $\frac{2}{1+\theta_{1}}=p_{1}$ strictly larger than $1$ and close to $1$ and $\frac{2}{1+\theta_{2}}=p_{2}$ strictly smaller than $2$ and close to $2$. The series over $\mu\in\mathbb{N}$ in \eqref{5.26} is summable if we choose $N$ large enough (say, $N\simeq4\theta_{3}^{-1}$). We finally get the estimate
\begin{equation}\label{5.27}
  \widetilde{\Lambda}^{\varepsilon}_{\vec{\mathbb{P}},II}(f_{1},f_{2},f_{3})\lesssim_{\varepsilon,p,p_{1},p_{2}}|E_{1}|^{\frac{1}{p_{1}}}|E_{2}|^{\frac{1}{p_{2}}}
\end{equation}
for any $\varepsilon>0$ and $p$, $p_{1}$, $p_{2}$ satisfy the hypothesis of Proposition \ref{RW-type equivalent}.

\subsection{Conclusions}
By inserting the estimates \eqref{5.15} and \eqref{5.27} into \eqref{5.4}, we finally get
\begin{equation}\label{5.28}
   |\widetilde{\Lambda}^{\varepsilon}_{\vec{\mathbb{P}}}(f_{1},f_{2},f_{3})|\lesssim_{\varepsilon,p,p_{1},p_{2}}|E_{1}|^{\frac{1}{p_{1}}}|E_{2}|^{\frac{1}{p_{2}}}
\end{equation}
for any $\varepsilon>0$, which completes the proof of Proposition \ref{RW-type equivalent} for the model operators $\widetilde{\Pi}^{\varepsilon}_{\vec{\mathbb{P}}}$.

This concludes the proof of Theorem \ref{main2}.\\

\end{document}